\documentclass[12pt]{amsproc} %

\usepackage{amsmath, amssymb, amsfonts, amsthm}
\usepackage{setspace}
\usepackage{latexsym}
\usepackage{float}
\usepackage{xcolor}
\usepackage{mathrsfs}
\usepackage{hyperref}
\usepackage{bm}
\usepackage[margin=1in]{geometry}
\usepackage{enumerate}
\usepackage{url}
\usepackage{hyperref}
\usepackage{colonequals}
\usepackage{algorithm}
\usepackage{algorithmicx}
\usepackage{algpseudocode}
\usepackage{graphicx}
\usepackage{longtable}
\usepackage{array}
\usepackage{listings}
\usepackage{chngcntr}
\usepackage{placeins}

\newtheorem{thm}{Theorem}[section]
\newtheorem{cor}[thm]{Corollary}

\newtheorem{lem}[thm]{Lemma}

\newtheorem{prop}[thm]{Proposition}
\theoremstyle{definition}\newtheorem{defn}[thm]{Definition}
\theoremstyle{definition}\newtheorem{rem}[thm]{Remark}
\theoremstyle{definition}\newtheorem{ex}[thm]{Example}

\newcounter{mycount}
\counterwithin{algorithm}{mycount}
\refstepcounter{mycount}

\newcommand{\GF}{{\rm GF}}
\newcommand{\Irr}{{\rm Irr}}

\newcommand{\GAP}{\textbf{{\rm GAP}}}

\newcommand{\GQ}{{\rm GQ}}

\newcommand{\Cay}{{\rm Cay}}

\newcommand\rk{{\rm rk}\;}

\newcommand{\calL}{\mathcal{L}}

\newcommand{\calO}{\mathcal{O}}

\renewcommand\le{\leqslant}
\renewcommand\ge{\geqslant}

\newcommand{\Q}{\mathbb{Q}}

\newcommand{\C}{\mathbb{C}}

\newcommand{\Z}{\mathbb{Z}}
\renewcommand{\S}{{\mathrm S}}

\newcommand{\frakR}{\mathfrak{R}}
\newcommand{\repF}{\mathfrak{F}}
\newcommand{\p}{\mathfrak{p}}
\newcommand{\z}{\zeta}

\newcolumntype{P}[1]{>{\centering\arraybackslash}p{#1}}

\allowdisplaybreaks

\begin{document}

\title{Character theoretic techniques for nonabelian partial difference sets}

\author{Seth R. Nelson}
\address{Department of Mathematics, University of Georgia, Athens, GA 30602}
\email{sn57661@uga.edu}

\author{Eric Swartz}
\address{Department of Mathematics, William \& Mary, Williamsburg, VA 23187}
\email{easwartz@wm.edu}

\begin{abstract}
A $(v,k,\lambda, \mu)$-partial difference set (PDS) is a subset $D$ of size $k$ of a group $G$ of order $v$ such that every nonidentity element $g$ of $G$ can be expressed in either $\lambda$ or $\mu$ different ways as a product $xy^{-1}$, $x, y \in D$, depending on whether or not $g$ is in $D$. If $D$ is inverse closed and $1 \notin D$, then the Cayley graph $\Cay(G,D)$ is a $(v,k,\lambda, \mu)$-strongly regular graph (SRG). PDSs have been studied extensively over the years, especially in abelian groups, where techniques from character theory have proven to be particularly effective. Recently, there has been considerable interest in studying PDSs in nonabelian groups, and the purpose of this paper is to develop character theoretic techniques that apply in the nonabelian setting. We prove that analogues of character theoretic results of Ott \cite{ott} about generalized quadrangles of order $s$ also hold in the general PDS setting, and we are able to use these techniques to compute the intersection of a putative PDS with the conjugacy classes of the parent group in many instances. With these techniques, we are able to prove the nonexistence of PDSs in numerous instances and provide severe restrictions in cases when such PDSs may still exist. Furthermore, we are able to use these techniques constructively, computing several examples of PDSs in nonabelian groups not previously recognized in the literature, including an infinite family of genuinely nonabelian PDSs associated to the block-regular Steiner triple systems originally studied by Clapham \cite{clapham} and related infinite families of genuinely nonabelian PDSs associated to the block-regular Steiner $2$-designs first studied by Wilson \cite{wilson}.
\end{abstract}

\maketitle

\section{Introduction}

If $G$ is a group, then we say that $D \subset G$ is a $(v, k, \lambda, \mu)$-\textit{partial difference set} (PDS) if $G$ has order $v$, $D$ has size $k$, and if the product $gh^{-1}$ for all $g, h \in D$ with $g \neq h$ represents all nonidentity elements in $D$ exactly $\lambda$ times and all nonidentity elements in $G - D$ exactly $\mu$ times. We refer to the tuple $(v, k, \lambda, \mu)$ as the {\it parameters} of $D$. In addition, we call a PDS {\it regular} if $D$ is closed under taking inverses (in which case we write $D = D^{(-1)}$) and $1 \not \in D$. This definition is not particularly restrictive, for if $D$ is a PDS with $1 \in D$, then $D - \{1\}$ is a PDS as well, and if $\lambda \neq \mu$, then $D = D^{(-1)}$. If we have $\lambda = \mu$ then $D$ is called a difference set, and we say that $D$ is a $(v, k, \lambda)$-\textit{difference set} (DS). We additionally call $D$ a \textit{reversible} $(v, k, \lambda)$-DS if $D = D^{(-1)}$.

PDSs have been the subject of great interest over the years (see, e.g., \cite{ma}), especially in abelian groups.  In the abelian group setting, character theoretic methods have proven to be particularly effective.  For example, if we extend an irreducible character $\chi$ for an abelian group $G$ in the natural way to the group algebra $\C[G]$ and identify an inverse-closed subset $D \subset G$ with $\sum_{d \in D} d \in \C[G]$, then $D$ is a $(v, k, \lambda, \mu)$-PDS if and only if
\begin{equation}
\label{eq:abelchar}
\chi(D) = \begin{cases}
              k, \text{ if $\chi$ is the trivial character,}\\
              \frac{(\lambda-\mu) \pm \sqrt{\Delta}}{2}, \text{ otherwise,}
             \end{cases}
\end{equation}
where $\Delta \colonequals (\lambda - \mu)^2 + 4(k - \mu)$ (see \cite[Corollary 3.3]{ma}).

There has been interest recently in PDSs in nonabelian groups \cite{Davis_etal_2023, Feng_He_Chen_2020, Swartz_etal_2024, swartz-tauscheck}; for a recent survey of known results related to PDSs in nonabelian groups, see \cite{pol-dav-smi-swa}. However, most character theoretic results that apply in the abelian case do not hold in the nonabelian case: for example, while linear characters evaluate to the values listed in Equation \eqref{eq:abelchar} above, general irreducible characters do not necessarily need to evaluate to one of these values on a PDS (see Lemma \ref{eigenvalues}). Perhaps more importantly, the characters evaluating to one of a few specific values on an inverse-closed subset is not a \textit{sufficient} condition for the subset to be a PDS in the nonabelian case (only a necessary condition); see Remark \ref{rem:111}. 

Up until now, the only restrictions that apply in the nonabelian setting are those of \cite{swartz-tauscheck}, which were inspired by the work of Yoshiara on groups acting regularly on the point set of a generalized quadrangle (GQ) \cite{yoshiara}. The main method in each of \cite{swartz-tauscheck, yoshiara} was to constrain the value $\Phi(g)$, which counts the number of vertices (resp., points) moved to adjacent vertices (resp., points) by an automorphism $g$ in the automorphism group $G$ of the SRG (resp., GQ). Yoshiara was able to utilize Benson's Lemma \cite[Lemma 4.3]{benson}, which gave modular constraints on $\Phi(g)$ for GQs. De Winter, Kamischke, and Wang \cite{dewinter-kamischke-wang} later proved a generalization of Benson's Lemma for arbitrary SRGs, and Swartz and Tauscheck \cite{swartz-tauscheck} used the generalization of De Winter, Kamischke, and Wang to push Yoshiara's result into an analogous modular restriction on SRGs with a regular automorphism group.

Recently, Ott \cite{ott} studied $\Phi$ -- which is indeed a class function on $G$ -- from a character-theoretic perspective, and Ott used this to prove that GQs of order $s$ ($s+1$ points incident with each line, $s+1$ lines incident with each point), $s$ even, cannot have a group acting regularly on its point set. As it turns out, the ideas Ott used are applicable in situations far beyond GQs (and reversible difference sets); in fact, one could argue that the PDS setting is where these techniques are most effective.

The purpose of this paper is to extend Ott's results to PDSs in nonabelian groups. (Indeed, just as techniques such as Equation \eqref{eq:abelchar} are only truly effective in the abelian setting, the results outlined here are most effective in certain nonabelian settings.) These techniques have enabled us to rule out the existence of a PDS inside a range of groups and to provide severe restrictions on certain families of PDSs. We also develop new techniques for discovering and constructing new examples of PDSs inside of nonabelian groups. For a given group $G$ of order $v$, if there exists a parameter set $(v, k, \lambda, \mu)$ which could potentially correspond to a PDS inside $G$, then these methods use character theory to give information about the size of the intersection of this potential PDS with every conjugacy class of $G$.

We have used these methods to show that if a $(v, k, \lambda, \mu)$-PDS exists in a group $G$, where there is a minimal nontrivial normal subgroup $N$ of $G$ such that $G/N$ is abelian and $\gcd(|G/N|, \sqrt{\Delta}) = 1$, then the size of the intersection of the PDS with the cosets of $N$ are known, providing severe restrictions on such PDSs. We have also used our methods to construct several PDSs. For $p$ a prime, we demonstrate the existence of an infinite family of PDSs in nonabelian groups related to Steiner triple systems with parameters $(p^d(p^d-1)/6, 3(p^d-3)/2, (p^d + 3)/2, 9)$, where $p^d$ is greater than $9$ and congruent to $7 \pmod {12}$, as well as Steiner $2$-designs $\S(2, k, v)$ with parameters $(p^d(p^d - 1)/(k(k-1)), k(p^d-k)/(k-1), (p^d -1)/(k-1) + (k-1)^2 -2, k^2)$ for $p^d$ sufficiently large and $p^d \equiv k(k-1) + 1 \pmod{2k(k-1)}$. (These Steiner triple systems were shown to be block-regular by Clapham \cite{clapham}, and the Steiner $2$-designs were shown to be block regular by Wilson \cite{wilson}. It is immediate that these will be PDSs; however, PDS families in Steiner $2$-designs do not seem to have been recognized thus far in the literature beyond a recent mention in \cite{PonomarenkoRyabov_2025}.) 


This paper is organized as follows. In Section \ref{sect:prelims}, we outline preliminary results related to PDSs, SRGs, and character theory over local rings necessary for later sections. Section \ref{sect:ott} is dedicated to generalizing the results of Ott to the PDS setting. In Section \ref{sect:comp}, we introduce computational methods for the case when one of $\sqrt{\Delta}$, $v$ has factors which are relatively prime to the others that are used both to rule out the existence of numerous PDSs and also discover several new examples in nonabelian groups; we also provide an example demonstrating how the techniques from this paper can be used when $\sqrt{\Delta}$ and $v$ have the same primes as factors (such as when $\Delta = v$). Finally, we include appendices in which we list parameters that can be ruled out either theoretically or computationally by methods already known (Appendix \ref{app:A}) or by new methods (Appendix \ref{app:B}) as well as a list of open cases and cases in which a PDS exists for sufficiently small $v$ (Appendix \ref{app:C}).


\section{Preliminaries}
\label{sect:prelims}

\subsection{Partial difference sets}

 For our purposes, $G$ will always be a finite group, and $D$ will always denote a regular PDS in $G$. In this paper, we study the structure of $G$ and $D$ utilizing character theory. To do so, we often identify $D$ with the element $\sum_{d\in D} d$ in the group ring $\C[G]$, and we occasionally abuse notation by writing $D = \sum_{d\in D} d$. This alternative description of a PDS is useful due to the following result (see \cite[Theorem 1.3]{ma}).

\begin{lem}\label{D-squared}
	Let $D$ be a regular PDS in a group $G$. Then, in the group ring $\C[G]$,
		\begin{equation*}
			D^2 = k1 + \lambda D + \mu (G - D - 1).
		\end{equation*}
\end{lem}

\noindent In fact this condition is sufficient for $D$ to be a regular PDS.

\subsection{Strongly regular graphs}
As we have mentioned, the notion of a strongly regular graph is closely related to that of a PDS, so we now take the opportunity to formally introduce SRGs. For further reference, see \cite[Chapter 10]{godsil-royle}. We call a graph $\Gamma$ a $(v, k, \lambda, \mu)$-SRG on $v$ vertices if $\Gamma$ is $k$-regular and if, for any pair of vertices $u,v$ in $\Gamma$, either $u, v$ have $\lambda$ neighbors in common when $uv$ is an edge, or they have $\mu$ neighbors in common when $uv$ is not an edge. Now, if $\Gamma$ is a $(v, k, \lambda, \mu)$-SRG with an adjacency matrix $A$, then $A$ has eigenvalues $k$,
	\begin{center}
	$\theta_1 = \frac12(\lambda - \mu + \sqrt\Delta), \;\;\; \theta_2 = \frac12(\lambda - \mu - \sqrt\Delta)$
	\end{center}
\noindent where $\Delta \colonequals (\lambda - \mu)^2 + 4(k - \mu)$; we call this value the {\it discriminant} of the PDS. The eigenvalue $k$ is known as the Perron eigenvalue, and it has multiplicity $1$. The eigenvalues $\theta_1$ and $\theta_2$ occur with multiplicities $m_1$ and $m_2$ in the adjacency matrix of the associated SRG, where
            \[ m_i = \frac{1}{2} \left( (v-1) + (-1)^i \frac{2k + (v-1)(\lambda - \mu)}{\sqrt{\Delta}}\right)\]
             for $i = 1, 2$ (see \cite[Chapter 10]{godsil-royle}).
As we will see, the values $\theta_1, \theta_2$, and $\Delta$ are of particular importance in the study of PDSs.

We set aside now two important relations for the eigenvalues, which we will use freely throughout the paper. Their proof follows immediately from the definition for $\theta_1, \theta_2$ provided above.

\begin{lem}\label{eigen-relations}
Let $\Gamma$ be a $(v, k, \lambda, \mu)$-SRG with eigenvalues $k, \theta_1, \theta_2$. Then,
	\begin{itemize}
		\item[(i)] $\theta_1\theta_2 = \mu - k$, and
		\item[(ii)] $\lambda - \mu = \theta_1 + \theta_2$.
	\end{itemize}
\end{lem}

Note that for our purposes $\sqrt\Delta$ will always be an integer. Indeed, the case when $\sqrt\Delta$ is not an integer is quite restricted: the SRG $\Gamma$ must be a \textit{conference graph}, which has parameters $(v, (v-1)/2, (v-5)/4, (v-1)/4)$; see \cite[Lemmas 10.3.2, 10.3.3]{godsil-royle}. 

Furthermore, we say that the $(v, k, \lambda, \mu)$-SRG is \textit{imprimitive} if either $\mu = 0$ or $\mu = k$ and is \textit{primitive} otherwise. As noted in \cite[Section 1.1.3]{brouwer-maldeghem}, imprimitive SRGs are either the union of complete graphs of the same size or are complete multipartite graphs. We will only consider parameter sets that correspond to primitive strongly regular graphs, and so we will always assume that $0 < \mu < k$.

Finally, we remark that the existence of a $(v,k,\lambda, \mu)$-SRG is equivalent to the existence of its complement, which is itself a $(v, v - k - 1, v - 2k +\mu - 2, v - 2k + \lambda)$-SRG, and, indeed, the complement of a regular $(v, k, \lambda, \mu)$-PDS in $G - \{1\}$ is itself a $(v, v - k - 1, v - 2k +\mu - 2, v - 2k + \lambda)$-PDS in $G$. Note that one can rule out the existence of a PDS by ruling out the existence of its complement. 

\subsection{Some families of SRGs}
\label{subsec:famSRG}
We have already mentioned the particular class of finite geometries known as generalized quadrangles. We now provide some definitions and notation for generalized quadrangles, as we will need them in Proposition \ref{odd-order-GQ}. A \textit{generalized quadrangle} of order $(s, t)$ is an incidence geometry of points and lines satisfying the following conditions:
	\begin{itemize}
		\item[(i)] Each point is incident with $t+1$ lines, and two points are both incident with at most one line.
		\item[(ii)] Each line is incident to $s+1$ points, and any two lines are both incident to at most one point.
		\item[(iii)] If $P$ is a point and $\ell$ a line not incident with $P$, there is exactly one point on $\ell$ collinear with $P$.
	\end{itemize}
\noindent We denote a $\GQ$ of order $(s, t)$ -- which may not be unique -- by $\GQ(s, t)$. We may construct an incidence graph $\Gamma$ from a GQ $S$ by identifying the points of $S$ with the vertices of $\Gamma$, and specifying that two points are adjacent if they are collinear. It is an easy exercise that $\Gamma$ is a $((s+1)(st+1), s(t+1), s-1, t+1)$-SRG. See \cite[Section 2.2.9]{brouwer-maldeghem} for further information on GQs.

One particularly important family of strongly regular graphs are those related to Steiner $2$-designs, which we use in Proposition \ref{Clapham-PDS-family}, Theorem \ref{WilsonBuratti-PDS-family}, Proposition \ref{prop:buratti}, and Proposition \ref{Fuji-PDS-family}. A \textit{Steiner $2$-design}, $\S(2, k, v)$, is a $v$-element set, along with a collection of $k$-subsets of this set, called {\it blocks}, such that any two elements are contained in exactly one block. The families of Steiner $2$-designs in Proposition \ref{Clapham-PDS-family}, Theorem \ref{WilsonBuratti-PDS-family}, Proposition \ref{prop:buratti}, and Proposition \ref{Fuji-PDS-family} are all algebraic constructions using the elements of an abelian group $G$. In these cases, the Steiner $2$-design is represented by a distinguished set $\mathscr{B}$ of blocks, known as {\it base blocks}. The remaining blocks take the form $g + B$ for $B \in \mathscr{B}$ and $g \in G$.

A Steiner $2$-design has a natural representation as a strongly regular graph if one identifies the blocks of the Steiner $2$-design with the vertices of the graph such that two vertices are adjacent if the intersection of their corresponding blocks has one element. If $v$ is the number of elements in the Steiner $2$-design, then the corresponding SRG has parameters
	\[ \left(\frac{v(v - 1)}{k(k-1)}, \frac{k(v-k)}{k-1}, \frac{v -1}{k-1} + (k-1)^2 -2, k^2 \right).\] 
We will speak often of one particular class of Steiner $2$-designs, known as Steiner triple systems, which occur when $k = 3$. For further information on Steiner triple systems, Steiner $2$-designs, and Steiner systems more generally, see \cite[Section I.4]{handbook-comb-des}.

\subsection{Character theory and representations over local rings}

Throughout this paper, we assume that the reader is familiar with fundamental results from representation theory and character theory. For reference, the reader may consult \cite[Chapter 2]{isaacs}. In addition, we also utilize local rings and local representations. We place here the main facts.

Local rings are constructed for particular primes to describe the ``local behavior" around this prime. They are used in a variety of different settings, though for our purposes we only consider them in the field $\Q(\zeta)$ for $\zeta$ an $n$th root of unity.

\begin{defn}\label{def-R}
	For a specified prime $p$, and prime ideal $\p_0 \subseteq \Z[\zeta]$ such that $p \in \p_0$, define the ring $\frakR$ by
		\begin{equation*}
			\frakR = \biggl\{\frac{x}{y}: x,y\in\Z[\zeta], y\not\in\p_0\biggr\}.
		\end{equation*}
	$\frakR$ is known as the \textit{local ring} at $\p_0$.
\end{defn}

The local ring $\frakR$ has the following important properties, which are proven in \cite[4.1]{feit}. For a more general explanation of the subject for non-experts, see \cite[16]{DummitFoote_2004}.

\begin{lem}\label{max-ideal}
	Let $p$ be prime, and let $\frakR$ be the local ring at $\p_0$ such that $p \in \p_0$, and let $G$ be an order $n$ group, then
		\begin{itemize}
			\item[(i)] any $\Q(\zeta)$-representation of $G$ is similar to a $\frakR$-representation of $G$,
			\item[(ii)] $\frakR$ is a PID, and
			\item[(iii)] $\frakR$ has a unique maximal ideal $\p$ containing $p$.
		\end{itemize}
\end{lem}

We also need the following theorem of Brauer, see \cite[Theorem 10.3]{isaacs}.

\begin{thm}\label{Q(z)-rep}
	Let $G$ have exponent $n$, and let $\zeta$ be an $n^\text{th}$ root of unity. Then every $\C$-character of $G$ is afforded by a $\Q(\zeta)$-representation of $G$.
\end{thm}

We will use part (ii) of Lemma \ref{max-ideal} to construct representations of $G$ over fields of prime characteristic from the ring $\frakR$. We then relate these representations over prime characteristic to the $\C$-characters of $G$ using Theorem \ref{Q(z)-rep} and part (i) of Lemma \ref{max-ideal}. An explanation of this method is provided in the brief introduction to the proof of Theorem \ref{linear-invariance}. For a deeper treatment, see the overview in \cite[Chapter 15]{isaacs}, or see \cite{goldschmidt} for a dedicated analysis of the subject from the perspective of Brauer characters.


\section{The character approach}
\label{sect:ott}

In this section, we generalize (and, in some cases, expand upon) Ott's results about reversible $((1+s)(1+s^2), s^2 + s + 1, s+1)$-DSs \cite{ott}. Ott was able to guarantee several restrictions on reversible difference sets corresponding to generalized quadrangles of order $s$ for which $\gcd(v, \sqrt\Delta) = 1$. In Subsection \ref{subsec:ott1}, we first prove results that will hold for any $v$ and $\sqrt{\Delta}$.  In Subsection \ref{subsec:ott2}, we expand Ott's results to analyze an arbitrary regular $(v, k, \lambda, \mu)$-partial difference set for which there is at least one prime $p$ dividing $v$ which does not divide $\sqrt\Delta$. (Note that, in the abelian setting, $v$ and $\sqrt{\Delta}$ must have the same prime divisors \cite[Theorem 6.1]{ma}, so the results in Subsection \ref{subsec:ott2} only apply in the nonabelian setting.) Largely, the results in Subsections \ref{subsec:ott1} and \ref{subsec:ott2} carry over \textit{mutatis mutandis} from results in \cite{ott}; we include these proofs for the sake of completeness.  In Subsection \ref{subsec:normal}, we prove results about the intersection of a PDS with cosets of certain normal subgroups; like the results of Subsection \ref{subsec:ott2}, the results in Subsection \ref{subsec:normal} only apply in the nonabelian setting when $v$ has at least one linear character whose order is coprime to $\sqrt{\Delta}$.  Finally, in Subsection \ref{subsec:applications}, we provide some examples of applications of these results.


Throughout this section, $G$ will be a group containing a regular $(v,k,\lambda,\mu)$-PDS, $D$, with an irreducible representation $\repF: G \longrightarrow M_n(\C)$. Any representation $\repF$ given by a mapping into complex matrices can be extended to a $\C[G]$-representation by sending $\sum_{g\in G} c_gg$ to $\sum_{g\in G} c_g\repF(g)$. For this reason, we do not distinguish between the functions $\C[G]\longrightarrow M_n(\C)$ and $G\longrightarrow M_n(\C)$. For an irreducible representation $\repF$, we denote the vector space underlying $\repF$ as $V_\repF$. Additionally, $\Irr(G)$ will be the set of complex irreducible characters of $G$, $\chi$ will always be a general irreducible character, $\xi$ will always be a linear character, and $\mathcal{L}$ will be the group of linear characters.

\subsection{Results that hold for any $v$, $\sqrt{\Delta}$}
\label{subsec:ott1}

Define the function \[\Phi(h) = |C_G(h)||h^G\cap D|,\] where $h^G = \{g^{-1}hg:g\in G\}$ is the conjugacy class of $h$ in $G$. This is a class function, since $h^G = (g^{-1}hg)^G$ and $|C_G(g^{-1}hg)| = |C_G(h)|$. Thus, $\Phi$ may be represented as a sum $\Phi(h) = \sum_{\chi\in\Irr(G)}c_\chi\chi(h)$ of $\chi\in\Irr(G)$ with coefficients $c_\chi$ (see \cite[Section 1]{ott}).

\begin{lem}\label{phi-function}
	For all $g \in G$,  $\Phi(g) = \sum_{\chi\in\Irr(G)} \overline{\chi(D)} \chi(g)$.
\end{lem}
	\begin{proof}
	Denote the distinct conjugacy classes of $G$ by $h_1^G, \dots, h_r^G$. By the orthogonality relations, 
		\begin{align*}
			c_\chi &= [\Phi,\chi] = \frac1{|G|}\sum_{h\in G}\Phi(h)\overline{\chi(h)} = \frac1{|G|}\sum_{h\in G}|C_G(h)||h^G\cap D|\overline{\chi(h)} = \sum_{h\in G}\frac{|C_G(h)|}{|G|}|h^G\cap D|\overline{\chi(h)}\\
			       &= \sum_{i=1}^r\frac{|C_G(h)||H_i|}{|G|}|h_i^G\cap D|\overline{\chi(h_i)} = \sum_{i=1}^r |h_i^G\cap D|\overline{\chi(h_i)} = \sum_{x \in D}\overline{\chi(x)} = \overline{\chi(D)}.
		\end{align*}
	So, $\Phi(g) = \sum_{\chi\in\Irr(G)} c_\chi \chi(g) = \sum_{\chi\in\Irr(G)} \overline{\chi(D)} \chi(g)$ as claimed.
	\end{proof}

Now, consider a nonprincipal irreducible character $\chi$ of $G$ afforded by the representation $\repF$; compare to \cite[Lemma 2.2]{ott}.

\begin{lem}\label{eigenvalues}
	For $D$ a regular $(v, k ,\lambda, \mu)$-PDS inside $G$ and a representation $\repF$ affording a nonprincipal irreducible character $\chi$ with underlying vector space $V_\repF$, $\repF(D)$ has two eigenvalues $\theta_1 = \frac12(\lambda-\mu + \sqrt{\Delta})$ and $\theta_2 = \frac12(\lambda-\mu - \sqrt{\Delta})$ with corresponding eigenspaces $V_\repF(\theta_1)$ and $V_\repF(\theta_2)$ such that $V_\repF = V_\repF(\theta_1) \oplus V_\repF(\theta_2)$, and
		\begin{center}
			$\chi(D) = \dim V_\repF(\theta_1)\theta_1 + \dim V_\repF(\theta_2)\theta_2$.
		\end{center}
	In particular, if $\chi$ is a nonprincipal linear character, then $\chi(D) = \theta_1$ or $\chi(D) = \theta_2$.
\end{lem}
		\begin{proof}
		Let $\repF$ be an irreducible representation of $G$ affording $\chi$ with underlying vector space $V_\repF$. By Lemma \ref{D-squared}, $D^2 = (k-\mu)1 + (\lambda-\mu)D + \mu G$. Then,
			\begin{align*}
				\repF(D^2) &= \repF\bigl((k-\mu)1 + (\lambda-\mu)D + \mu G\bigr)\\ 
				\repF(D^2) &= (k-\mu)\repF(1) + (\lambda-\mu)\repF(D) +\mu(\repF(G))\\
				\repF(D^2) &= (k-\mu)I + (\lambda-\mu)\repF(D).
				\intertext{Note that $\repF(G) = 0$, for $\repF$ is an irreducible representation. Therefore,}
				0 &= \repF(D^2) -(k-\mu)I - (\lambda-\mu)\repF(D).
				\intertext{Let $\alpha$ be an eigenvalue of $\repF(D)$ with corresponding eigenvector $v$. Then,}
				0 &=\Bigl(\repF(D)^2 - (\lambda-\mu)\repF(D)) -(k-\mu)I\Bigr)v\\
				0&=\alpha^2 + (\mu - \lambda)\alpha+(\mu -k).
			\end{align*}
		So $\alpha = \frac12(\lambda - \mu \pm\sqrt{\Delta})$, with $\Delta = 4(k-\mu) + (\lambda-\mu)^2$. Thus, $\alpha$ equals one of $\theta_1 = \frac12(\lambda - \mu + \sqrt{\Delta})$ or $\theta_2 = \frac12(\lambda - \mu - \sqrt{\Delta})$ as claimed.
		
		We have just shown that $\repF(D)$ is a solution to the polynomial $z^2 - (\lambda - \mu)z + \mu - k$, where $z^2 - (\lambda - \mu)z + \mu - k = (z - \theta_1)(z-\theta_2)$ by Lemma \ref{eigen-relations}. Moreover, $\theta_1 \neq \theta_2$. For if $\theta_1 = \theta_2$, then we must have $\mu \ge k$, but this is impossible, for $\mu \le k$ by definition, and if $\mu = k$, then $\lambda = \mu$ cannot hold. So, $\theta_1 \neq \theta_2$. Therefore, the minimal polynomial of $\repF(D)$ is degree $1$ or $2$ with distinct roots, so $\repF(D)$ is diagonalizable with $\theta_1, \theta_2$ on the diagonal. If $V_\repF(\theta_i)$ is the eigenspace corresponding to the eigenvectors of $\theta_i$, then $V_\repF = V_\repF(\theta_1) \oplus V_\repF(\theta_2)$ and $\chi(D) = \dim V_\repF(\theta_1)\theta_1 + \dim V_\repF(\theta_2)\theta_2$. Note that if $\chi$ is a linear character, then it has one eigenspace of dimension $1$, in which case $\chi(D) = \theta_1$ or $\chi(D) = \theta_2$.
		\end{proof}

The following is a direct analogue of \cite[Lemma 2.3]{ott}.

	\begin{lem}\label{kernel-D}
		Assume $G$ contains a regular $(v, k, \lambda, \mu)$-PDS with nonprincipal eigenvalues $\theta_1$, $\theta_2$, $\Cay(G, D)$ is not a conference graph, and $\mu > 0$. Let $\mathcal{L}$ be a nontrivial group of linear characters of $G$, and let $\xi\in \mathcal{L}$ be a nonprincipal linear character of order $p$ with kernel $N$. Then, $p | (k-\xi(D))$, $|N\cap D| = \frac{k-\xi(D)}{p} + \xi(D)$, and $|Na \cap D| = \frac{k-\xi(D)}{p}$ for every $a \not \in N$.
	\end{lem}
		\begin{proof}
		Now, $D\not\subseteq\ker\xi$, for $\xi$ is a nonprincipal irreducible linear character. Therefore, $\xi(D) = \theta_i \neq k$ by Lemma \ref{eigenvalues}, so for some $g\in D$ we know that $\xi(g) = \omega$, for $\omega$ a $p^\text{th}$ root of unity. Recall from Lemma \ref{phi-function} that $\overline{\chi(D)} = \sum_{1 = i}^r|h_i^G\cap D|\overline{\chi( h_i)}$ for conjugacy class representatives $h_1, \dots, h_r$. Then,
			\begin{align*}
				0 &= \xi(D) - \xi(D)\\
				  &= \sum_{i=1}^r |h_i^G\cap D|\overline{\xi(h_i)} - \xi(D).\\
				\intertext{We re-express this as a sum over the cosets $Ng_0, \dots, Ng_{p-1}$ for coset representatives $g_0, \dots, g_{p-1} \in G$. In addition, we may order the $g_0, \dots, g_{p-1}$ so that $\overline{\xi(Ng_i)} = \omega^i$.}
				0 &= \sum_{i=0}^{p-1}|Ng_i\cap D|\overline{\xi(Ng_i)} - \xi(D)\\
				  &=\sum_{i=1}^{p-1}|Ng_i\cap D|\omega^i + (|N\cap D| - \xi(D))\omega^0
			\end{align*}
		So, $\omega$ is a root of the polynomial $p(x) = \sum_{i=1}^{p-1}|Ng_i\cap D|x^i + (|N\cap D| - \xi(D))$. Since $\xi(D) = \theta_\alpha \in \Q$ by assumption, then the cyclotomic polynomial $x^{p-1} + \dots + x + 1$ divides $p(x)$, forcing
			\begin{equation*}
				|N\cap D| - \xi(D) = |Ng_1\cap D|  = \dots = |Ng_{p-1}\cap D| = c\in \Z.
			\end{equation*}
			
		Furthermore, 
			
			\begin{equation*}
				k = 1_G(D) = \sum_{j=0}^{p-1}|Ng_j\cap D| = c+\xi(D) + (p-1)c = \xi(D) + pc
			\end{equation*}
			
		so $\frac{k-\xi(D)}{p} = c\in \Z$ and $|N\cap D| = \frac{k-\xi(D)}{p} + \xi(D)$. In particular, $p$ divides $k-\xi(D)$.
		\end{proof}

We take this opportunity to prove some consequences of these results.

\begin{lem}\label{lem:sumPhi}
Assume $G$ contains a regular $(v,k,\lambda, \mu)$-PDS $D$.
\begin{itemize}
    \item[(i)] If $\chi_1$ and $\chi_2$ are two irreducible characters of $G$, $\z$ is a primitive $v^\text{th}$ root of unity, and $\sigma$ is an automorphism of $\Q(\z)$ such that $\chi_2 = \chi_1^\sigma$, then $\chi_2(D) = \chi_1(D)$.  In particular, for all $\chi \in \Irr(G)$, $\chi(D) \in \Z$ and hence $\overline{\chi}(D) = \chi(D)$.
    \item[(ii)] If $h_1^G, \dots, h_r^G$ denote the distinct conjugacy classes of $G$ and $\chi \in \Irr(G)$, then
            \[ \sum_{i = 1}^r |h_i^G \cap D|\chi(h_i) = \chi(D).\]
    \item[(iii)] If $K$ is a normal subgroup of $G$, then
            \[ \sum_{g \in K} \Phi(g) = |G|\cdot|K \cap D|.\]        
    \item[(iv)] If $m_i$ is the multiplicity of the eigenvalue $\theta_i$, then 
    \[ m_i = \sum_{\repF \text{ irreducible}} (\dim V_\repF)(\dim V_\repF(\theta_i)).\]
    
\end{itemize}            
\end{lem}
        
\begin{proof}
First, (i) follows directly from Lemma \ref{eigenvalues} (so $\chi_1(D) \in \Z$) and the fact that $\sigma|_\Q$ is the identity:
 \[ \chi_2(D) = \chi_1^\sigma(D) = \chi_1(D)^\sigma = \chi_1(D).\]

Next, for (ii), Lemma \ref{phi-function} and its proof imply that
\[ \sum_{i = 1}^r |h_i^G \cap D|\chi(h_i) = \frac{1}{|G|}\sum_{h \in G} \chi(h) \overline{|h^G \cap D|} = [\chi, \Phi] = \chi(D).\]

For (iii), if $g_1^G, \dots, g_s^G$ are the distinct conjugacy classes of $K$, then
    \begin{align*}
         \sum_{g \in K} \Phi(g) &= \sum_{g \in G} |g^G \cap D||C_G(g)|\\
                                &= \sum_{i = 1}^s |g_i^G||g_i^G \cap D||C_G(g_i)|\\
                                &= |G| \cdot \sum_{i = 1}^s |g_i^G \cap D|\\
                                &= |G| \cdot |K \cap D|.
    \end{align*}

Finally, since the action of $G$ on the associated SRG can be identified with the right regular representation $\rho$ of $G$, and it is well known that its corresponding character $\chi_\rho$ satisfies
            \[ \chi_\rho = \sum_{\chi \in \Irr(G)} \chi(1) \chi.\]
Lemma \ref{eigenvalues} implies that
    \[ m_i = \sum_{\repF \text{ irreducible}} (\dim V_\repF)(\dim V_\repF(\theta_i))\]
for $i = 1,2$, proving (iv).    
\end{proof}
        
		\begin{rem}\label{diagonalization}
		 \begin{itemize}
            \item[(1)] We again remark that the results in this subsection do not require that $G$ be an abelian or nonabelian group or any conditions on the factors of $|G|$ or $\sqrt\Delta$. Thus, these results may be applied even in cases when our techniques in the rest of the paper do not otherwise work; an example of such an application is given in Subsection \ref{Non-modular-techniques}
		  \item[(2)] Unlike in the abelian case, while the conditions listed in Lemma \ref{eigenvalues} are necessary, they are not sufficient for the existence of a PDS in a nonabelian group; see Remark \ref{rem:111}.
		  \item[(3)] The condition that $D = D^{(-1)}$ is necessary in the proofs of \cite[Lemma 2.2]{ott} and Lemma \ref{eigenvalues}. Unfortunately, most known difference sets are not reversible (and DSs are only conjectured to be reversible in very limited circumstances), so new ideas will be necessary to extend these results to general DSs.
		  \item[(4)] The main technique in \cite{ott} and our own Theorem \ref{linear-invariance} is to extract information about $D$ by analyzing $\repF(D)$ over fields of characteristic $p$, when $p$ does not divide $\sqrt{\Delta} = \theta_1 - \theta_2$ but does divide $|G|$. A critical element in this proof is that $\repF(D)$ induces a representation on a certain field $L$ of characteristic $p$, and that the induced representation $\overline{\repF(D)}$ on $L$ is diagonalizable, for its minimal polynomial is $(z - \theta_1)(z-\theta_2) \pmod p$, where $\theta_1 \not \equiv \theta_2 \pmod p$. As a result, while the results of this subsection still hold, the techniques in the following subsections largely fail if we ever have $|G|$ sharing all of its prime divisors with $\theta_1 - \theta_2$, for then the minimal polynomial of $\repF(D)$ for any representation $\repF$ is $(z - \theta_1)(z-\theta_2) \equiv (z-\theta_1)^2 \pmod p$, so $\repF(D)$ is not diagonalizable. Unfortunately, this case happens reasonably often, so for such parameter sets techniques involving local rings are ineffective. In Subsection \ref{Non-modular-techniques}, we provide an example computation for how character theory may still be applied to PDSs where $\sqrt\Delta$ does divide $v$.
		 \end{itemize}

		\end{rem}

\subsection{Results that hold when a prime $p$ divides $v$ but not $\sqrt{\Delta}$}
\label{subsec:ott2}

We now prove an analogue of \cite[Corollary 2.1]{ott}. Note again that the assumption that $\mu > 0$, while necessary, amounts to assuming that the associated SRG is not the union of complete subgraphs of the same size. Going forward, we will tacitly assume $\mu > 0$.

	\begin{lem} \label{coprime-condition}
		Let $G$ be a finite group containing a regular $(v, k, \lambda, \mu)$-PDS $D$ that has nonprincipal eigenvalues $\theta_1 > \theta_2$, and let $\sqrt{\Delta} = \theta_1 - \theta_2$. Suppose $p$ is a prime dividing $v$ but not $\sqrt{\Delta}$. Assuming $\mu > 0$, there exists a factorization $\mu = \mu_1 \mu_2$ such that $v_1 = (k - \theta_1)/\mu_1$, $v_2 = (k - \theta_2)/\mu_2$, $v = v_1 v_2$, and $p$ divides only one of $v_1$, $v_2$.
	\end{lem}

	\begin{proof}
	Recall for a primitive SRG that $v\mu = (k-\theta_1)(k-\theta_2)$ (see \cite[Section 1.1.1]{brouwer-maldeghem}). For some factorization of $\mu$ into $\mu_1\mu_2$, we have $v = \frac{k-\theta_1}{\mu_1}\frac{k-\theta_2}{\mu_2}$ where $v_1 =\frac{k-\theta_1}{\mu_1}$ and $v_2=\frac{k-\theta_2}{\mu_2}$ are integers. If $p$ divides $v_1$ and $v_2$,
		\begin{center}
			$0 \equiv 0 - 0 \equiv (k-\theta_2) - (k-\theta_1) \equiv \theta_1-\theta_2 \equiv \sqrt{\Delta}\pmod p$
		\end{center}
	contradicting our assumption that $p$ and $\Delta$ were relatively prime. 
	\end{proof}

Compare the following result to \cite[Lemma 2.4]{ott}.

\begin{lem}\label{constant-lin-char}
Let $G$ be a group with a regular $(v, k, \lambda, \mu)$-PDS $D$ with nonprincipal eigenvalues $\theta_1 > \theta_2$, and let $\mathcal{L}$ be the group of linear characters of $G$. Let $P\subseteq \mathcal{L}$ be a Sylow $p$-subgroup of $\mathcal{L}$, and suppose further that $p$ and $\sqrt{\Delta} = \theta_1 - \theta_2$ are coprime. Then, for any nontrivial $\xi' \in P$,
	\begin{center}
		$\displaystyle\sum_{\xi\in P}\xi(D) = k + (|P|-1)\xi'(D)$.
	\end{center}
In particular, all nontrivial linear characters from $P$ are equivalent over $D$.
\end{lem}
	\begin{proof}
	First, $|P|$ divides $k-\xi(D)$ by Lemmas \ref{coprime-condition} and \ref{kernel-D}. Set $S \colonequals \sum_{\xi\in P}\xi(D)$. Then,
		\begin{align*}
		S &=\sum_{\xi\in P}\xi(D)\\
		  &=1_G(D) + \sum_{\xi\in P - \{1_G\}}\xi(D)\\
		\intertext{Fix $\theta_\alpha$ to be an output of at least one $\xi \in P$ on $D$, and let $r$ of the remaining $\xi$ evaluate to $\theta_\beta$ on $D$. Then there are $|P| - r - 1$ eigenvalues $\theta_\alpha$ in the sum. So,}
		S &=1_G(D) + r\theta_\beta + (|P| - r - 1)\theta_\alpha\\
		  &= 1_G(D)-\theta_\alpha + r\theta_\beta + (|P| - r)\theta_\alpha\\
		  &= k-\theta_\alpha + r\theta_\beta + (|P| - r)\theta_\alpha
		\intertext{On the other hand, if we define $N_p$ to be the kernel of $P$,}
		S &= \sum_{g\in D}\sum_{\xi\in P}\xi(g)\\
		  &= \sum_{g\in D\cap N_p}\sum_{\xi\in P}\xi(g)\ +  \sum_{g\in D- N_p}\sum_{\xi\in P}\xi(g)\\
		  &=\sum_{\xi\in P}\xi(D\cap N_p)\ +  \sum_{g\in D- N_p}0\\
		  &=\sum_{\xi\in P}|D\cap N_p| + 0\\
		  &= |P||D\cap N_p|
		\end{align*}
	Recall that $|P|$ divides $k-\xi(D)$. Thus, evaluating $k-\theta_\alpha + r\theta_\beta + (|P| - r)\theta_\alpha$ modulo $|P|$ gives $0 \equiv r\theta_\beta - r\theta_\alpha \equiv \pm r\sqrt{\Delta}$. Further, $\sqrt{\Delta}$ and $|P|$ are coprime, so $r = |P|$ or $r = 0$. If $r = |P|$, then every $\xi$ evaluates to $\theta_\beta$, which contradicts our selection of $\theta_\alpha$ to be the output of some $\xi$, so $r = 0$. Therefore, all the linear characters of $P$ are equivalent over $D$.
	\end{proof}

For the next theorem (which is an analogue of \cite[Theorem 2.5]{ott}), we once again set $\zeta$ to be a $|G|^\text{th}$ root of unity. Recall the definition of the local ring $\frakR$ given in Definition \ref{def-R}. If $\frakR$ is the local ring for the prime ideal $\p$, then we define the field $L \colonequals \frakR / \p$. Note that $L$ is a finite field of characteristic $p$.

\begin{thm}\label{linear-invariance}
Let $G$ be a group with a regular $(v, k, \lambda, \mu)$-PDS $D$ with nonprincipal eigenvalues $\theta_1 > \theta_2$, and let $\mathcal{L}$ be the group of linear characters of $G$.  Let $\xi\in \mathcal{L}$ be a linear character of order $p^a$ with $p$ coprime to $\sqrt\Delta = \theta_1 - \theta_2$, and let $\chi\in\Irr(G)$. Then, $\xi\chi(D) = \chi(D)$.
\end{thm}

The following proof is quite technical, so we provide a broad explanation of its various moving parts. We would like to show that $\chi(D) = \xi\chi(D)$. To do this, we use Theorem \ref{Q(z)-rep} to take a $\Q(\zeta)$-representation $\repF$ affording the $\C$-character $\chi$, use Lemma \ref{max-ideal} to assume that $\repF$ is a $\frakR$-representation, and construct a homomorphism, which we denote by $\overline{\phantom{x}}$, from the $\frakR$-module underlying $\repF$ to an $L$-module, which is also an $L$-vector space since $L$ is a field. Under this homomorphism, $\overline{\xi(D)\repF(D)} = \overline{\repF(D)}$, since for all $x\in G$, $\overline{\xi(x)}$ is a scalar whose order divides $p^a$ in $L$, so $\overline{\xi(x)}$ must evaluate to $1$. We then show that for any matrix $\repF(D)$ with $\Q(\zeta)$-eigenspace $V_\repF(\theta_\alpha)$ and $\frakR$-eigenspace $M_\repF(\theta_\alpha)$, that
	\[\dim_{\Q(\zeta)} V_\repF(\theta_\alpha) = \rk M_\repF(\theta_\alpha) = \dim_L \overline{M_\repF(\theta_\alpha)} = \dim_L \overline{M_{\overline{\repF}}}(\theta_\alpha),\]
	where $\overline{M_{\overline{\repF}}}(\theta_\alpha)$ is the eigenspace of the representation $\overline\repF$ induced by $\repF$ on the $L$-vector space. This finally allows us to conclude
	\[\dim_{\Q(\zeta)} V_\repF(\theta_\alpha) = \dim_L \overline{M_{\overline{\repF}}}(\theta_\alpha) \stackrel{*}{=}  \dim_L \overline{M_{\overline{\xi\repF}}}(\theta_\alpha) = \dim_{\Q(\zeta)} V_{\xi\repF}(\theta_\alpha), \]
	where we apply the equality $\overline{\xi\repF(D)} = \overline{\repF(D)}$ at $*$. By Lemma \ref{eigenvalues}, equality of the dimensions of the eigenspaces is enough to establish $\xi\chi(D) = \chi(D)$.

	\begin{proof}[Proof of Theorem \ref{linear-invariance}]
	Let $\chi$ be a $\C$-character in $\Irr(G)$. By Theorem \ref{Q(z)-rep}, we may pick a $\Q(\zeta)$-representation $\repF$ of $G$ affording $\chi$. We denote the $\Q(\zeta)$-vector space underlying $\repF$ as $V_\repF$, and we pick a basis $B$ for this space. We define $\frakR$ as in Definition \ref{def-R}, noting that it has a unique maximal ideal $\p$ containing $p$ by Lemma \ref{max-ideal}. Now, $V_\repF$ contains a $\frakR[G]$-module defined by taking all $\frakR$-linear combinations of the vectors $\{\repF(g)v_i:v_i\in B, g\in G\}$. Call this $\frakR$-module $M_\repF$. Since $\frakR$ is a PID by Lemma \ref{max-ideal}, we may pick a basis $w_1, \dots, w_n$ of $M_\repF$. This is also a basis of $V_\repF$, for $B \subseteq M_\repF$. Therefore, we may coordinatize the space $V_\repF$ by the basis vectors $w_1, \dots, w_n$, so that all $\repF(g)$ have their entries in $\frakR$. Thus, $\repF(D)$ is an $\frakR$-matrix, proving part (i) of Lemma \ref{max-ideal}. For further reference, see especially the exposition of Feit \cite[4.1 in Chapter 4]{feit}.
	
	We define the submodule $\p M_\repF \colonequals \{ w \in M_\repF : w = qv, q \in \p, v \in M_\repF\}$. This is a closed subspace of $M_\repF$, so $M_\repF/\p M_\repF$ is well-defined. We therefore define the natural projection $\overline{\phantom{x}} : M_\repF \longrightarrow M_\repF/\p M_\repF$. For an $\frakR$-representation $\repF$, there is a corresponding representation $\overline{\repF}$ on $M_\repF/\p M_\repF$ by $\overline{\repF(g)}\overline w = \overline{\repF(g)w}$. As before, $\overline{\xi(g)} \equiv 1 \pmod \p$, so for any vector $w \in M_\repF$, we have $\overline{\xi(g)} \overline{w} = \overline{w}$. Thus, for all $w \in V_\repF$,
		\begin{align*}
			\overline{\xi\repF(D)}\overline{w} &=  \sum_{x\in D} \overline{\xi(x)\repF(x)} \overline w\\
			                                   &= \sum_{x\in D}(1)\overline{\repF(x)} \overline w\\
											&=\overline{\repF(D)}\overline w
		\end{align*}
	so $\overline{\xi\repF(D)} = \overline{\repF(D)}$ as claimed.
	
	Now, let $V_\repF(\theta_\alpha)$ and $M_\repF(\theta_\alpha)$ be the $\Q(\zeta)$-eigenspace and $\frakR$-eigenspace, respectively, for the eigenvalue $\theta_\alpha$ ($\alpha \in \{1, 2\}$) of the matrix $\repF(D)$. We prove the sequence of equalities
		\begin{equation*}
			\dim_{\Q(\zeta)} V_\repF(\theta_\alpha) = \rk M_\repF(\theta_\alpha) = \dim_L \overline{M_\repF(\theta_\alpha)} = \dim_L \overline{M_{\overline\repF}}(\theta_\alpha),
		\end{equation*}
	where $\overline{M_{\overline\repF}}$ is the vector space $M_\repF/\p M_\repF$ acted on by $\overline\repF$. In particular, $\overline{M_{\overline\repF}}(\theta_\alpha)$ is the $\overline{\theta_\alpha}$ eigenspace of $\overline{\repF(D)}$ in $\overline{M_{\overline\repF}}$.
	
	We first show that $\dim_{\Q(\zeta)} V_\repF(\theta_\alpha) = \rk M_\repF(\theta_\alpha)$. We evidently have $\rk M_\repF(\theta_\alpha) \le \dim_{\Q(\zeta)} V_\repF(\theta_\alpha)$, for we may pick a basis $v_1, \dots, v_k$ of $M_\repF(\theta_\alpha)$ by \cite[Theorem 7.1]{lang}. These are linearly independent in $V_\repF(\theta_\alpha)$, for if there exists some $\Q(\zeta)$-linear combination of the $v_i$ such that $\frac{a_1}{b_i}v_1 + \dots + \frac{a_n}{b_n}v_n = \frac{a_j}{b_j}v_j$, then setting $b = b_1\dots b_n$, we have $b(\frac{a_1}{b_i}v_1 + \dots + \frac{a_n}{b_n}v_n) = b\frac{a_j}{b_j}v_j$, a dependence relation in $M_\repF(\theta_\alpha)$, forcing $a_i = 0$ for each $i$. On the other hand, $\dim_{\Q(\zeta)} V_\repF(\theta_\alpha) \le \rk M_\repF(\theta_\alpha)$, for if we have a basis of $V_\repF(\theta_\alpha)$, then the $\frakR$-multiples of the basis vectors are in $M_\repF(\theta_\alpha)$ and thus linearly independent in $M_\repF(\theta_\alpha)$.
	
	We now prove $\dim_L \overline{M_\repF(\theta_\alpha)} \le \rk M_\repF(\theta_\alpha)$. We must show first that there exists a set of basis vectors $v_1, \dots, v_k$ of $M_\repF(\theta_\alpha)$ which do not vanish in $\overline{M_\repF(\theta_\alpha)}$. For this, we use the Elementary Divisor Theorem (see \cite[Chapter III, Theorem 7.8]{lang}), which states that there exists set of vectors $v_1, \dots, v_n$ of $V_\repF(\theta_\alpha)$ so that for some $r_1, \dots, r_k \in \frakR$, the vectors $r_1v_1, \dots, r_kv_k$ form a basis for $M_\repF(\theta_\alpha)$. Now, if $r_iv_i$ is an eigenvector of $\repF(D)$, then $v_i$ itself must be an eigenvector of $\repF(D)$, for $r_i$ is a scalar. Therefore, $v_i$ is in the span of the basis, so we may choose $r_i$ to be $1$. Therefore, $v_1, \dots, v_k$ is a basis for $M_\repF(\theta_\alpha)$. Since $\p M_\repF$ is defined as the set of $\p$ multiples of the basis vectors of $M_\repF$, then $\overline{v_i} \neq 0$ for each $v_i$. Each $\overline{w} \in \overline{M_\repF(\theta_\alpha)}$ may be expressed as a sum of the $\overline{v_i}$ for we may take some pre-image $w$ of $\overline{w}$, and then $a_1v_1 + \dots + a_kv_k = w$, so $\overline{a_1}\overline{v_1} + \dots + \overline{a_k}\overline{v_k} = \overline w$ as needed.

	The equality $\dim_L \overline{M_\repF(\theta_\alpha)} = \rk M_\repF(\theta_\alpha)$ is obtained by the fact that if for the basis vectors $v_1, \dots, v_n$ and coefficients $a_1, \dots, a_n \in \frakR$, there is a dependence $\overline{a_1v_1 + \dots + a_nv_n} = 0$ in $M_\repF / \p M_\repF$, then $a_1v_1 + \dots + a_nv_n = b_1v_1 + \dots + b_nv_n$ where $b_i \equiv 0 \pmod \p$. Therefore, $(a_1 - b_1)v_1 + \dots + (a_n - b_n)v_n = 0$, so $a_1 = b_1, \dots, a_n = b_n$, and thus $a_i \equiv 0 \pmod \p$. So, there exists only a trivial dependence relation among the $\overline{v_i}$ in $\overline{M_\repF(\theta_\alpha)}$, forcing $\rk M_\repF(\theta_\alpha) \le \dim_L \overline{M_\repF(\theta_\alpha)}$, and $\rk M_\repF(\theta_\alpha) = \dim_L \overline{M_\repF(\theta_\alpha)}$ by extension.

	 Finally, we show that $\dim_L \overline{M_\repF(\theta_\alpha)} = \dim_L \overline{M_{\overline\repF}}(\theta_\alpha)$, where $\overline{M_{\overline{\repF}}}$ is the vector space $M_\repF/\p M_\repF$ acted on by $\overline{\repF}$. The inequality $\dim_L \overline{M_\repF(\theta_\alpha)} \le \dim_L \overline{M_{\overline{\repF}}}(\theta_\alpha)$ is straightforward, for $\overline{M_\repF(\theta_\alpha)}$ is a vector space contained in $\overline{M_{\overline{\repF}}}(\theta_\alpha)$. Now, $\rk M_\repF = \dim V_\repF$, and $\dim_L \overline{M_{\overline\repF}} \le \rk M_\repF$. Therefore, \[\dim V_\repF = \dim V_\repF(\theta_\alpha) + \dim V_\repF(\theta_\beta) = \rk M_\repF(\theta_\alpha) + \rk M_\repF(\theta_\beta) = \dim_L \overline{M_\repF(\theta_\alpha)} + \dim_L \overline{M_\repF(\theta_\beta)}, \] and \[\dim_L \overline{M_\repF(\theta_\alpha)} + \dim_L \overline{M_\repF(\theta_\beta)} \le \dim_L \overline{M_{\overline{\repF}}}(\theta_\alpha) + \dim_L \overline{M_{\overline{\repF}}}(\theta_\beta) \le \dim_L \overline{M_{\overline\repF}} \le \dim V_\repF. \]

	Note, the inequality $\dim_L \overline{M_{\overline{\repF}}}(\theta_\alpha) + \dim_L \overline{M_{\overline{\repF}}}(\theta_\beta) \le \dim_L \overline{M_{\overline\repF}}$ holds because $\overline\theta_\alpha \neq \overline\theta_\beta$, for only in this case is $\overline{\repF}$ diagonalizable. Otherwise, we have $\dim_L \overline{M_{\overline{\repF}}}(\theta_\alpha) + \dim_L \overline{M_{\overline{\repF}}}(\theta_\beta) > \dim_L \overline{M_{\overline\repF}}$.

	Our sequence of inequalities gives \[\dim V_\repF = \dim_L \overline{M_\repF(\theta_\alpha)} + \dim_L \overline{M_\repF(\theta_\beta)} \le \dim_L \overline{M_{\overline{\repF}}}(\theta_\alpha) + \dim_L \overline{M_{\overline{\repF}}}(\theta_\beta) \le \dim V_\repF,\] so the inequalities are saturated, forcing \[\dim_L \overline{M_\repF(\theta_\alpha)} + \dim_L \overline{M_\repF(\theta_\beta)} = \dim_L \overline{M_{\overline{\repF}}}(\theta_\alpha) + \dim_L \overline{M_{\overline{\repF}}}(\theta_\beta).\] By the above equality, $\dim_L \overline{M_{\repF}(\theta_\beta)} \le \dim_L \overline{M_{\overline{\repF}}}(\theta_\beta)$ implies $\dim_L \overline{M_{\overline{\repF}}}(\theta_\alpha) \le \dim_L \overline{M_{\repF}(\theta_\alpha)}$, so we achieve the equality $\dim_L \overline{M_{\overline{\repF}}}(\theta_\alpha) = \dim_L \overline{M_{\repF}(\theta_\alpha)}$.

	Therefore, we have shown that $\dim_{\Q(\zeta)} V_\repF(\theta_\alpha) = \rk M_\repF(\theta_\alpha) = \dim_L \overline{M_\repF(\theta_\alpha)} = \dim_L \overline{M_{\overline\repF}}(\theta_\alpha)$. By the equality $\overline{\repF(D)} = \overline{\xi\repF(D)}$ we also have $\dim_L \overline{M_{\overline\repF}}(\theta_\alpha) = \dim_L \overline{M_{\overline{\xi\repF}}}(\theta_\alpha)$. Therefore, \[\dim_{\Q(\zeta)} V_\repF(\theta_\alpha) = \dim_L \overline{M_{\overline\repF}}(\theta_\alpha) = \dim_L \overline{M_{\overline{\xi\repF}}}(\theta_\alpha) = \dim_{\Q(\zeta)} V_{\xi\repF}(\theta_\alpha).\] Combining this equality with Lemma \ref{eigenvalues},
		\begin{align*}
		\xi\chi(D) &= (\dim_{\Q(\zeta)} V_{\xi\repF}(\theta_1)) \theta_1 + (\dim_{\Q(\zeta)} V_{\xi\repF}(\theta_2) )\theta_2\\
		           &= (\dim_{\Q(\zeta)}V_\repF(\theta_1)) \theta_1 + (\dim_{\Q(\zeta)} V_\repF(\theta_2) )\theta_2\\
		           &= \chi(D),
		\end{align*}
	which proves the theorem.
	\end{proof}

\begin{cor}\label{coprime-linear-char}
Let $G$ be a group with a regular $(v, k, \lambda, \mu)$-PDS $D$ with nonprincipal eigenvalues $\theta_1 > \theta_2$.  Let $\xi,\chi$ be nontrivial linear characters of $G$ of prime-power orders $p^a,q^b$, respectively, such that $p$ and $q$ are coprime to $\sqrt\Delta = \theta_1 - \theta_2$. Then, $\xi(D) = \chi(D)$. Moreover, all linear characters of order coprime to $\sqrt\Delta$ are equivalent over $D$.
\end{cor}
	\begin{proof}
	By Theorem \ref{linear-invariance}, $\chi(D) = \xi\chi(D) = \chi\xi(D) = \xi(D)$. Now, suppose that $\xi$ is a linear character of order coprime to $\sqrt\Delta$. By the fundamental theorem of abelian groups, there exist prime power order linear characters $\xi_1, \dots, \xi_n$ such that $\xi_1\dots\xi_n = \xi$. Again using Theorem \ref{linear-invariance}, we have by induction $\xi(D) = \xi_1\dots\xi_n(D) = \xi_1(D)$, for $\xi_1 = \xi'$ of prime power order. For another linear character $\chi$ of order coprime to $\sqrt\Delta$, we use the same technique to achieve $\chi(D) = \chi'(D)$, for $\chi'$ of prime power order. Since $\chi'(D) = \xi'(D)$, we achieve $\xi(D) = \xi'(D) = \chi'(D) = \chi(D)$ as claimed.
	\end{proof}

\begin{defn}\label{def:pi}
Recall from Lemma \ref{coprime-condition} that, if there exists a prime $p$ dividing $v$ but not $\sqrt{\Delta} = \theta_1 - \theta_2$, then $k - \theta_1$ and $k - \theta_2$ are coprime, and so there exists a factorization $\mu = \mu_1 \mu_2$ such that $v = v_1v_2$, where $v_\alpha = (k - \theta_\alpha)/\mu_\alpha$ for $\alpha = 1,2$, and $p$ divides only one of $v_1$ or $v_2$.  We define the set of primes $\Pi_\alpha$ and $\Pi_\beta$ ($\{\alpha, \beta\} = \{1, 2\}$) with products $\pi_\alpha$ and $\pi_\beta$ defined so that $\pi_\alpha$ and $\pi_\beta$ are the largest integers dividing $v_\alpha$ and $v_\beta$, respectively, such that $\gcd(\sqrt\Delta, \pi_\alpha) = \gcd(\sqrt\Delta,\pi_\beta) = 1$.
\end{defn}

\begin{cor}\label{xi-divides-theta_alpha}
    Let $G$ be a group with a regular $(v, k, \lambda, \mu)$-PDS $D$ with nonprincipal eigenvalues $\theta_1 > \theta_2$, and let $\mathcal{L}$ be the group of linear characters of $G$.  Suppose $\pi_\alpha$ is defined as in Definition \ref{def:pi}, and suppose that $\xi \in \mathcal{L}$ has order dividing $\pi_\alpha$. Then, each $\xi \in \mathcal{L}$ with order coprime to $\sqrt\Delta$ has an order dividing $\pi_\alpha$.
\end{cor}
	\begin{proof}
		Let $\xi \in \mathcal{L}$ of prime order $p$ such that $p$ divides $\pi_\alpha$. Let $\xi \in \mathcal{L}$ have prime order $q$ coprime to $\sqrt\Delta$. Then, by Corollary \ref{coprime-linear-char}, $\xi(D) = \chi(D)$. So, $k - \xi(D) = k - \theta_\alpha = k - \chi(D)$. Using Lemma \ref{kernel-D}, both $p$ and $q$ divide $k - \theta_\alpha$. Since $p, q$ divide $v$, and neither $p$ nor $q$ divide $k - \theta_\beta$ by Lemma \ref{coprime-condition}, we conclude that $q$ must divide $v_\alpha$. Since $q$ is coprime to $\sqrt\Delta$, we have $q \in \Pi_\alpha$. It follows that all primes $q$ dividing $|\mathcal{L}|$ and coprime to $\sqrt\Delta$ must divide $\pi_\alpha$. Therefore, all $q^a$ dividing $|\mathcal{L}|$ must also divide $\pi_\alpha$. In conclusion, for every linear character $\xi$ of order $n$ coprime to $\sqrt\Delta$, the order $n$ is a product of prime powers $q_1^{a_1}\dots q_m^{a_m}$. Since each $q_i^{a_i}$ is coprime to $\sqrt\Delta$, we know $q_i^{a_i}$ divides $\pi_\alpha$, for each $i$, so $q_1^{a_1}\dots q_m^{a_m} = n$ divides $\pi_\alpha$.
	\end{proof}
	
The following theorem is a consequence of Corollary \ref{coprime-linear-char}. The reader may compare Theorem \ref{value-phi(a)} to \cite[Theorem 2.7]{ott}.

\begin{thm}\label{value-phi(a)}
Let $G$ be a group with a regular $(v, k, \lambda, \mu)$-PDS $D$ with nonprincipal eigenvalues $\theta_1 > \theta_2$, let $H$ be a subgroup of linear characters of $G$ with order coprime to $\sqrt{\Delta} = \theta_1 - \theta_2$, and let $N$ be the intersection of the kernels of the characters in $H$.  If $a\not\in N$, then for nontrivial $\xi\in H$, $|C_G(a)||a^G\cap D|=k-\xi(D)$.
\end{thm}	
	\begin{proof}
	Let $\xi\in H$ be a nonprincipal linear character. Let $\chi$ be an arbitrary irreducible character of $G$. We have
		\begin{align*}
			\Phi(a) &= |C_G(a)||a^G\cap D|\\
			        &= \sum_{\xi\in H}\xi(D)\xi(a) + \sum_{\chi\not\in H}\chi(D)\chi(a).\\
			\intertext{By Corollary \ref{coprime-linear-char}, all the nonprincipal $\xi(D)$ are equal. Therefore,}
			\Phi(a) &= k - \xi(D) + \xi(D)\sum_{\xi\in H}\xi(a) + \sum_{\chi\not\in H}\chi(D)\chi(a)\\
			\intertext{By orthogonality, the sum of characters in $H$ is zero, leaving}
			\Phi(a) &= k - \xi(D) + \sum_{\chi\not\in H}\chi(D)\chi(a)
			\intertext{We now show that $\sum_{\chi\not\in H}\chi(D)\chi(a)$ is zero. The abelian group $H$ acts on the set of characters by left multiplication. Denote an orbit of $\chi$ under $H$ by $\calO$. Then it suffices to show that}
			0 &=\sum_{\tau\in \calO} \tau(a)
			\intertext{We let $H_\chi$ be the stabilizer of $\chi$ in $H$. Let $\xi_1,...,\xi_r\in H$ be coset representatives of $H_\chi$ in $H$. Therefore,}
			\sum_{\tau\in \calO} \tau(a) &= \sum_{i=1}^r\xi_i\chi(a)\\
			            &= \left(\sum_{i=1}^r\xi_i(a)\right)\chi(a)
			\intertext{We multiply by the sum of elements in the stabilizer of $\chi$.}
			\left(\sum_{\eta\in H_\chi}\eta(a)\right)\left(\sum_{\tau\in \calO} \tau(a)\right) &=\left(\sum_{\eta\in H_\chi}\eta(a)\right)\left(\sum_{i=1}^r\xi_i(a)\right)\chi(a)\\
			            &= \left(\sum_{\eta\in H_\chi}\sum_{i=1}^r\eta\xi_i(a)\right) \chi(a)\\
			            &=\left(\sum_{\xi\in H}\xi(a)\right) \chi(a)
		\end{align*}
	Again, $\left(\sum_{\xi\in H}\xi(a)\right)\chi(a) = 0$ by the orthogonality relations. If $\sum_{\eta\in H_\chi}\eta(a)\neq 0$, we are done. If it is $0$, then there must be some $\eta$ such that $\eta(a)\neq 1$, since $a$ is then not in the kernel of the stabilizer. In this case, $\eta(a)\chi(a) = \chi(a)$, so $\chi(a)=0$. Either way, $\sum_{\chi\not\in H}\chi(D)\chi(a) = 0$, so $|C_G(a)||a^G\cap D| = k-\xi(D)$.
	\end{proof}

\begin{lem}\label{coprime-centralizer}
    Let $G$ be a group with a regular $(v, k, \lambda, \mu)$-PDS $D$ with nonprincipal eigenvalues $\theta_1 > \theta_2$, let $H$ be a subgroup of linear characters of $G$ with order coprime to $\sqrt{\Delta} = \theta_1 - \theta_2$, let $N$ be the intersection of the kernels of the characters in $H$, and let $\pi_\alpha, \pi_\beta$ be defined as in Definition \ref{def:pi}.  For $a \not \in N$, the greatest common divisor of $|C_G(a)|$ and at least one of $\pi_\alpha, \pi_\beta$ is $1$.
\end{lem}
	\begin{proof}
	By Theorem \ref{value-phi(a)}, $|C_G(a)|$ divides $k-\xi(D)$. Without loss of generality, we may suppose $\xi(D) = \theta_\alpha$. Then, by Lemma \ref{coprime-condition}, $\gcd(|C_G(a)|, \pi_\beta) = 1$, for otherwise there exists some $p\in \Pi_\beta$ dividing $|C_G(a)|$, and so by Theorem \ref{value-phi(a)}, $p$ divides $k-\theta_\alpha = |C_G(a)||a^G\cap D|$. Therefore, $p$ divides both $k-\theta_\alpha$ and $k-\theta_\beta$, a contradiction.
	\end{proof}

	Once again, we invite the reader to compare the following Theorem \ref{mod-restriction} to \cite[Theorem 2.8]{ott}.

\begin{thm}\label{mod-restriction}
    Let $G$ be a group with a regular $(v, k, \lambda, \mu)$-PDS $D$ with nonprincipal eigenvalues $\theta_1 > \theta_2$, let $\mathcal{L}$ be the group of linear characters of $G$, and let $\Pi_\alpha, \Pi_\beta, \pi_\alpha, \pi_\beta$ be defined as in Definition \ref{def:pi}.  Suppose that there exists a prime $p$ dividing $|G|$ so that $p\not|\;\sqrt\Delta = \theta_1 - \theta_2$. Without loss of generality, let $p\in\Pi_\alpha$ and suppose that $p$ divides $|\mathcal{L}|$. Finally, let $P$ be a Sylow $p$-subgroup of $G$. Then $\gcd(|N_G(P)|, \pi_\beta) = 1$. Moreover, if $G$ is solvable, then
	\[ \pi_\beta \equiv 1\pmod p. \]
\end{thm}
	\begin{proof}
		Since $p$ is coprime to $\sqrt\Delta$, and $p$ divides $|\mathcal{L}|$, we have that $H$ is nontrivial. Thus, $N \neq G$, so the hypotheses for Lemma \ref{coprime-centralizer} apply. Suppose there exists $q$ dividing $|N_G(P)|$ such that $q\in\Pi_\beta$. Then, there is an element $g\in N_G(P)$ of order $q$. By \cite[Theorem 3.5]{gorenstein}, $P = C_P(g)[\langle g\rangle, P]$. Each element in $C_P(g)$ centralizes $g$, so each element $a\in C_P(g)$ contains $g$ in its centralizer. Thus, for some $m$, $\gcd(|C_G(a)|, \pi_\beta) = q\cdot m$, and since $a \in C_P(g)$ and $a \in C_G(a)$, then the order of $a$ divides both $|P|$ and $|C_G(a)|$. Thus, $\gcd(|C_G(a)|, |P|) > 1$, forcing $\gcd(|C_G(a)|, \pi_\alpha) > 1$. Therefore, by the contrapositive of Lemma \ref{coprime-centralizer}, $a$ is in $N$. Thus, $P = C_P(g)[\langle g\rangle, P] \subseteq N G' \subseteq N$ (where $G' = [G,G]$), given that $G' \subseteq N$, since for each $\xi \in H$, we must have $G' \subseteq \ker\xi$. So, $P \subseteq N$. But, $p$ divides $|\mathcal{L}|$, a contradiction. Therefore, $\gcd(|N_G(P)|, \pi_\beta) = 1$.

	Now, suppose in addition that $G$ is a solvable group. By Hall's Theorem, there is a subgroup $K$ of $G$ which is a $\pi_\alpha\pi_\beta$-group. Moreover, since $G$ is solvable, $K$ is solvable, so there exists a subgroup $U$ of $K$ which is a $\pi_\alpha$-group. Since $p \in \Pi_\alpha$, then $|P|$ divides $\pi_\alpha$. Therefore, $P$ is contained in some Hall subgroup, $L$. Any two Hall subgroups are conjugate, and thus there exists $g \in G$ such that $g^{-1}Lg = K$. So, $K$ contains a Sylow $p$-subgroup, $g^{-1}Pg$. Therefore, we are justified in assuming that $P \subseteq K$, forcing $P \subseteq U$. Thus by our preceding argument, $N_K(P) = N_K(U)$. Therefore, using Sylow III,
		\begin{equation*}
		1\equiv \frac{|K|}{|N_K(P)|}\equiv \frac{\pi_\alpha \pi_\beta}{|N_K(P)|} \equiv \frac{\pi_\alpha}{|N_K(P)|}\pi_\beta\equiv \frac{|U|}{|N_U(P)|}\pi_\beta \equiv \pi_\beta \pmod p.
		\end{equation*}
	\end{proof}

\subsection{Intersections of $D$ with certain normal subgroups}
\label{subsec:normal}

Let $G' = [G,G]$ be the derived subgroup of $G$, and let $\calL$ be the group of linear characters of $G$; note that $\calL \cong G/G'$. We will define $H$ to be a subgroup of linear characters with order coprime to $\sqrt{\Delta}$ and $N$ to be the intersection of the kernels of the characters in $H$. In the case when $\gcd(v, \sqrt{\Delta}) = 1$, we see that $|\calL|$ is coprime to $\sqrt{\Delta}$, so we can choose $H$ to be $\calL$, in which case $N = G'$. We note first the following result.

\begin{prop}\label{prop:HsameonD}
        Let $G$ be a group with a regular $(v, k, \lambda, \mu)$-PDS $D$ with nonprincipal eigenvalues $\theta_1 > \theta_2$, let $H$ be a subgroup of linear characters of $G$ with order coprime to $\sqrt{\Delta} = \theta_1 - \theta_2$, and let $N$ be the intersection of the kernels of the characters in $H$.  If $|H| \neq \{1\}$ and $\xi(D) = \theta_\alpha$ for some nonprincipal linear character $\xi$ in $H$, then $\xi(D) = \theta_\alpha$ for all nonprincipal $\xi \in H$. Moreover, for all $a \in G \backslash N$,
		 \[ \Phi(a) = |C_G(a)| |a^G \cap D| = k - \theta_\alpha.\]
\end{prop}
        \begin{proof}
                This follows immediately from the above discussion, Corollary \ref{coprime-linear-char}, and Theorem \ref{value-phi(a)}.
        \end{proof}

Going forward, we will assume that $N$ is a proper subgroup of $G$ (so that $H$ exists), $\xi(D) = \theta_\alpha$ for all nonprincipal linear characters $\xi$ of $H$, and we will denote the eigenvalue that is not $k$ or $\theta_\alpha$ by $\theta_\beta$. While Proposition \ref{prop:HsameonD} is quite strong, we will in fact be able to calculate $|N \cap D|$ exactly.

\begin{thm}\label{thm:G'capD}
        Let $G$ be a group with a regular $(v, k, \lambda, \mu)$-PDS $D$ with nonprincipal eigenvalues $\theta_1 > \theta_2$, let $H$ be a subgroup of linear characters of $G$ with order coprime to $\sqrt{\Delta} = \theta_1 - \theta_2$, and let $N$ be the intersection of the kernels of the characters in $H$.  If $N < G$, then
		\[ |N \cap D| = \frac{k - \theta_\alpha}{|H|} + \theta_\alpha,\]
		where $\xi(D) = \theta_\alpha$ for all nonprincipal $\xi \in H$.
\end{thm}
        \begin{proof}
 		First, we apply Lemma \ref{lem:sumPhi}(iii) to $N$ to obtain
 			\[ \sum_{x \in N} \Phi(x) = |G| \cdot |N \cap D|.\]
 		Next, using Proposition \ref{prop:HsameonD} and recalling that $|G| = |N|\cdot|H|$, we have
 			\begin{align*}
				 \sum_{x \in N} \Phi(x) &= \sum_{g \in G} \Phi(g) - \sum_{g \notin N} \Phi(g)\\
                        		 &= |G|k - \left(|G| - |N|\right)(k - \theta_\alpha)\\
                        		 &= |N|(k - \theta_\alpha) + |G|\theta_\alpha \\
 			 \end{align*}
		Therefore,
			\begin{align*}
 				|N \cap D| &= \frac{1}{|G|}\left( |N|(k - \theta_\alpha) + |G|\theta_\alpha) \right)\\
              			&= \frac{k - \theta_\alpha}{|H|} + \theta_\alpha.
			\end{align*}
	\end{proof}

Next, we are able to put even more severe restrictions on the intersection of the other cosets of $N$ with $D$.

\begin{thm}\label{thm:Ncosetintersection}
        Let $G$ be a group with a regular $(v, k, \lambda, \mu)$-PDS $D$ with nonprincipal eigenvalues $\theta_1 > \theta_2$, let $H$ be a subgroup of linear characters of $G$ with order coprime to $\sqrt{\Delta} = \theta_1 - \theta_2$, and let $N$ be the intersection of the kernels of the characters in $H$.  If $N < G$, and $a \in G \backslash N$, then
                       \[ |Na \cap D| = \frac{k - \theta_\alpha}{|H|},\]
        where $\xi(D) = \theta_\alpha$ for all nonprincipal $\xi \in H$.
\end{thm}
        \begin{proof}
                Note by assumption that $G' \subseteq N$, since $G'$ is the intersection of the kernels of all linear characters. Moreover, if $b \in a^G$, then 
                \[ b = a^g = g^{-1}ag = (g^{-1}aga^{-1})a = [g, a^{-1}]a \in G'a,\]
                and so $a^G \subseteq G'a \subseteq Na$, and so $Na$ will be a disjoint union of conjugacy classes. Let $a_1^G, \dots, a_s^G$ be the distinct conjugacy classes in $Na$. Then, using Proposition \ref{prop:HsameonD}, we have
                \begin{align*}
                    |Na \cap D| &= \sum_{i = 1}^s |a_i^G \cap D|\\
                                &= \sum_{i = 1}^s \frac{k - \theta_\alpha}{|C_G(a_i)|}\\
                                &= \frac{k - \theta_\alpha}{|G|} \sum_{i = 1}^s |a_i^G|\\
                                &= \frac{(k - \theta_\alpha)|N|}{|G|}\\
                                &= \frac{k - \theta_\alpha}{|H|}.
                \end{align*}
        \end{proof}

When $\gcd(v, \sqrt{\Delta}) = 1$, it follows that $H = \calL$ and $N = G'$, so we get the following corollary, which helps to explain why the sizes of conjugacy class intersections with a regular PDS take on very few values in practice in these situations (see Appendix \ref{app:C}).

\begin{cor}
    \label{cor:LsameonD}
    Assume the group $\mathcal{L}$ of linear characters of $G$ is nontrivial.  If $G$ contains a regular $(v,k, \lambda, \mu)$-PDS, $\gcd(v, \sqrt{\Delta}) = 1$, and $\xi(D) = \theta_\alpha$ for a nontrivial $\xi \in \mathcal{L}$, then
    \[ |G' \cap D| = \frac{k - \theta_\alpha}{|\calL|} + \theta_\alpha\]
    and, for all $a \notin G'$,
    \[ |G'a \cap D| = \frac{k - \theta_\alpha}{|\calL|}.\]
\end{cor}

\subsection{Example applications}
\label{subsec:applications}

We now demonstrate how these generalizations of Ott's work can be applied in other settings.

\begin{prop}\label{odd-order-GQ}
There is no group acting regularly on the points of a GQ with order $(q,q^2-q)$ when $q$ is even.
\end{prop}
	\begin{proof}
	Let $q$ be even, and assume that a group $G$ acts regularly on the point set of a GQ with order $(q, q^2 - q)$. This implies the existence of a $((q+1)(q^2-q+1), q(q^2-q+1), q-1, q^2-q+1)$-PDS; in this case, the eigenvalues $\theta_1, \theta_2$ are $q-1$, $-q^2+q-1$, and $\sqrt\Delta=q^2$. Note that if $p|\sqrt\Delta$, then $p\not|\; v = (q+1)(q^2-q+1)$, so we may apply Theorem \ref{mod-restriction}. Thus, we let $p$ divide $v$. Now, $k-\theta_1 = q^3-q^2+1$ and $k-\theta_2 = (q+1)(q^2-q+1)$. We show that $v_1 = k-\theta_1$ and $v_2 = \frac{k-\theta_2}{\mu}$. So, suppose $p|\mu$ and $p|k-\theta_1$. Then mod $p$,
		\begin{align*}
		0 &\equiv q^3-q^2+1\\
		  &\equiv q(q^2-q+1) - q + 1\\
		  &\equiv 1-q \\
		  &\equiv 1-q - (q^2-q+1)\\
		  &\equiv q^2
		\end{align*}
	 So $p|q$ and $p|1-q$, forcing $p=1$, so $\mu$ cannot divide $k-\theta_1$. Therefore, $v_1 = q^3-q^2+1$ and $v_2 = q+1$.
	 
	 Note that $v$ is odd, so if $G$ acts regularly on the points of this GQ, then $G$ is solvable by Feit-Thompson. Therefore, if there is some linear character of order $p$, then either $v_1 \equiv 0$ and $v_2\equiv 1\pmod p$ or $v_1\equiv 1$ and $v_2\equiv 0\pmod p$ by Theorem \ref{mod-restriction}. If $v_1 \equiv 0$ and $v_2 \equiv 1$, then $q \equiv 0$, forcing $v_1 \equiv 1$, and thus $p=1$. If $v_1 \equiv 1$ and $v_2 \equiv 0$, then $q \equiv -1$, so $v_1 \equiv (-1)^3 - (-1)^2 + 1 \equiv -1$, forcing $-1 \equiv 1$ and thus $p = 2$. Therefore, $2 | q+1$, and so $q$ is odd, a contradiction. Therefore, no such $G$ exists.
	 \end{proof}
	 
\begin{prop}\label{DS-restriction}
If there exists a $(4w^2 - 1, 2w^2, w^2, w^2)$-PDS $D$ in a group $G$, then $w \equiv 1 \pmod 3$, all linear characters of $G$ evaluate to $-w$ on $D$, every conjugacy class intersection with $D$ is a multiple of $w$, and the order of the group of linear characters is a power of $3$.
\end{prop}
	\begin{proof}
	Note that $4w^2-1$ is odd, so $G$ is a solvable group. Moreover, $\sqrt{\Delta} = 2w$, $\theta_1 = w$, and $\theta_2 = -w$, so $v_1 = 2w-1$ and $v_2 = 2w+1$. Since $G$ is solvable, it has a nontrivial group of linear characters, so let $p$ be a prime dividing the order of this group. Then, by Theorem \ref{mod-restriction}, either $v_1 \equiv 1 \pmod p$ and $v_2 \equiv 0 \pmod p$, or $v_1 \equiv 0 \pmod p$ and $v_2 \equiv 1 \pmod p$. After some algebraic manipulation, it is clear that $v_1 \equiv 0 \pmod p$ and $v_2 \equiv 1 \pmod p$ implies $p=1$, and that $v_1 \equiv 1 \pmod p$ and $v_2 \equiv 0 \pmod p$ implies $p=3$. Therefore, $3$ divides $2w+1$, implying $w \equiv 1 \pmod 3$. Additionally, by Theorem \ref{mod-restriction}, the linear group of characters must be a power of $3$.
	
	Now, let $a \in G$ be an element of order divisible by $3$ such that $a \not\in [G,G]$. Such an element must exist, for $3$ divides $|\mathcal{L}|$. In addition, let $\xi \in \mathcal{L}$. Then, by Theorem \ref{value-phi(a)}, $|C_G(a)||a^G\cap D| = k - \xi(D)$. By Lemma \ref{eigenvalues}, $\xi(D) = \pm w$. Set $w = 3m+1$ for some $m \in \Z$. Therefore, 
		\begin{equation*}
		|C_G(a)| = \frac{k - (\pm w)}{|a^G\cap D|} = \frac{2w^2 - (\pm w)}{|a^G\cap D|} = w\cdot\frac{2w - (\pm 1)}{|a^G\cap D|} = w\cdot \frac{6m + 2 - (\pm 1)}{|a^G\cap D|}.
		\end{equation*}
	Since $a$ has order divisible by $3$, then $3$ divides $|C_G(a)|$, so $6m + 2 - (\pm 1) \equiv 0 \pmod 3$. Since $w\neq 0\pmod 3$, then $2 - (\pm 1) \equiv 0 \pmod 3$. It follows that $\xi(D) = -w$, so that $2 - (\pm 1)$ equals $3$.
	\end{proof}


\section{Computational methods}
\label{sect:comp}

In the previous section we established certain theoretical results on the characters of groups containing a partial difference set. We now develop related techniques to construct partial difference sets inside of a group using the group characters. In particular, for a partial difference set $D$ inside of a group $G$ with $p^\ell$ dividing $\sqrt\Delta$, we develop techniques to give the value of $|D\cap h^G| \pmod{p^\ell}$ for particular primes $p$, and calculate $|D\cap h^G|$ by extension.

These methods, combined with Theorem \ref{value-phi(a)}, makes it rather quick to search for groups which could contain a PDS just by utilizing the irreducible characters of the group. In some cases, simply knowing the values $|D\cap h^G| \pmod{p^\ell}$ can rule out the existence of a PDS for a certain parameter sets inside of a nonabelian group, such as in Example \ref{No-DS}. In other cases, targeted information about $|D\cap h^G|$ can be used to quickly construct partial difference sets even inside of groups with up to $1000$ elements, using the search method described in \cite{brady}.

\subsection{Modular computation of conjugacy class intersection sizes for PDSs}

Now, as before, let $G$ be a group containing a $(v, k, \lambda, \mu)$-PDS $D$ with eigenvalues $\theta_1 = \frac12(\lambda-\mu +\sqrt{\Delta})$ and $\theta_2 = \frac12(\lambda-\mu - \sqrt{\Delta})$. Let $\repF$ be an irreducible $G$ representation, and let $\chi\in\Irr(G)$ be afforded by $\repF$. Denote the eigenspace of $\repF(D)$ corresponding to $\theta_i$ by $V_\repF (\theta_i)$. For $p$ dividing $\sqrt\Delta$, construct the corresponding ring $\frakR$ given by Definition \ref{def-R}. Let $\ell$ be maximal such that $p^\ell$ divides $\sqrt\Delta$, and let $(\pi) = \p$ be the unique maximal ideal containing $p$. Then, $p^\ell \in (\pi^\ell)$, so set $R = \frakR/(\pi^\ell)$. Then, we have the following. (Compare to \cite[Lemma 3.1]{swartz-tauscheck}.)

\begin{lem}\label{char-p-phi}
Using the notation above, for $h \in G$ such that $h \neq 1$, suppose that $p^\ell$ does not divide $|C_G(h)|$. Then, over the ring $R$,
	\begin{equation*}
		|h^G \cap D||C_G(h)| \equiv k - \theta_2 \pmod {(\pi^\ell)},
	\end{equation*}
and if $h = 1$, then \[|h^G \cap D||C_G(h)| \equiv k + \theta_2(|G| - 1) \pmod {(\pi^\ell)}.\]
In particular, 
\begin{equation*}
		|h^G \cap D||C_G(h)| \equiv k - \theta_2 \pmod {p^\ell},
	\end{equation*}
and if $h = 1$, then \[|h^G \cap D||C_G(h)| \equiv k + \theta_2(|G| - 1) \pmod {p^\ell}.\]
\end{lem}
	\begin{proof}
	First, for a character $\chi$ afforded by a representation $\repF$, $\chi(1) = \dim V_\repF(\theta_1) + \dim V_\repF(\theta_2)$. Moreover, for $p^\ell$ a prime power dividing $\sqrt\Delta$, $\theta_1 \equiv \theta_2 \pmod{p^\ell}$. Thus, $\overline{\chi(D)} \equiv \dim V_\repF(\theta_1)\theta_1 + \dim V_\repF(\theta_2)\theta_2 \equiv (\dim V_\repF(\theta_1) + \dim V_\repF(\theta_2))\theta_2 \equiv \chi(1)\theta_2 \pmod{p^\ell}$. It follows that 
	
		\[\Phi \equiv k - \theta_2 + \sum_{\chi\in\Irr(G)}(\chi(1)\theta_2)\chi \pmod {(\pi^\ell)}. \]
	
	Then, when $h \neq 1$,
	
		\[\Phi(h) \equiv k - \theta_2 + \theta_2\sum_{\chi\in\Irr(G)}\chi(1)\chi(h) \equiv k - \theta_2 \pmod {(\pi^\ell)}\]
		
	by column orthogonality of the characters, and when $h = 1$, then
	
		\[\Phi(1) \equiv k - \theta_2 + \theta_2\sum_{\chi\in\Irr(G)}\chi(1)^2 \equiv k-\theta_2 + \theta_2|G| \equiv k + \theta_2(|G| - 1) \pmod {(\pi^\ell)}.\]
	
	The full equality is given by the definition $\Phi(h) = |h^G \cap D||C_G(h)|$. Finally, the last equations follow since $|h^G \cap D||C_G(h)|$ is an integer. 
	\end{proof}

Occasionally, Lemma \ref{char-p-phi} can be used to establish whether or not a PDS exists on a certain parameter set more or less by hand, occasionally aided by minor computer calculations.

\begin{ex}\label{No-DS}
There does not exist a reversible $(112, 75, 50)$-DS in a nonabelian group $G$ of order $112$. We suppose for contradiction that there does exist such a DS in a nonabelian group $G$ of order $112$. This parameter set has eigenvalues $\theta_1=5$ and $\theta_2=-5$. By Lemma \ref{eigenvalues}, if $D$ exists then $\overline{\chi(D)} \equiv 0 \pmod 5$. Therefore, by Lemma \ref{char-p-phi}, $\Phi(g) \equiv 75 - 5 + \sum_{\chi\in\Irr(G)}(\chi(1)\theta_2)\chi(g) \equiv 0 \pmod 5$. Finally, there is no centralizer $C_G(g)$ of order divisible by $5$, since $5$ does not divide $112$. Therefore, for all $g\in G$, $|g^G \cap D| \equiv 0 \pmod 5$.

	Consulting \GAP \cite{GAP4} on all possible nonabelian groups of order $112$, we may for a given group $G$ compute the largest multiple of $5$ less than or equal to the size of each conjugacy class of $G$. For example, if $G = $ SmallGroup(112, 41), then it has conjugacy class sizes $[ 1, 1, 8, 7, 8, 7, 8, 8, 8, 8, 8, 8, 8, 8, 8, 8 ]$, and the largest multiple of 5 for each conjugacy class is $[ 0, 0, 5, 5, 5, 5, 5, 5, 5, 5, 5, 5, 5, 5, 5, 5]$. In this case, $G$ cannot contain a reversible difference set, for Sum($[ 0, 0, 5, 5, 5, 5, 5, 5, 5, 5, 5, 5, 5, 5, 5, 5]) = 70 < k = 75$.
	
	An analogous computation for each group verifies that there are no nonabelian groups of order $112$ containing a regular $(112, 75, 50)$-DS. A direct verification may be done in \GAP \cite{GAP4}, using the following command:

		\begin{verbatim}
			gps := Filtered(AllSmallGroups(112), g -> not IsAbelian(g));;
			List(gps, g -> 5*Sum(List(ConjugacyClasses(g), i -> Int(Size(i)/5))));
		\end{verbatim}
	The output will be a list, and every entry of the list corresponds to the largest sum of elements in a possible PDS for a given group. Note that for every group, that sum is strictly less than $k = 75$.
	\end{ex}

In fact, we may upgrade this sort of technique to include PDSs more generally. For this, we first require the following Lemma.

\begin{lem}\label{k-mod-p}
Let $G$ contain a regular $(v, k, \lambda, \mu)$-PDS with eigenvalues $\theta_1 > \theta_2$ such that $p^\ell$ divides both $\theta_1$ and $\theta_2$, then $k \equiv 0 \pmod {p^\ell}$. In particular, $k \equiv 0 \pmod {\gcd(\theta_1, \theta_2)}$.
\end{lem}
	\begin{proof}
	If $D$ is a regular PDS, then $\Phi(1) = 0$. Let $p^\ell$ divide $\theta_1, \theta_2$. Let $\frakR$ be the local ring corresponding to $p$ as before. Evaluating mod $(\pi^\ell)$, we have
	\[0 \equiv \Phi(1) \equiv k - \theta_2 + \theta_2\sum_{\chi\in\Irr(G)}\chi(1)\chi(1) \equiv k \pmod {(\pi^\ell)},\] by Lemma \ref{char-p-phi}. Since $k$ is an integer, $k \equiv 0 \pmod {p^\ell}$.
	\end{proof}

\begin{lem}\label{shared-prime-class-intersection}
Let $G$ be a group containing a regular $(v, k, \lambda, \mu)$-PDS $D$ with eigenvalues $\theta_1 > \theta_2$ such that $\gcd(\theta_1, \theta_2) \neq 1$. If $p^\ell$ divides $\gcd(\theta_1, \theta_2)$ and $p^\ell$ does not divide $v$, then for each conjugacy class $h^G$, \[|h^G \cap D| \equiv 0 \pmod{p^\ell}.\]
\end{lem}
	\begin{proof}
	By Lemma \ref{char-p-phi}, \[|C_G(h)||h^G\cap D| \equiv k - \theta_2 + \theta_2\sum_{\chi\in\Irr(G)}\chi(1)\chi(h) \equiv 0 \pmod {(\pi^\ell)},\] and by Lemma \ref{k-mod-p}, $k \equiv 0 \pmod{p^\ell}$. Since $p^\ell$ does not divide $v$, then $p$ does not divide $|C_G(h)|$, forcing $|h^G\cap D| \equiv 0 \pmod{p^\ell}$.
	\end{proof}
	
On certain parameter sets and for certain groups, Lemma \ref{shared-prime-class-intersection} in conjunction with Theorem \ref{value-phi(a)} allow one to compute very rapidly the structure of the class intersections of $G$ and a PDS $D$. We place here the following example to illustrate this point.

\begin{ex}
	We explore the structure of a $(183, 112, 66, 72)$-PDS $D$ in a group $G$ of order $183$. The eigenvalues of this PDS are $\theta_1 = 4$ and $\theta_2 = -10$. Then, $\gcd(4, -10) = 2$, but $2$ does not divide $183$. Therefore, by Lemma \ref{shared-prime-class-intersection}, for every $h \in G$, we have $|h^G \cap D| \equiv 0 \pmod 2$. By a computation in \GAP \cite{GAP4}, there is one nonabelian group given by SmallGroup(183, 1) with two conjugacy classes $h_1^G, h_2^G$ of order $61$, $20$ conjugacy classes of order $3$, and one conjugacy class of order $1$. If $x \in h_1^G, h_2^G$, then $x \not \in G'$ by a computation in \GAP \cite{GAP4}. Moreover, every prime $p$ dividing $183$ does not divide $14$, and $G$ must have a subgroup of linear characters, for it is solvable by Feit-Thompson. So, by Theorem \ref{value-phi(a)}, $|h_1^G \cap D|, |h_2^G \cap D| \in \{\frac{112-4}{3}, \frac{112+10}{3}\}$. Only $\frac{112-4}{3} = 36$ is integral, so $|h_1^G \cap D|, |h_2^G \cap D| = 36$. Therefore, $112-72 = 40$ elements must be partitioned among $20$ conjugacy classes. The maximum size of each conjugacy class is $3$, and each class must have $0 \pmod 2$ elements, so each of the remaining conjugacy classes contain $2$ elements of $D$.
\end{ex}

A PDS does indeed exist for this parameter set, which we found through a modified form of the Random Hill Climb search discussed in \cite{brady}. We will speak on this more extensively in the next section.

\subsection{An algorithmic application of Lemma \ref{char-p-phi}}

Combining Lemma \ref{char-p-phi} and Theorem \ref{value-phi(a)} lets us build a fast algorithm to solve for all values $|D\cap h^G|$ even in comparatively large groups of order up to $1000$. We briefly describe the concept of such an algorithm, and then provide the pseudocode for an implementation.

Suppose $G$ is a fixed nonabelian group with conjugacy class representatives $h_1, \dots, h_r$. We would like to search for a $(v, k, \lambda, \mu)$-PDS $D$ in $G$ with discriminant $\sqrt\Delta$. So, we suppose $D$ exists, and attempt to find the values $|h_i^G \cap D|$. If there exists a prime $p$ dividing $|\mathcal{L}|$ and not $\sqrt\Delta$, then for $h \in \{h_1, \dots, h_r\}$ such that $h \not \in G'$, we may directly compute $|h^G \cap D|$ as one of $\frac{k-\theta_1}{|C_G(h)|}$ or $\frac{k-\theta_2}{|C_G(h)|}$ by Theorem \ref{value-phi(a)}. Otherwise, if no such $p$ exists or $h \in G'$, we list the prime powers $p_1^{\ell_1}, \dots, p_m^{\ell_m}$ dividing $\sqrt\Delta$ and not dividing $|C_G(h)|$. If $h \neq 1$, we compute $x \equiv |C_G(h)|^{-1}(k - \theta_2) \pmod{p_i^{\ell_i}}$ for each $p_i^{\ell_i}$. By Lemma \ref{char-p-phi}, $x\equiv |h^G \cap D| \pmod{p_1^{\ell_1}\dots p_m^{\ell_m}}$.

We then generate every list of values $d_1, \dots, d_r$ so that $d_1 + \dots + d_r = k$, and either $d_i$ equals one of $\frac{k-\theta_1}{|C_G(h)|}$ or $\frac{k-\theta_2}{|C_G(h)|}$ if $h_i$ satisfies the necessary criteria, or $d_i \equiv |C_G(h)|^{-1}(k - \theta_2) \pmod{p_i^{\ell_i}}$ for each of the $p_i^{\ell_i}$ corresponding to that $h_i$. We end by checking if each list $d_1, \dots, d_r$ gives a valid $\Phi$ function according to the identification $\overline{\chi_j(D)} = d_1\overline{\chi_j(h_1)} + \dots d_r\overline{\chi_j(h_r)}$ for all $\chi_j \in \Irr(G)$.

We should mention as well the following fact, which we use as part of Algorithm \ref{CCI} to determine that certain lists of conjugacy class intersection sizes are invalid and is often successful in ruling out PDSs inside of even order groups.

\begin{lem}\label{order-2-action}
Let $G$ be a group containing a regular PDS $D$. If $h \neq h^{-1}$ and $h^{-1} \in h^G$, then $|h^G \cap D|$ is even.
\end{lem}
	\begin{proof}
		Suppose for $h \in G$ that $h \neq h^{-1}$ but $h^{-1} \in h^G$. Since $D = D^{(-1)}$, we can partition $h^G \cap D$ into inverse pairs, and so $h^G \cap D$ must be even.
	\end{proof}
	
It is very easy to rule out a large number of even order groups using Lemma \ref{order-2-action}.

\begin{ex}
We show that there does not exist a $(100, 33, 14, 9)$-PDS inside of $G =$ SmallGroup(100, 3). By a computation in \GAP \cite{GAP4}, $G$ contains a conjugacy class $h^G$ of order $4$ such that $g, g^{-1} \in h^G$ for all $g\in h^G$. Moreover, the centralizer $C_G(h)$ has order $25$. Since the eigenvalues corresponding to $(100, 33, 14, 9)$ are $8$ and $-3$, by Lemma \ref{char-p-phi},
\[|h^G \cap D| \equiv (33 - 8)25^{-1} \equiv 1 \pmod{11}.\] The only nonnegative integer less than or equal to $4$ that is congruent to $1 \pmod{11}$ is $1$, so $|h^G \cap D| = 1$. Then, by Lemma \ref{order-2-action}, $G$ cannot contain a $(100, 33, 14, 9)$-PDS.
\end{ex}
    
 \begin{algorithm} 
    \caption{Construct Class Intersections}
    \label{CCI}
	\begin{algorithmic}[1]
		\Ensure $\theta_1, \theta_2$ are eigenvalues corresponding to $D$ so that $\theta_1 \neq k \neq \theta_2$
		\Ensure $h_1, \dots, h_r$ are conjugacy class representatives of $G$.

		\State $\sqrt\Delta \gets \theta_1 - \theta_2$
		\State DModIntersectionList $\gets$ empty length $r$ list
		\State phiLists $\gets$ empty list

		\For{$i = 1$ to $r$}
			\State $h \gets h_i$
			\State primeDivsDelta $\gets$ prime power divisors of $\sqrt\Delta$ coprime to $|C_G(h)|$
			\State DHModIntersection $\gets N$ such that $N \equiv |C_G(h)|^{-1}(k - \theta_2) \pmod{p^\ell}$ for each $p^\ell$ in primeDivsDelta
			
			\If{$h \not \in G'$}
				\State $C_1 \gets \frac{k-\theta_1}{|C_G(h)|}$
				\State $C_2 \gets \frac{k-\theta_2}{|C_G(h)|}$
	
				\For{$p^\ell$ in primeDivsDelta}
					\If{DHModIntersection $\pmod{p^\ell} \not \in \{C_1 \pmod{p^\ell}, C_2 \pmod{p^\ell}\}$}
						\State Terminate
					\EndIf
					
				\EndFor
		\algstore{CCI} 
	\end{algorithmic}
\end{algorithm}
\begin{algorithm}
	\begin{algorithmic} [1]
		\algrestore{CCI}
				\If{$C_1$ is not integral and $C_2$ is not integral}
						\State Terminate
				\ElsIf{$C_1$ is integral and $C_2$ is not integral}
					\State DModIntersectionList[i] $\gets [C_1, 0]$
				\ElsIf{$C_1$ is not integral and $C_2$ is integral}
					\State DModIntersectionList[i] $\gets [C_2, 0]$

				\Else
					\State DModIntersectionList[i] $\gets [C_2, \frac{\sqrt\Delta}{|C_G(h)|}]$
				\EndIf

				\Else
					\State DModIntersectionList[i] $\gets [N, \text{Product(primeDivsDelta)}]$
				\EndIf
			\EndFor

			\State Append every list of integers $d_1, \dots, d_r$ such that $d_i \equiv \text{DModIntersectionList[i][1]} \pmod{\text{DModIntersectionList[i][2]}}$ and Sum$(d_1, \dots, d_r) = k$ to phiLists
			\State Filter all $d_1, \dots, d_r$ from phiLists for which the corresponding $\Phi$ is invalid.
			\State Output phiLists
		\end{algorithmic}
\end{algorithm}

\FloatBarrier

The phrase ``Filter all $d_1, \dots, d_r$ from phiLists" is rather expansive, so we make the following comment. We first check that each $\chi(D)$ generated by the $d_1, \dots, d_r$ is a valid sum of eigenvalues of the corresponding possible PDS. We also verify the result against Lemma \ref{order-2-action} and against Theorem \ref{thm:Ncosetintersection}. We should note as well that the full implementation of Algorithm \ref{CCI} may also be found \href{https://github.com/srnelson1/PDS-class-intersections}{here}. Now, Algorithm \ref{CCI} runs quickly on most groups for which there is some prime dividing $\sqrt\Delta$ not dividing $v$, and it terminates especially rapidly in cases where the number of primes or the size of primes are relatively large. In fact, we were able to compute all values $|D\cap h^G|$ of all parameter sets for which $\gcd(v, \sqrt\Delta) = 1$ listed in Brouwer's tables \cite{brouwertables}. We have combined this algorithm with the Random Restart Hill Climb search given by Brady \cite{brady} as well as the linear optimization software Gurobi \cite{gurobi} to construct a number of partial difference sets.

\begin{rem}
 \label{rem:111}
 The existence of a set of conjugacy class intersection sizes that satisfies the modular constraints of Lemmas \ref{char-p-phi} and \ref{order-2-action} is not sufficient to guarantee the existence of a PDS. For example, Algorithm \ref{CCI} returns the list \[[ 0, 9, 1, 9, 1, 1, 1, 1, 1, 1, 1, 1, 1, 1, 1 ]\]
 as a potentially valid list of conjugacy class intersection sizes for a $(111,30,5,9)$-PDS in SmallGroup(111, 1). Note that a subset of this group with these intersection sizes will evaluate on irreducible characters to values admissible by Lemma \ref{eigenvalues}.  However, by turning the conjugacy class and PDS constraints into a linear program, we are able to rule out the existence of such a PDS using the linear optimization software Gurobi \cite{gurobi}, and hence there does not exist a $(111, 30 , 5, 9)$-PDS. 
 
\end{rem}

\subsection{PDSs constructed using Algorithm \ref{CCI}}

For a number of groups, we used this method to calculate all possible sizes of intersections of the conjugacy classes of the group with a possible PDS inside this group. We then ran a restricted search for PDSs satisfying the class intersection restrictions using a modified Random Restart Hill Climb search described in \cite{brady}. Using the technique from \cite{brady}, we were able to construct nineteen distinct PDSs in nonabelian groups, many of which were already explicitly recognized, such as PDSs arising from triangular graphs. We found three parameter sets, $(111, 44, 19, 16)$, $(305, 76, 27, 16)$, and $(981,140,43,16)$, which correspond to an infinite family of groups acting regularly on Steiner systems $\S(2, 4, p^d)$ for $p$ prime. Five others that we were able to construct computationally -- in SmallGroup(57, 1), SmallGroup(155, 1), SmallGroup(301, 1), SmallGroup(737, 1), and SmallGroup(1027, 1) -- have parameters corresponding to an infinite family of PDSs corresponding to Steiner triple systems. The infinite families on $\S(2, 4, p)$ and $\S(2, 3, p^d)$ correspond to block-regular Steiner $2$-design constructions given by Clapham in \cite{clapham} and Wilson in \cite{wilson}. Finally, we found a $(183,70,29,25)$-PDS in SmallGroup(183, 1) which is related to a $\S(2,5,61)$ Steiner system and, as it turns out, is also part of an infinite family. To the authors' knowledge, the PDSs arising from \cite{wilson} have not previously been recognized in the literature, and we are only aware of one paper \cite{PonomarenkoRyabov_2025} recognizing the constructions given by Clapham \cite{clapham} as giving rise to PDSs.

\begin{prop}\label{Clapham-PDS-family}
Let $p$ be a prime such that $p^d > 9$ and $p^d \equiv 7 \pmod{12}$. Then, there exists a $(p^d(p^d-1)/6, 3(p^d-3)/2, (p^d + 3)/2, 9)$-PDS in the nonabelian group $C_{p^d} \rtimes C_{(p^d-1)/6} \le (\GF(p^d), +) \rtimes (\GF(p^d)^\times, \cdot)$.
\end{prop}

\begin{proof}
 By the results of Clapham \cite[Section 4]{clapham}, if $p^d \equiv 7 \pmod {12}$, there exists a Steiner triple system on $p^d$ elements with block-regular automorphism group isomorphic to $C_{p^d} \rtimes C_{(p^d-1)/6} \le (\GF(p^d), +) \rtimes (\GF(p^d)^\times, \cdot)$. By the discussion in Subsection \ref{subsec:famSRG}, this will correspond to a PDS in the nonabelian group $C_{p^d} \rtimes C_{(p^d-1)/6}$.
\end{proof}

\begin{rem}
 Clapham enumerated the number of isomorphism classes of such block-regular Steiner triple systems for each prime power $p^d \equiv 7 \pmod {12}$ (see \cite[Theorem 4.7]{clapham}). It is well known that two Steiner triple systems are isomorphic if and only if their block-intersection graphs are isomorphic (see, e.g., \cite{Colbourn_Rosa_1999, Pike_1999}), so the nonisomorphic Steiner triple systems will yield nonisomorphic SRGs.
\end{rem}

In fact, Wilson \cite{wilson} provided a construction of Steiner 2-designs with a prime power order $p^d$ number of elements (provided that $p^d$ is sufficiently large in relation to the size $k$ of the blocks), which to the authors' knowledge has not been mentioned in the literature as giving rise to PDSs.

\begin{thm}\label{WilsonBuratti-PDS-family}
Let $p$ be a prime such that $p^d > (\frac12k(k-1))^{k(k-1)}$. If $p^d \equiv k(k-1) + 1 \pmod{2k(k-1)}$, then there exists a $(p^d(p^d - 1)/(k(k-1)), k(p^d-k)/(k-1), (p^d -1)/(k-1) + (k-1)^2 -2, k^2)$-PDS in the nonabelian group $C_{p^d} \rtimes C_{(p^d-1)/(k(k-1))} \le (\GF(p^d), +) \rtimes (\GF(p^d)^\times, \cdot)$.
\end{thm}
	\begin{proof}
	Let $\alpha$ be the generator for $(\GF(p^d)^\times, \cdot)$, and let $H = \langle \alpha^{(p^d-1)/k(k-1)} \rangle$. Let $B$ be a $k$ subset of $(\GF(p^d)^\times, \cdot)$, and set $\mathscr{B} = \{ Bg : g \in H\}$. Consider the set of nonzero differences $\delta B = \{b_i - b_j : b_i, b_j \in B, b_i \neq b_j\}$ evaluated over the field $\GF(p^d)$. If $|\delta B \cap Hg|$ is constant for each coset $Hg$ in $(\GF(p^d)^\times, \cdot)$, then by \cite[Theorem 2]{wilson}, $\mathscr{B}$ forms the set of base blocks (defined in \ref{subsec:famSRG}) for an $\S(2, k, p^d)$. By \cite[Corollary 1]{wilson}, there is always some set $B$ such that $\delta B$ is evenly distributed over the cosets $Hg$ whenever $p^d > (\frac12k(k-1))^{k(k-1)}$.
	
	It is straightforward to see that this construction provides a block regular construction for an $\S(2, k, p^d)$, for each block is of the form $h + Bg$ for $h \in (\GF(p^d), +)$ and $g \in H$. Thus, by the discussion in Subsection \ref{subsec:famSRG}, these constructions correspond to a PDS in the nonabelian group $C_{p^d} \rtimes C_{(p^d-1)/(k(k-1))}$, which has parameters $(p^d(p^d - 1)/(k(k-1)), k(p^d-k)/(k-1), (p^d -1)/(k-1) + (k-1)^2 -2, k^2)$.
	\end{proof}

When $p^d \le (\frac12k(k-1))^{k(k-1)}$, Buratti \cite{buratti} provides sufficient conditions for block regular $\S(2, k, p^d)$; we mention one result here, which implies the existence of the $(183, 70, 29, 25)$-PDS we found in SmallGroup(183, 1) related to a block regular $\S(2, 5, 61)$.

\begin{prop}
\label{prop:buratti}
Let $p^d = 20t + 1$ be a prime power, let $2^e$ be the highest power of $2$ in $t$, and suppose
\[ \frac{1}{2}(11 + 5 \sqrt{5}) \notin \{x^{2^{e+1}}: x \in GF(p^d)^\times \}.\] Then, there exists a $(p^d(p^d-1)/20, 5(p^d - 5)/4, (p^d-1)/4 + 14, 25)$-PDS in the nonabelian group $C_{p^d} \rtimes C_{(p^d-1)/20} \le (\GF(p^d), +) \rtimes (\GF(p^d)^\times, \cdot)$.
\end{prop}

\begin{proof}
Let $\alpha$ generate the multiplicative group $(\GF(p^d)^\times, \cdot)$. By \cite[Theorem 7]{buratti}, there always exists a block $B$ whose differences form a system of distinct representatives for the cosets of the group $H = \langle \alpha^{\frac{p-1}{20}} \rangle$. By \cite[Theorem 2]{wilson}, the set $\mathscr{B} = \{ Bg : g \in H\}$ forms the set of base blocks for a block regular $\S(2, 5, p)$. By the discussion in Subsection \ref{subsec:famSRG}, this gives a $(p^d(p^d-1)/20, 5(p^d - 5)/4, (p^d-1)/4 + 14, 25)$-PDS.
\end{proof}

In our own constructions, we have encountered multiple examples of block regular $\S(2, 4, p)$ for $p \le 6^{12}$. In fact, Fuji-Hara et al. \cite{fuji-miao-shinohara} have shown via computer search that for all $p \le 6^{12}$, so long as $p \neq 13$, there does exist a block $B$ satisfying the conditions in \cite[Theorem 2]{wilson} for a block regular $\S(2, 4, p)$. Thus, we have the following proposition.

\begin{prop}\label{Fuji-PDS-family}
If $p \equiv 13 \pmod{2k(k-1)}$ and $p \neq 13$, then there exists a $(p(p - 1)/12, 4(p-4)/3, (p + 20)/3, 16)$-PDS inside the nonabelian group $C_{p} \rtimes C_{(p-1)/12} \le (\GF(p), +) \rtimes (\GF(p)^\times, \cdot)$.
\end{prop}
	\begin{proof}
	Let $\alpha$ generate the multiplicative group $(\GF(p)^\times, \cdot)$. By \cite[Theorem 3.9]{fuji-miao-shinohara}, there always exists a block $B$ whose differences form a system of distinct representatives for the cosets of the group $H = \langle \alpha^{\frac{p-1}{12}} \rangle$. By \cite[Theorem 2]{wilson}, the set $\mathscr{B} = \{ Bg : g \in H\}$ forms the set of base blocks for a block regular $\S(2, 4, p)$. By the discussion in Subsection \ref{subsec:famSRG}, this gives a $(p(p - 1)/12, 4(p-4)/3, (p + 20)/3, 16)$-PDS.
	\end{proof}

\subsection{Non-modular techniques}\label{Non-modular-techniques}
The modular techniques are useful for a number of groups and parameter sets. Nevertheless, there exists a large class of parameters $(v, k, \lambda, \mu)$ for which every prime dividing $\sqrt\Delta$ divides $v$. (For example, every PDS in an abelian group has such parameters.) In these cases, we frequently cannot guarantee that there is a single prime dividing $\sqrt\Delta$ which does not divide $|C_G(h)|$ for most of the elements $h \in G$. As a result, we may only compute the modular values of some of the intersections $|h^G \cap D|$, and it can be quite computationally intensive to guess and check all possibilities for the remaining $|h^G \cap D|$.

Here we provide an example computation, which shows that even if $\sqrt\Delta$ divides $v$, it is still possible to rule out the existence of a PDS for a particular group $G$ just through analyses of the character table of $G$ and the possible corresponding $\Phi$.   Although the modular techniques of this paper do not apply in this setting, the results of Subsection \ref{subsec:ott1} still do.  Note that we are mandating that the difference sets in the following example are reversible; by doing so, we do not need to analyze concrete representations of our group $G$ such as in, e.g., \cite{Liebler, Smith}.

\begin{ex}\label{no-Hadamard-DS}
	We will prove that there is no reversible $(144, 66, 30)$-DS $D$ in SmallGroup(144, 68). In this case, $\sqrt\Delta = 12$, where $12^2 = 144$. Therefore, Theorem \ref{linear-invariance}, Theorem \ref{value-phi(a)}, Theorem \ref{mod-restriction}, and Algorithm \ref{CCI} are all inapplicable.
	
	 Examining the character table for $G$ (which can be generated, for example, using \GAP \cite{GAP4}), there are six conjugacy classes of noncentral elements of order $3$. Each of these conjugacy classes has size $16$, so each element in one of these classes has a centralizer of order $9$. Thus, for any noncentral element $x$ of order $3$ in $G$, we have
		\[ \Phi(x) = 9 |x^G \cap D| \equiv 0 \pmod 9.\]
	Moreover, if $\mathcal{L}$ denotes the set of linear characters for $G$, then $\chi \notin \mathcal{L}$ implies $\chi(x) = 0$. Note that $|\mathcal{L}| = 9$, and each nonprincipal character has order $3$; other than the principal character, there are four conjugate pairs $\{\xi_i, \overline{\xi_i}\},$ $1 \le i \le 4$. Extracting the principal character from this sum and noting that $D$ is reversible (so $\overline{\chi(D)} = \chi(D)$ by Lemma \ref{lem:sumPhi}(i)), we have
		\[ \Phi(x) = 66 + \sum_{i = 1}^4 \xi_i(D) (\xi_i(x) + \overline{\xi_i}(x)).\]
	Since each $\xi_i(D)$ is $\pm 6$ by Lemma \ref{eigenvalues}, there are only $2^4 = 16$ possibilities in total. Checking which have $\Phi(x)$ divisible by $9$, there are only $5$ possibilities that work. (In each case, $\Phi(x) = 36$ or $\Phi(x) = 72$.)

	Of the $15$ remaining nonlinear characters, there is one that equals its conjugate and seven conjugate pairs. Let $\chi$ be a nonlinear irreducible character. Note that $\theta_1 = 6$ and $\theta_2 = -6$. By Lemma \ref{eigenvalues}, if $\repF$ is the representation affording $\chi$, then $\chi(D) = 6 (\dim V_\repF(6) - \dim V_\repF(-6))$, where $-\chi(1) \le (\dim V_\repF(6) - \dim V_\repF(-6)) \le \chi(1)$. So, we may conclude that $\chi(D) \in \{-18, -6, 6, 18\}$.

	Thus, there are $4^8$ possibilities for the coefficients $\chi(D)$ of $\Phi$ ($\chi$ nonlinear), and, from what we have ascertained about the linear characters, there are $5\cdot 4^8$ possibilities for $\Phi$.

	Since either $1 \in D$ or $1 \notin D$, we have $\Phi(1) \in \{0, 144\}$. Of the $5 \cdot 4^8$ possibilities for $\Phi$, a GAP \cite{GAP4} calculation shows that only $15820$ of these possibilities have $\Phi(1) \in \{0, 144\}$.

	Now, $Z(G) = \langle z \rangle$ is cyclic of order $3$. Again, since $D$ is reversible, either $\{z, z^{-1}\} \subseteq D$ or $\{z, z^{-1}\} \cap D = \varnothing$. Thus, either
		\[ \Phi(z) = \Phi(z^{-1}) = 0\]
	or
		\[ \Phi(z) = \Phi(z^{-1}) = 144.\]
	None of the remaining possibilities for $\Phi$ satisfy either of these constraints (in fact, all the values of these class functions are $6 \pmod {12}$ when evaluated at $z$), and so no such $\Phi$ can exist. Therefore, there is no reversible $(144,66,30)$-DS (and hence no regular $(144,66,30,30)$-PDS) in $G$.
	\end{ex}


\subsection*{Acknowledgements}
This work was funded by the Cissy Patterson Fund at William \& Mary. The authors acknowledge William \& Mary Research Computing for providing computational resources and/or technical support that have contributed to the results reported within this paper. URL: \url{https://www.wm.edu/it/rc}

\bibliographystyle{plainurl}
\bibliography{references}

\newpage

\appendix

\section{Parameters already ruled out by \texorpdfstring{\cite{swartz-tauscheck}}{[Swartz-Tauscheck]}}
\label{app:A}

The following tables include all parameters $(v, k, \lambda, \mu)$ with $k < (v-1)/2$ listed in Brouwer's tables \cite{brouwertables} that can be ruled out (entirely, in any group) using either Theorem \ref{mod-restriction} or \cite[Corollaries 5.2, 5.3]{swartz-tauscheck} on either the original parameters or those of its complement, excluding those parameters with $v \in \{512, 768, 1024\}$, in which cases the computations are too expensive due to the number of groups on those orders. Note, we do not list both a parameter set and its complement.

The modular restriction described by Theorem \ref{mod-restriction} overlaps significantly with the modular restriction in \cite[Corollaries 5.2, 5.3]{swartz-tauscheck}. That is, if a parameter set is ruled out by Theorem \ref{mod-restriction}, then it is likely ruled out by \cite[Corollaries 5.2, 5.3]{swartz-tauscheck}, and vice versa. So far, we have yet to find a parameter set that Theorem \ref{mod-restriction} rules out, but which \cite[Corollaries 5.2, 5.3]{swartz-tauscheck} does not rule out. For the sake of comprehensiveness, we list all parameter sets which \cite[Corollaries 5.2, 5.3]{swartz-tauscheck} or Theorem \ref{mod-restriction} rule out.

\begin{table}[ht!]
    \centering
    \begin{minipage}{0.25\textwidth}
        \centering
        \begin{tabular}{c|c|c|c}
            $v$ & $k$ & $\lambda$ & $\mu$ \\ 
            \hline
      15&6&1&3\\ 
      28&12&6&4\\ 
      35&16&6&8\\ 
      45&12&3&3\\ 
      45&16&8&4\\ 
      63&30&13&15\\ 
      69&20&7&5\\ 
      77&16&0&4\\ 
      85&14&3&2\\ 
      85&20&3&5\\ 
      85&30&11&10\\ 
      88&27&6&9\\ 
      91&24&12&4\\ 
      99&14&1&2\\ 
      99&42&21&15\\ 
      99&48&22&24\\ 
      105&26&13&4\\ 
      105&32&4&12\\ 
      105&52&21&30\\ 
      115&18&1&3\\ 
      117&36&15&9\\ 
      119&54&21&27\\ 
      133&24&5&4\\ 
      133&32&6&8\\ 
      133&44&15&14\\ 
      143&70&33&35\\ 
      153&32&16&4\\ 
      153&56&19&21\\ 
      175&30&5&5\\ 
      175&66&29&22\\ 
      175&72&20&36\\ 
        \end{tabular}
    \end{minipage}
    \hfill
    \begin{minipage}{0.25\textwidth}
        \centering
        \begin{tabular}{c|c|c|c}
            $v$ & $k$ & $\lambda$ & $\mu$ \\ 
            \hline
      176&25&0&4\\ 
      176&45&18&9\\ 
      176&70&18&34\\ 
      176&70&24&30\\ 
      176&85&48&34\\ 
      189&48&12&12\\ 
       195&96&46&48\\ 
      208&75&30&25\\ 
      208&81&24&36\\ 
      209&52&15&12\\ 
      209&100&45&50\\ 
      217&66&15&22\\ 
      217&88&39&33\\ 
      221&64&24&16\\ 
      225&96&51&33\\ 
      231&30&9&3\\ 
      231&40&20&4\\ 
      231&70&21&21\\ 
      231&90&33&36\\ 
      232&33&2&5\\ 
      232&63&14&18\\ 
      232&77&36&20\\ 
      232&81&30&27\\ 
      235&42&9&7\\ 
      235&52&9&12\\ 
      236&55&18&11\\ 
      245&52&3&13\\ 
      245&64&18&16\\ 
      245&108&39&54\\ 
      247&54&21&9\\ 
      249&88&27&33\\ 
        \end{tabular}
    \end{minipage}
    \hfill
    \begin{minipage}{0.25\textwidth}
        \centering
        \begin{tabular}{c|c|c|c}
            $v$ & $k$ & $\lambda$ & $\mu$ \\ 
            \hline
      255&126&61&63\\ 
      259&42&5&7\\ 
      261&52&11&10\\ 
      261&64&14&16\\ 
      261&80&25&24\\ 
      261&84&39&21\\ 
      265&96&32&36\\ 
      275&112&30&56\\ 
      279&128&52&64\\ 
      285&64&8&16\\ 
      287&126&45&63\\ 
      297&128&64&48\\
      341&70&15&14\\ 
      341&84&19&21\\ 
      341&102&31&30\\ 
      343&102&21&34\\ 
      343&114&45&34\\ 
      344&147&50&72\\ 
      344&168&92&72\\ 
      345&120&35&45\\ 
      345&128&46&48\\ 
      352&36&0&4\\ 
      352&39&6&4\\ 
      352&117&36&40\\ 
      352&156&60&76\\ 
      352&171&90&76\\ 
      357&100&35&25\\ 
      371&120&44&36\\ 
      375&110&25&35\\ 
      375&136&44&52\\ 
      376&105&32&28\\ 
            \end{tabular}
    \end{minipage}
\end{table}

\begin{table}[ht!]
    \centering
    \begin{minipage}{0.25\textwidth}
        \centering
        \begin{tabular}{c|c|c|c}
      376&175&78&84\\ 
      376&180&88&84\\ 
      377&180&81&90\\ 
      391&140&39&56\\ 
      391&182&93&77\\ 
      407&126&45&36\\ 
      411&130&45&39\\ 
      413&112&36&28\\ 
      424&99&26&22\\ 
      425&72&27&9\\ 
      425&160&60&60\\ 
      428&112&21&32\\ 
      429&108&27&27\\ 
      435&56&28&4\\ 
      435&154&53&55\\ 
      435&182&73&78\\ 
      437&100&15&25\\ 
      451&130&33&39\\
      451&156&57&52\\ 
      459&208&82&104\\ 
      475&90&25&15\\ 
      475&96&32&16\\ 
      477&140&31&45\\  
      304&108&42&36\\ 
      319&150&65&75\\ 
      323&160&78&80\\ 
      325&48&24&4\\ 
      325&54&3&10\\ 
      325&60&15&10\\ 
      325&68&3&17\\ 
      325&72&15&16\\ 
      325&144&68&60\\ 
      329&40&3&5\\ 
      477&168&57&60\\ 
      483&240&118&120\\ 
      495&38&1&3\\ 
      495&78&29&9\\ 
      495&104&28&20\\ 
      495&190&53&85\\ 
      495&208&86&88\\ 
      495&234&93&126\\ 
      496&60&30&4\\ 
      496&110&18&26\\ 
      496&135&38&36\\ 
            \end{tabular}
	\end{minipage}
    \hfill
    \begin{minipage}{0.25\textwidth}
        \centering
        \begin{tabular}{c|c|c|c}  
      496&198&80&78\\ 
      496&231&102&112\\ 
      496&240&120&112\\
      508&234&100&114\\ 
      508&247&126&114\\ 
      511&68&15&8\\ 
      511&78&5&13\\ 
      527&256&120&128\\ 
      533&114&25&24\\ 
      533&132&31&33\\ 
      533&154&45&44\\ 
      536&130&48&26\\ 
      539&234&81&117\\ 
      539&250&105&125\\ 
      553&72&11&9\\ 
      559&264&116&132\\ 
      561&60&9&6\\ 
      561&64&32&4\\
      561&140&31&36\\ 
      561&176&55&55\\ 
      561&240&99&105\\ 
      565&240&104&100\\
       568&217&66&93\\ 
      568&252&96&124\\ 
      568&279&150&124\\ 
      575&112&45&16\\ 
      575&126&29&27\\ 
      575&286&141&143\\ 
      583&96&14&16\\ 
      585&72&7&9\\ 
      585&224&88&84\\ 
      592&216&90&72\\ 
      595&66&33&4\\ 
      595&90&5&15\\ 
      595&108&35&16\\ 
      595&114&33&19\\ 
      595&126&21&28\\ 
      595&132&33&28\\ 
      595&144&18&40\\ 
      595&198&81&58\\ 
      595&216&68&84\\ 
      609&224&91&77\\ 
      611&250&105&100\\ 
      615&294&133&147\\ 
        \end{tabular}
    \end{minipage}
    \hfill
    \begin{minipage}{0.25\textwidth}
        \centering
        \begin{tabular}{c|c|c|c} 
      616&75&2&10\\ 
      616&120&20&24\\ 
      616&205&90&57\\
      616&240&89&96\\ 
      616&270&108&126\\ 
      621&60&3&6\\ 
      621&144&24&36\\ 
      621&220&79&77\\ 
      621&300&123&165\\ 
      625&246&119&82\\ 
      637&60&11&5\\ 
      637&96&5&16\\ 
      637&176&60&44\\ 
      637&186&35&62\\ 
      637&252&91&105\\ 
      639&288&112&144\\ 
      640&210&66&70\\ 
      640&213&72&70\\ 
      645&140&31&30\\ 
      645&160&38&40\\ 
      645&184&53&52\\ 
      645&210&85&60\\ 
      649&72&15&7\\ 
      649&216&63&76\\ 
      649&256&108&96\\ 
      664&170&24&50\\ 
      665&280&135&105\\ 
      667&96&0&16\\ 
      675&336&166&168\\ 
      693&192&66&48\\ 
      697&96&20&12\\ 
      697&108&9&18\\ 
      697&348&167&180\\
      703&72&36&4\\ 
      703&182&81&35\\ 
      705&256&80&100\\ 
      705&308&145&126\\ 
      711&70&5&7\\ 
      711&230&85&69\\ 
      715&120&20&20\\ 
      715&154&33&33\\ 
      715&224&48&80\\ 
      715&228&65&76\\ 
      715&266&105&95\\ 
        \end{tabular}
    \end{minipage}
\end{table}

\begin{table}[ht!]
    \centering
    \begin{minipage}{0.25\textwidth}
        \centering
        \begin{tabular}{c|c|c|c} 
      721&192&48&52\\ 
      721&220&69&66\\ 
      721&324&147&144\\
      725&196&63&49\\ 
      729&208&37&68\\ 
      729&248&67&93\\ 
      729&280&127&95\\ 
      731&250&105&75\\ 
      735&318&109&159\\ 
      735&360&172&180\\ 
      736&60&14&4\\ 
      736&105&20&14\\ 
      736&180&68&36\\ 
      736&270&114&90\\ 
      736&294&122&114\\ 
      736&315&106&156\\ 
      736&330&140&154\\ 
      736&350&162&170\\ 
      736&357&176&170\\ 
      736&364&204&156\\ 
      765&176&28&44\\ 
      765&192&48&48\\ 
      771&266&85&95\\ 
      775&150&45&25\\ 
      775&288&98&112\\ 
      779&378&177&189\\ 
      781&290&93&116\\ 
      781&348&167&145\\ 
      783&230&49&75\\ 
      783&390&193&195\\ 
      793&162&21&36\\ 
      793&280&87&105\\ 
      793&330&147&130\\ 
      805&324&123&135\\
      805&336&140&140\\ 
      817&240&83&65\\ 
      825&96&4&12\\ 
      825&128&40&16\\ 
      825&280&75&105\\ 
      833&112&21&14\\ 
      833&130&21&20\\ 
      833&232&81&58\\ 
      833&256&84&76\\ 
      837&76&15&6\\ 
      837&176&40&36\\ 
            \end{tabular}
    \end{minipage}
    \hfill
    \begin{minipage}{0.25\textwidth}
        \centering
        \begin{tabular}{c|c|c|c} 
      837&196&35&49\\ 
      837&286&85&104\\ 
      837&308&127&105\\
      837&396&195&180\\ 
      845&204&43&51\\ 
      845&256&108&64\\ 
      845&396&171&198\\ 
      847&94&21&9\\ 
      847&144&26&24\\ 
      847&270&109&75\\ 
      847&282&81&100\\ 
      847&300&117&100\\ 
      847&360&143&160\\ 
      847&390&161&195\\ 
      848&121&24&16\\ 
      848&231&70&60\\ 
      848&297&96&108\\ 
      848&385&156&190\\
      848&418&222&190\\ 
      851&306&97&117\\ 
      851&360&163&144\\ 
      855&126&21&18\\ 
      855&168&20&36\\ 
      855&182&37&39\\ 
      855&336&150&120\\ 
      856&95&6&11\\ 
      856&120&21&16\\ 
      856&209&72&44\\ 
      861&80&40&4\\ 
      861&140&31&21\\ 
      861&280&105&84\\ 
      861&300&103&105\\ 
      861&380&163&171\\ 
      871&60&5&4\\ 
      871&144&22&24\\ 
      871&348&137&140\\ 
      872&208&54&48\\ 
      875&114&13&15\\ 
      885&260&55&85\\ 
      891&290&109&87\\ 
      891&320&148&96\\ 
      897&126&15&18\\ 
      899&448&222&224\\ 
      901&60&3&4\\ 
      901&200&45&44\\ 
            \end{tabular}
    \end{minipage}
    \hfill
    \begin{minipage}{0.25\textwidth}
        \centering
        \begin{tabular}{c|c|c|c}  
      901&224&54&56\\ 
      901&252&71&70\\ 
      913&192&56&36\\
      923&432&186&216\\ 
      925&108&39&9\\ 
      925&220&75&45\\ 
      925&330&95&130\\ 
      925&336&104&132\\ 
      925&374&123&170\\ 
      925&378&143&162\\ 
      925&384&158&160\\ 
      925&448&192&240\\ 
      928&300&90&100\\ 
      928&441&224&196\\ 
      931&120&11&16\\ 
      931&138&5&23\\ 
      931&162&33&27\\ 
      931&180&29&36\\
      931&248&60&68\\ 
      931&280&84&84\\ 
      931&372&143&152\\ 
      931&390&173&156\\ 
      931&450&241&195\\ 
      935&448&204&224\\ 
      945&64&8&4\\ 
      945&176&31&33\\ 
      949&84&11&7\\ 
      957&256&80&64\\ 
      969&176&13&36\\ 
      969&198&57&36\\ 
      969&392&175&147\\ 
      973&272&96&68\\ 
      976&117&28&12\\ 
      976&273&92&70\\ 
      976&300&128&76\\ 
      976&455&198&224\\ 
      976&480&248&224\\ 
      987&170&49&25\\ 
      989&484&231&242\\ 
      999&448&172&224\\ 
      1001&136&27&17\\ 
      1001&150&13&24\\ 
      1001&160&12&28\\ 
      1001&280&63&84\\ 
      1001&300&103&84\\ 
        \end{tabular}
    \end{minipage}
\end{table}

\begin{table}[ht]
    \begin{minipage}{0.25\textwidth}
        \centering
        \begin{tabular}{c|c|c|c}  
      1001&384&152&144\\ 
      1001&450&207&198\\ 
      1003&300&65&100\\
      1007&486&225&243\\ 
      1015&312&113&88\\ 
      1016&259&42&74\\ 
      1017&344&91&129\\ 
      1023&336&90&120\\ 
      1023&510&253&255\\ 
      1035&88&44&4\\ 
      1035&330&105&105\\ 
      1035&440&196&180\\ 
      1035&462&201&210\\ 
      1035&490&217&245\\ 
      1041&336&115&105\\ 
      1045&234&53&52\\ 
      1045&260&63&65\\ 
      1045&276&83&69\\
      1045&290&81&80\\ 
      1045&324&123&90\\ 
      1045&384&108&160\\ 
      1045&396&143&154\\ 
      1065&266&103&54\\ 
      1065&364&113&130\\ 
      1065&416&172&156\\ 
      1067&208&42&40\\ 
      1071&110&13&11\\ 
      1071&320&94&96\\ 
      1072&252&56&60\\ 
      1072&261&80&58\\ 
      1072&345&102&115\\ 
      1072&391&150&138\\ 
      1073&64&0&4\\ 
      1075&192&44&32\\ 
      1075&210&65&35\\ 
      1079&238&57&51\\ 
      1084&513&232&252\\ 
      1084&532&270&252\\ 
      1085&256&48&64\\ 
      1089&320&67&105\\ 
      1089&408&177&138\\ 
        \end{tabular}
    \end{minipage}
    \hfill
    \begin{minipage}{0.25\textwidth}
        \centering
        \begin{tabular}{c|c|c|c}  
      1105&80&15&5\\ 
      1105&138&11&18\\
      1105&208&39&39\\ 
      1105&336&95&105\\ 
      1105&378&135&126\\ 
      1107&378&117&135\\ 
      1111&110&9&11\\ 
      1112&396&135&144\\ 
      1125&308&103&77\\ 
      1131&320&76&96\\ 
      1135&216&45&40\\ 
      1135&238&45&51\\
      1136&280&108&56\\ 
      1147&216&60&36\\ 
      1155&576&286&288\\ 
      1161&360&99&117\\ 
      1161&380&145&114\\ 
      1177&204&41&34\\ 
      1177&224&36&44\\ 
      1184&546&230&270\\ 
      1184&585&308&270\\ 
      1189&132&11&15\\ 
      1189&216&35&40\\ 
      1189&396&135&130\\ 
      1189&588&287&294\\ 
      1199&550&225&275\\ 
      1208&459&150&189\\ 
      1215&544&208&272\\ 
      1215&574&253&287\\ 
      1215&598&289&299\\ 
      1216&135&6&16\\ 
      1216&144&24&16\\ 
      1216&162&18&22\\ 
      1216&165&24&22\\ 
      1216&270&38&66\\ 
      1216&375&150&100\\ 
      1216&405&144&130\\ 
      1216&540&204&268\\ 
      1216&567&246&280\\ 
      1216&585&276&286\\ 
      1216&594&294&286\\ 
            \end{tabular}
    \end{minipage}
    \hfill
    \begin{minipage}{0.25\textwidth}
        \centering
            \begin{tabular}{c|c|c|c}
      1216&600&312&280\\
      1225&168&35&21\\ 
      1225&264&63&55\\ 
      1225&384&138&112\\ 
      1225&480&200&180\\ 
      1225&520&255&195\\ 
      1233&396&129&126\\ 
      1233&532&231&228\\ 
      1241&160&24&20\\ 
      1241&310&81&76\\ 
      1241&490&189&196\\ 
      1247&126&45&9\\ 
      1251&270&73&54\\ 
      1261&492&171&205\\ 
      1261&574&279&246\\
      1267&186&5&31\\ 
      1269&288&42&72\\ 
      1271&160&48&16\\ 
      1275&98&13&7\\ 
      1275&98&49&4\\ 
      1275&336&134&72\\ 
      1275&350&85&100\\ 
      1275&364&113&100\\ 
      1275&490&185&190\\ 
      1275&490&225&165\\ 
      1275&504&228&180\\ 
      1275&546&281&198\\ 
      1275&572&247&264\\ 
      1285&528&212&220\\ 
      1288&162&36&18\\ 
      1288&195&26&30\\ 
      1288&195&54&25\\ 
      1288&234&80&34\\ 
       1288&495&206&180\\ 
      1288&567&246&252\\ 
      1295&646&321&323\\  
      1216&603&330&268\\ 
      1221&324&99&81\\ 
      1221&500&175&225\\ 
      1221&610&279&330\\ 
      1225&96&48&4\\ 
            \end{tabular}
	\end{minipage}
\end{table}

\FloatBarrier

$\phantom{x}$

\newpage

\section{Parameters for regular PDSs ruled out using Algorithm \ref{CCI}.}
\label{app:B}

The first set of tables consists of all parameters $(v, k, \lambda, \mu)$ with $k < (v-1)/2$ listed in Brouwer's tables \cite{brouwertables} satisfying the requirement that $\Delta$ is divisible by some prime $p$ not dividing $v$, for which we have shown no PDS exists in any group through direct computation using Algorithm \ref{CCI} on either the original parameters or those of the complement. We have only included all parameters with $v < 507$, for otherwise the computations become too expensive. Once again, we do not include complements of listed parameters.

\begin{table}[ht]
    \centering
    \begin{minipage}{0.25\textwidth}
        \centering
        \begin{tabular}{c|c|c|c}
            $v$ & $k$ & $\lambda$ & $\mu$ \\ 
            \hline
		10 & 3 & 0 & 1 \\ 
		26 & 10 & 3 & 4 \\ 
		28 & 12 & 6 & 4 \\ 
		36 & 14 & 7 & 4 \\ 
		40 & 12 & 2 & 4 \\ 
		50 & 21 & 8 & 9 \\ 
		56 & 10 & 0 & 2 \\ 
		63 & 30 & 13 & 15 \\ 
		66 & 20 & 10 & 4 \\ 
		70 & 27 & 12 & 9 \\ 
		78 & 22 & 11 & 4 \\ 
		82 & 36 & 15 & 16 \\ 
		88 & 27 & 6 & 9 \\ 
		96 & 35 & 10 & 14 \\ 
		100 & 33 & 14 & 9 \\ 
		105 & 26 & 13 & 4 \\ 
		105 & 32 & 4 & 12 \\ 
		105 & 40 & 15 & 15 \\ 
		105 & 52 & 21 & 30 \\ 
		112 & 30 & 2 & 10 \\ 
		112 & 36 & 10 & 12 \\ 
		117 & 36 & 15 & 9 \\ 
		120 & 28 & 14 & 4 \\ 
		120 & 42 & 8 & 18 \\ 
		122 & 55 & 24 & 25 \\ 
		126 & 25 & 8 & 4 \\ 
		126 & 50 & 13 & 24 \\ 
		126 & 60 & 33 & 24 \\ 
		130 & 48 & 20 & 16 \\ 
		135 & 64 & 28 & 32 \\ 
		136 & 30 & 15 & 4 \\ 
		136 & 63 & 30 & 28 \\ 
		148 & 63 & 22 & 30 \\ 
		154 & 48 & 12 & 16 \\ 
		154 & 72 & 26 & 40 \\ 
        \end{tabular}
    \end{minipage}
    \hfill
    \begin{minipage}{0.25\textwidth}
        \centering
        \begin{tabular}{c|c|c|c}
            $v$ & $k$ & $\lambda$ & $\mu$ \\ 
            \hline
		156 & 30 & 4 & 6 \\ 
		165 & 36 & 3 & 9 \\ 
		170 & 78 & 35 & 36 \\ 
		171 & 50 & 13 & 15 \\ 
		176 & 25 & 0 & 4 \\ 
		176 & 40 & 12 & 8 \\ 
		176 & 45 & 18 & 9 \\ 
		176 & 49 & 12 & 14 \\ 
		176 & 70 & 18 & 34 \\ 
		176 & 70 & 24 & 30 \\ 
		176 & 85 & 48 & 34 \\ 
		189 & 48 & 12 & 12 \\ 
		189 & 60 & 27 & 15 \\ 
		189 & 88 & 37 & 44 \\ 
		190 & 36 & 18 & 4 \\ 
		190 & 45 & 12 & 10 \\ 
		190 & 84 & 33 & 40 \\ 
		190 & 84 & 38 & 36 \\ 
		190 & 90 & 45 & 40 \\ 
		195 & 96 & 46 & 48 \\ 
		196 & 39 & 2 & 9 \\ 
		196 & 60 & 23 & 16 \\ 
		196 & 81 & 42 & 27 \\ 
		204 & 28 & 2 & 4 \\ 
		204 & 63 & 22 & 18 \\ 
		208 & 45 & 8 & 10 \\ 
		208 & 75 & 30 & 25 \\ 
		208 & 81 & 24 & 36 \\ 
		210 & 38 & 19 & 4 \\ 
		220 & 72 & 22 & 24 \\ 
		220 & 84 & 38 & 28 \\ 
		222 & 51 & 20 & 9 \\ 
		225 & 96 & 51 & 33 \\ 
		226 & 105 & 48 & 49 \\ 
		231 & 30 & 9 & 3 \\ 
        \end{tabular}
    \end{minipage}
    \hfill
    \begin{minipage}{0.25\textwidth}
        \centering
        \begin{tabular}{c|c|c|c}
            $v$ & $k$ & $\lambda$ & $\mu$ \\ 
            \hline
		231 & 40 & 20 & 4 \\ 
		231 & 70 & 21 & 21 \\ 
		231 & 90 & 33 & 36 \\ 
		232 & 33 & 2 & 5 \\ 
		232 & 63 & 14 & 18 \\ 
		232 & 77 & 36 & 20 \\ 
		232 & 81 & 30 & 27 \\ 
		236 & 55 & 18 & 11 \\ 
		238 & 75 & 20 & 25 \\ 
		243 & 66 & 9 & 21 \\ 
		244 & 117 & 60 & 52 \\ 
		246 & 85 & 20 & 34 \\ 
		246 & 105 & 36 & 51 \\ 
		246 & 119 & 64 & 51 \\ 
		260 & 70 & 15 & 20 \\ 
		266 & 45 & 0 & 9 \\ 
		273 & 72 & 21 & 18 \\ 
		273 & 136 & 65 & 70 \\ 
		275 & 112 & 30 & 56 \\ 
		276 & 44 & 22 & 4 \\ 
		276 & 75 & 10 & 24 \\ 
		276 & 75 & 18 & 21 \\ 
		276 & 110 & 52 & 38 \\ 
		276 & 135 & 78 & 54 \\ 
		279 & 128 & 52 & 64 \\ 
		280 & 36 & 8 & 4 \\ 
		280 & 117 & 44 & 52 \\ 
		285 & 64 & 8 & 16 \\ 
		286 & 95 & 24 & 35 \\ 
		286 & 125 & 60 & 50 \\ 
		288 & 105 & 52 & 30 \\ 
		288 & 112 & 36 & 48 \\ 
		290 & 136 & 63 & 64 \\ 
		297 & 40 & 7 & 5 \\ 
		297 & 104 & 31 & 39 \\ 
        \end{tabular}
    \end{minipage}
\end{table}

\begin{table}[ht]
    \centering
    \begin{minipage}{0.25\textwidth}
        \centering
        \begin{tabular}{c|c|c|c}
		297 & 128 & 64 & 48 \\ 
		300 & 46 & 23 & 4 \\ 
		304 & 108 & 42 & 36 \\ 
		306 & 55 & 4 & 11 \\ 
		306 & 60 & 10 & 12 \\ 
		320 & 99 & 18 & 36 \\ 
		322 & 96 & 20 & 32 \\ 
		324 & 68 & 7 & 16 \\ 
		324 & 95 & 34 & 25 \\ 
		330 & 63 & 24 & 9 \\ 
		336 & 80 & 28 & 16 \\ 
		336 & 125 & 40 & 50 \\ 
		340 & 108 & 30 & 36 \\ 
		342 & 33 & 4 & 3 \\ 
		342 & 66 & 15 & 12 \\ 
		343 & 102 & 21 & 34 \\ 
		343 & 114 & 45 & 34 \\ 
		344 & 147 & 50 & 72 \\ 
		344 & 168 & 92 & 72 \\ 
		351 & 70 & 13 & 14 \\ 
		351 & 140 & 49 & 60 \\ 
		351 & 160 & 64 & 80 \\ 
		352 & 26 & 0 & 2 \\ 
		352 & 36 & 0 & 4 \\ 
		352 & 39 & 6 & 4 \\ 
		352 & 108 & 44 & 28 \\ 
		352 & 117 & 36 & 40 \\ 
		352 & 126 & 50 & 42 \\ 
		352 & 156 & 60 & 76 \\ 
		352 & 171 & 90 & 76 \\ 
		357 & 100 & 35 & 25 \\ 
		362 & 171 & 80 & 81 \\ 
		364 & 33 & 2 & 3 \\ 
		364 & 66 & 20 & 10 \\ 
		364 & 88 & 12 & 24 \\ 
		364 & 120 & 38 & 40 \\ 
		364 & 121 & 48 & 36 \\ 
		364 & 165 & 68 & 80 \\ 
		364 & 176 & 90 & 80 \\ 
		372 & 56 & 10 & 8 \\
		375 & 102 & 45 & 21 \\
		375 & 110 & 25 & 35 \\
        \end{tabular}
    \end{minipage}
    \hfill
    \begin{minipage}{0.25\textwidth}
        \centering
        \begin{tabular}{c|c|c|c}
		375 & 136 & 44 & 52 \\
		375 & 182 & 85 & 91 \\ 
		376 & 105 & 32 & 28 \\ 
		376 & 175 & 78 & 84 \\ 
		376 & 180 & 88 & 84 \\ 
		378 & 52 & 1 & 8 \\ 
		378 & 52 & 26 & 4 \\ 
		385 & 60 & 5 & 10 \\ 
		385 & 168 & 77 & 70 \\ 
		392 & 69 & 26 & 9 \\ 
		392 & 153 & 54 & 63 \\ 
		396 & 135 & 30 & 54 \\ 
		396 & 150 & 51 & 60 \\ 
		399 & 198 & 97 & 99 \\ 
		400 & 21 & 2 & 1 \\ 
		400 & 56 & 6 & 8 \\ 
		400 & 133 & 42 & 45 \\ 
		405 & 132 & 63 & 33 \\ 
		405 & 196 & 91 & 98 \\ 
		406 & 54 & 27 & 4 \\ 
		406 & 108 & 30 & 28 \\ 
		406 & 165 & 68 & 66 \\ 
		406 & 189 & 84 & 91 \\ 
		406 & 195 & 96 & 91 \\ 
		408 & 176 & 70 & 80 \\ 
		414 & 63 & 12 & 9 \\ 
		416 & 100 & 36 & 20 \\ 
		416 & 165 & 64 & 66 \\ 
		418 & 147 & 56 & 49 \\ 
		424 & 99 & 26 & 22 \\ 
		428 & 112 & 21 & 32 \\ 
		429 & 108 & 27 & 27 \\ 
		430 & 39 & 8 & 3 \\ 
		430 & 135 & 36 & 45 \\ 
		430 & 165 & 68 & 60 \\ 
		438 & 92 & 31 & 16 \\ 
		441 & 56 & 7 & 7 \\ 
		441 & 190 & 89 & 76 \\
		441 & 220 & 95 & 124 \\
		442 & 210 & 99 & 100 \\
		456 & 65 & 10 & 9 \\
		456 & 80 & 4 & 16 \\
        \end{tabular}
    \end{minipage}
    \hfill
    \begin{minipage}{0.25\textwidth}
        \centering
        \begin{tabular}{c|c|c|c}
		456 & 130 & 24 & 42 \\
		456 & 140 & 40 & 44 \\ 
		456 & 140 & 58 & 36 \\ 
		456 & 175 & 78 & 60 \\ 
		456 & 182 & 73 & 72 \\ 
		456 & 195 & 74 & 90 \\ 
		459 & 208 & 82 & 104 \\ 
		460 & 85 & 18 & 15 \\ 
		460 & 99 & 18 & 22 \\ 
		460 & 147 & 42 & 49 \\ 
		460 & 204 & 78 & 100 \\ 
		460 & 216 & 116 & 88 \\ 
		460 & 225 & 120 & 100 \\ 
		465 & 144 & 43 & 45 \\ 
		470 & 126 & 27 & 36 \\ 
		474 & 165 & 52 & 60 \\ 
		476 & 133 & 42 & 35 \\ 
		476 & 133 & 60 & 28 \\ 
		483 & 240 & 118 & 120 \\ 
		484 & 105 & 14 & 25 \\ 
		484 & 138 & 47 & 36 \\ 
		490 & 165 & 56 & 55 \\ 
		490 & 192 & 92 & 64 \\ 
		494 & 85 & 12 & 15 \\ 
		495 & 38 & 1 & 3 \\ 
		495 & 78 & 29 & 9 \\ 
		495 & 104 & 28 & 20 \\ 
		495 & 190 & 53 & 85 \\ 
		495 & 208 & 86 & 88 \\ 
		495 & 234 & 93 & 126 \\ 
		495 & 238 & 109 & 119 \\ 
		496 & 54 & 4 & 6 \\ 
		496 & 60 & 30 & 4 \\ 
		496 & 110 & 18 & 26 \\ 
		496 & 135 & 38 & 36 \\ 
		496 & 198 & 80 & 78 \\
		496 & 231 & 102 & 112 \\
		496 & 240 & 120 & 112 \\
		498 & 161 & 64 & 46 \\
		506 & 100 & 18 & 20 \\
		507 & 138 & 49 & 33 \\
		507 & 240 & 106 & 120 \\
        \end{tabular}
    \end{minipage}
\end{table}

\FloatBarrier

We have additionally computed all parameters $(v, k, \lambda, \mu)$ listed in \hypertarget{brouwer-table}{\href{https://aeb.win.tue.nl/graphs/srg/srgtab.html}{Brouwer's tables}} with $k < (v-1)/2$ and $507 \le v \le 1300$ such that $\gcd(v, \Delta) = 1$. Of those, the following parameter sets are ruled out by our computations using Algorithm \ref{CCI}.

\begin{table}[ht]
    \centering
    \begin{minipage}{0.25\textwidth}
        \centering
        \begin{tabular}{c|c|c|c}
            $v$ & $k$ & $\lambda$ & $\mu$ \\ 
            \hline
		528 & 62 & 31 & 4 \\ 
		530 & 253 & 120 & 121 \\ 
		532 & 81 & 30 & 9 \\ 
		546 & 125 & 40 & 25 \\ 
		568 & 217 & 66 & 93 \\ 
		585 & 72 & 7 & 9 \\ 
		606 & 105 & 4 & 21 \\ 
		610 & 87 & 32 & 9 \\ 
		615 & 294 & 133 & 147 \\ 
		616 & 205 & 90 & 57 \\ 
		616 & 240 & 89 & 96 \\ 
		626 & 300 & 143 & 144 \\ 
		630 & 68 & 1 & 8 \\ 
		636 & 250 & 95 & 100 \\ 
		645 & 160 & 38 & 40 \\ 
		645 & 184 & 53 & 52 \\ 
		658 & 162 & 51 & 36 \\ 
		666 & 70 & 35 & 4 \\ 
		675 & 336 & 166 & 168 \\ 
		676 & 150 & 23 & 36 \\ 
		676 & 189 & 62 & 49 \\ 
		690 & 318 & 121 & 168 \\ 
		696 & 125 & 10 & 25 \\ 
		700 & 243 & 90 & 81 \\ 
		711 & 70 & 5 & 7 \\ 
		714 & 161 & 40 & 35 \\ 
		714 & 248 & 97 & 80 \\ 
		715 & 228 & 65 & 76 \\ 
		715 & 266 & 105 & 95 \\ 
		726 & 145 & 44 & 25 \\ 
		730 & 351 & 168 & 169 \\ 
		741 & 74 & 37 & 4 \\ 
		742 & 255 & 92 & 85 \\ 
		742 & 351 & 180 & 153 \\ 
		748 & 180 & 53 & 40 \\ 
		754 & 243 & 72 & 81 \\ 
        \end{tabular}
    \end{minipage}
    \hfill
    \begin{minipage}{0.25\textwidth}
        \centering
        \begin{tabular}{c|c|c|c}
            $v$ & $k$ & $\lambda$ & $\mu$ \\ 
            \hline
		760 & 161 & 30 & 35 \\ 
		776 & 124 & 39 & 16 \\ 
		780 & 369 & 158 & 189 \\ 
		782 & 99 & 36 & 9 \\ 
		783 & 390 & 193 & 195 \\ 
		820 & 78 & 39 & 4 \\ 
		825 & 128 & 40 & 16 \\ 
		837 & 76 & 15 & 6 \\ 
		837 & 196 & 35 & 49 \\ 
		842 & 406 & 195 & 196 \\ 
		855 & 168 & 20 & 36 \\ 
		856 & 95 & 6 & 11 \\ 
		856 & 120 & 21 & 16 \\ 
		861 & 80 & 40 & 4 \\ 
		876 & 105 & 38 & 9 \\ 
		900 & 203 & 34 & 49 \\ 
		900 & 248 & 79 & 64 \\ 
		904 & 301 & 78 & 111 \\ 
		904 & 378 & 177 & 144 \\ 
		935 & 448 & 204 & 224 \\ 
		945 & 64 & 8 & 4 \\ 
		946 & 243 & 60 & 63 \\ 
		952 & 351 & 150 & 117 \\ 
		962 & 465 & 224 & 225 \\ 
		987 & 170 & 49 & 25 \\ 
		990 & 86 & 43 & 4 \\ 
		1000 & 117 & 18 & 13 \\ 
		1015 & 312 & 113 & 88 \\ 
		1023 & 510 & 253 & 255 \\ 
		1026 & 451 & 212 & 187 \\ 
		1036 & 405 & 138 & 171 \\ 
		1044 & 175 & 50 & 25 \\ 
		1045 & 234 & 53 & 52 \\ 
		1045 & 260 & 63 & 65 \\ 
		1045 & 276 & 83 & 69 \\
        \end{tabular}
        \vspace{0.47cm}
    \end{minipage}
    \hfill
    \begin{minipage}{0.25\textwidth}
        \centering
        \begin{tabular}{c|c|c|c}
            $v$ & $k$ & $\lambda$ & $\mu$ \\ 
            \hline
		1045 & 290 & 81 & 80 \\ 
		1054 & 363 & 132 & 121 \\ 
		1071 & 110 & 13 & 11 \\ 
		1072 & 345 & 102 & 115 \\ 
		1085 & 256 & 48 & 64 \\ 
		1086 & 155 & 4 & 25 \\ 
		1086 & 210 & 59 & 36 \\ 
		1090 & 528 & 255 & 256 \\ 
		1106 & 403 & 132 & 155 \\ 
		1112 & 396 & 135 & 144 \\ 
		1116 & 245 & 70 & 49 \\ 
		1120 & 363 & 110 & 121 \\ 
		1140 & 469 & 158 & 217 \\ 
		1148 & 481 & 210 & 195 \\ 
		1155 & 576 & 286 & 288 \\ 
		1156 & 264 & 47 & 64 \\ 
		1156 & 315 & 98 & 81 \\ 
		1176 & 94 & 47 & 4 \\ 
		1178 & 99 & 0 & 9 \\ 
		1190 & 123 & 44 & 9 \\ 
		1206 & 490 & 203 & 196 \\ 
		1208 & 459 & 150 & 189 \\ 
		1210 & 156 & 47 & 16 \\ 
		1212 & 525 & 230 & 225 \\ 
		1221 & 610 & 279 & 330 \\ 
		1222 & 481 & 180 & 195 \\ 
		1226 & 595 & 288 & 289 \\ 
		1236 & 95 & 10 & 7 \\ 
		1236 & 500 & 205 & 200 \\ 
		1248 & 203 & 22 & 35 \\ 
		1251 & 270 & 73 & 54 \\ 
		1254 & 203 & 52 & 29 \\ 
		1275 & 98 & 49 & 4 \\ 
		1275 & 546 & 281 & 198 \\ 
		1288 & 195 & 54 & 25 \\
        \end{tabular}
        \vspace{0.47cm}
    \end{minipage}
\end{table}

Finally, we mention again here that the parameter set $(111, 30, 5, 9)$ has been ruled out using a combination of Algorithm \ref{CCI} and Gurobi \cite{gurobi} (Remark \ref{rem:111}).

\newpage

\section{Parameter sets with a valid output from Algorithm \ref{CCI}}
\label{app:C}

We now list all parameter sets $(v, k, \lambda, \mu)$ with $k < (v-1)/2$ listed in Brouwer's tables \cite{brouwertables} satisfying the condition that either $\Delta$ has some distinct prime factor with $v$ and $v \le 507$ or $\gcd(v, \Delta) = 1$ and $v \le 1300$, such that there does exist a valid $\Phi$ function, or such that this function was not computable with the available computational resources.  We note once again that we exclude parameter sets with $v \in \{512, 768, 1024\}$, as well as the parameter set $(507, 176, 70, 56)$, for these computations are too expensive. As mentioned, we do not include complements of listed parameters -- these may be easily generated from what is provided. We begin with the table of all parameters in groups which were too computationally expensive to compute.

\begin{table}[ht]
    \centering
    \begin{minipage}{0.25\textwidth}
        \centering
        \begin{tabular}{c|c}
		 $(v, k, \lambda, \mu)$ & ID  \\ 
		 \hline
		 $(343, 162, 81, 72)$ & 3 \\
		 $(343, 162, 81, 72)$ & 4 \\
        \end{tabular}
    \end{minipage}
    \hfill
    \begin{minipage}{0.25\textwidth}
        \centering
        \begin{tabular}{c|c}
		 $(v, k, \lambda, \mu)$ & ID  \\
		 \hline
		 $(375,176,94,72)$ & 4 \\
		 $(375,176,94,72)$ & 5 \\
        \end{tabular}
    \end{minipage}
     \hfill
    \begin{minipage}{0.25\textwidth}
        \centering
        \begin{tabular}{c|c} 
		 $(v, k, \lambda, \mu)$ & ID  \\
		 \hline
		 $(495,190,85,65)$ & 1 \\
		 $(495,190,85,65)$ & 3 \\
        \end{tabular}
    \end{minipage}
\end{table}

\FloatBarrier

We now provide all successful computations. We have organized the table as follows. First by the parameter set $(v, k, \lambda, \mu)$, then by the column labeled ID, which provides the GAP \cite{GAP4} group ID for which a valid $\Phi$ function has been constructed, then by the column labeled ?, which describes if a PDS for that group on that parameter set has been constructed (in which case a reference is given to its construction, either in this paper or elsewhere) or has not been either constructed or ruled out (?), in which the existence is still an open question.  Finally, we provide the conjugacy class intersections of the possible PDS inside of the group. The intersections are given in the form of a list [$a_1, \dots, a_n$], such that, if $G$ is the group, then the $i$th entry of [$a_1, \dots, a_n$] corresponds to the intersection of the $i$th conjugacy class in ConjugacyClasses(CharacterTable($G$)) with the possible PDS.  (Note that ConjugacyClasses($G$) and ConjugacyClasses(CharacterTable($G$)) may have different conjugacy class orderings.) We have listed all conjugacy class intersections up to the automorphisms of $G$. That is, if a conjugacy class intersection is omitted for a given group, then an automorphically equivalent intersection is listed. We also provide a full .txt file containing all of these parameters in \href{https://github.com/srnelson1/PDS-class-intersections}{this} GitHub repository.

We also mention that the parameter set $(243, 112, 46, 56)$ has over 400 valid, automorphically distinct intersections. To avoid cluttering the paper, these have been placed in a separate file at the same GitHub repository located \href{https://github.com/srnelson1/PDS-class-intersections}{here}.

\begin{center}
	\begin{longtable}{p{3.66cm}|P{0.8cm}|P{0.69cm}|p{10cm}}
		 $(v, k, \lambda, \mu)$ & ID & ?  & Conjugacy Class Intersections \\
		\hline
		 $(21, 10, 3, 6)$  & 1 & \cite{pol-dav-smi-swa} & [ 0, 3, 2, 3, 2 ] \\
		\hline
		 $(27, 10, 1, 5)$  & 3 & \cite{Kantor_1986} & [ 0, 1, 1, 1, 1, 1, 1, 1, 1, 1, 1 ] \\
		   & 4 & \cite{Swartz_2015} & [ 0, 1, 1, 1, 1, 1, 1, 1, 1, 1, 1 ] \\
		\hline
		 $(55, 18, 9, 4)$  & 1 & \cite{pol-dav-smi-swa} & [ 0, 4, 1, 4, 1, 4, 4 ] \\
		\hline
		 $(57, 24, 11, 9)$  & 1 & \ref{Clapham-PDS-family} & [ 0, 9, 1, 9, 1, 1, 1, 1, 1 ] \\
		\hline
		 $(111, 30, 5, 9)$  & 1 & \hyperref[rem:111]{\footnotesize DNE} & [ 0, 9, 1, 9, 1, 1, 1, 1, 1, 1, 1, 1, 1, 1, 1 ] \\
		\hline
		 $(111, 44, 19, 16)$  & 1 & \ref{Fuji-PDS-family} & [ 0, 16, 1, 16, 1, 1, 1, 1, 1, 1, 1, 1, 1, 1, 1 ] \\
		\hline
		 $(125, 28, 3, 7)$  & 3 & \cite{Kantor_1986} & [ 0, 1, 1, 1, 1, 1, 1, 1, 1, 1, 1, 1, 1, 1, 1, 1, 1, 1, 1, 1, 1, 1, 1, 1, 1, 1, 1, 1, 1 ] \\
		   & 4 & \cite{Swartz_2015} & [ 0, 1, 1, 1, 1, 1, 1, 1, 1, 1, 1, 1, 1, 1, 1, 1, 1, 1, 1, 1, 1, 1, 1, 1, 1, 1, 1, 1, 1 ] \\
		\hline
		 $(125, 52, 15, 26)$  & 3 & \cite{pol-dav-smi-swa} & [ 0, 2, 2, 1, 2, 2, 2, 1, 2, 2, 2, 2, 1, 2, 2, 2, 2, 2, 1, 2, 2, 2, 2, 2, 2, 2, 2, 2, 2 ] \\
		   & 4 & ? & [ 0, 2, 2, 1, 2, 2, 2, 1, 2, 2, 2, 2, 1, 2, 2, 2, 2, 2, 1, 2, 2, 2, 2, 2, 2, 2, 2, 2, 2 ] \\
		\hline
		 $(136, 30, 8, 6)$  & 12 & ? & [ 0, 3, 3, 8, 2, 3, 3, 3, 3, 2 ] \\
		\hline
		 $(136, 60, 24, 28)$  & 12 & ? & [ 0, 7, 10, 4, 4, 7, 7, 10, 7, 4 ] \\
		\hline
		 $(147, 66, 25, 33)$  & 3 & \cite{brady} & [ 0, 3, 1, 3, 3, 3, 1, 1, 3, 3, 1, 1, 3, 3, 3, 1, 1, 1, 3, 3, 1, 1, 1, 3, 3, 1, 1, 1, 3, 3, 1, 1, 3, 1, 1 ] \\
		   & 4 & \cite{brady} & [ 0, 21, 3, 3, 21, 1, 3, 1, 1, 3, 1, 1, 1, 1, 1, 1, 1, 1, 1 ] \\
		   & 5 & \cite{brady} & [ 0, 21, 3, 3, 21, 1, 3, 1, 3, 1, 1, 1, 1, 1, 1, 1, 1, 1, 1 ] \\
		\hline
		 $(148, 70, 36, 30)$  & 3 & ? & [ 0, 15, 22, 2, 15, 2, 2, 2, 2, 2, 2, 2, 2 ] \\
		\hline
		 $(155, 42, 17, 9)$  & 1 & \ref{Clapham-PDS-family} & [ 0, 9, 1, 9, 9, 1, 9, 1, 1, 1, 1 ] \\
		\hline
		 $(160, 54, 18, 18)$  & 199 & ? & [ 0, 6, 6, 0, 0, 0, 6, 6, 6, 6, 6, 6, 6 ] \\
		   &  &  & [ 0, 6, 0, 6, 0, 0, 6, 6, 6, 6, 6, 6, 6 ] \\
		\hline
		 $(171, 34, 17, 4)$  & 3 & \cite{pol-dav-smi-swa} & [ 0, 4, 4, 1, 4, 4, 4, 1, 4, 4, 4 ] \\
		\hline
		 $(171, 60, 15, 24)$  & 3 & ? & [ 0, 8, 3, 3, 8, 8, 3, 3, 8, 8, 8 ] \\
		\hline
		 $(183, 52, 11, 16)$  & 1 & ? & [ 0, 16, 1, 16, 1, 1, 1, 1, 1, 1, 1, 1, 1, 1, 1, 1, 1, 1, 1, 1, 1, 1, 1 ] \\
		\hline
		 $(183, 70, 29, 25)$  & 1 & \ref{prop:buratti} & [ 0, 25, 1, 25, 1, 1, 1, 1, 1, 1, 1, 1, 1, 1, 1, 1, 1, 1, 1, 1, 1, 1, 1 ] \\
		\hline
		 $(205, 68, 15, 26)$  & 1 & ? & [ 0, 13, 2, 13, 2, 13, 2, 13, 2, 2, 2, 2, 2 ] \\
		\hline
		 $(205, 96, 50, 40)$  & 1 & ? & [ 0, 20, 2, 20, 2, 20, 2, 20, 2, 2, 2, 2, 2 ] \\
		\hline
		 $(216, 90, 39, 36)$  & 86 & ? & [ 0, 12, 4, 4, 4, 2, 12, 12, 4, 8, 8, 12, 8 ] \\
		   & 87 & ? & [ 0, 8, 8, 4, 2, 2, 8, 16, 16, 8, 2, 8, 8 ] \\
		   & 88 & ? & [ 0, 8, 8, 4, 4, 1, 8, 8, 8, 4, 4, 1, 8, 8, 8, 8 ] \\
		   & 153 & ? & [ 0, 7, 24, 4, 8, 7, 16, 4, 16, 4 ] \\
		   &  &  & [ 0, 2, 24, 4, 8, 2, 16, 9, 16, 9 ] \\
		\hline
		 $(244, 108, 42, 52)$  & 3 & ? & [ 0, 26, 26, 2, 26, 2, 2, 2, 2, 2, 2, 2, 2, 2, 2, 2, 2, 2, 2 ] \\
		\hline
		 $(250, 81, 24, 27)$  & 5 & ? & [ 0, 9, 3, 6, 0, 9, 3, 3, 6, 0, 9, 3, 0, 0, 9, 3, 0, 3, 9, 3, 3, 0 ] \\
		   &  &  & [ 0, 9, 3, 0, 0, 9, 3, 6, 0, 0, 9, 3, 0, 0, 9, 3, 0, 6, 9, 6, 6, 0 ] \\
		   &  &  & [ 0, 9, 3, 0, 0, 9, 3, 3, 0, 0, 9, 3, 3, 3, 9, 3, 3, 3, 9, 3, 3, 3 ] \\
		   &  &  & [ 0, 9, 0, 6, 0, 9, 0, 6, 6, 0, 9, 0, 0, 0, 9, 0, 0, 6, 9, 6, 6, 0 ] \\
		   &  &  & [ 0, 9, 0, 6, 0, 9, 0, 3, 6, 0, 9, 0, 3, 3, 9, 0, 3, 3, 9, 3, 3, 3 ] \\
		   &  &  & [ 0, 9, 0, 0, 0, 9, 0, 6, 0, 0, 9, 0, 3, 3, 9, 0, 3, 6, 9, 6, 6, 3 ] \\
		   & 6 & ? & [ 0, 9, 3, 6, 0, 9, 3, 3, 6, 0, 9, 3, 0, 0, 9, 3, 0, 3, 9, 3, 3, 0 ] \\
		   &  &  & [ 0, 9, 3, 0, 0, 9, 3, 6, 0, 0, 9, 3, 0, 0, 9, 3, 0, 6, 9, 6, 6, 0 ] \\
		   &  &  & [ 0, 9, 3, 0, 0, 9, 3, 3, 0, 0, 9, 3, 3, 3, 9, 3, 3, 3, 9, 3, 3, 3 ] \\
		   &  &  & [ 0, 9, 0, 6, 0, 9, 0, 6, 6, 0, 9, 0, 0, 0, 9, 0, 0, 6, 9, 6, 6, 0 ] \\
		   &  &  & [ 0, 9, 0, 6, 0, 9, 0, 3, 6, 0, 9, 0, 3, 3, 9, 0, 3, 3, 9, 3, 3, 3 ] \\
		   &  &  & [ 0, 9, 0, 0, 0, 9, 0, 6, 0, 0, 9, 0, 3, 3, 9, 0, 3, 6, 9, 6, 6, 3 ] \\
		   & 8 & ? & [ 0, 9, 6, 6, 0, 9, 6, 6, 6, 0, 9, 0, 0, 0, 9, 9, 0, 6, 0, 0, 0, 0 ] \\
		\hline
		 $(253, 42, 21, 4)$  & 1 & \cite{pol-dav-smi-swa} & [ 0, 4, 1, 4, 4, 4, 4, 1, 4, 4, 4, 4, 4 ] \\
		\hline
		 $(253, 112, 36, 60)$  & 1 & \cite{brouwer-maldeghem} & [ 0, 10, 6, 10, 10, 10, 10, 6, 10, 10, 10, 10, 10 ] \\
		\hline
		 $(273, 80, 19, 25)$  & 3 & ? & [ 0, 25, 1, 1, 25, 1, 1, 1, 1, 1, 1, 1, 1, 1, 1, 1, 1, 1, 1, 1, 1, 1, 1, 1, 1, 1, 1, 1, 1, 1, 1, 1, 1 ] \\
		   & 4 & ? & [ 0, 25, 1, 1, 25, 1, 1, 1, 1, 1, 1, 1, 1, 1, 1, 1, 1, 1, 1, 1, 1, 1, 1, 1, 1, 1, 1, 1, 1, 1, 1, 1, 1 ] \\
		\hline
		 $(273, 102, 41, 36)$  & 3 & ? & [ 0, 36, 1, 1, 36, 1, 1, 1, 1, 1, 1, 1, 1, 1, 1, 1, 1, 1, 1, 1, 1, 1, 1, 1, 1, 1, 1, 1, 1, 1, 1, 1, 1 ] \\
		   & 4 & ? & [ 0, 36, 1, 1, 36, 1, 1, 1, 1, 1, 1, 1, 1, 1, 1, 1, 1, 1, 1, 1, 1, 1, 1, 1, 1, 1, 1, 1, 1, 1, 1, 1, 1 ] \\
		\hline
		 $(301, 60, 23, 9)$  & 1 & \ref{Clapham-PDS-family} & [ 0, 9, 1, 9, 1, 9, 1, 9, 9, 9, 1, 1, 1 ] \\
		\hline
		 $(301, 108, 27, 45)$  & 1 & ? & [ 0, 15, 3, 15, 3, 15, 3, 15, 15, 15, 3, 3, 3 ] \\
		\hline
		 $(301, 150, 65, 84)$  & 1 & ? & [ 0, 21, 4, 21, 4, 21, 4, 21, 21, 21, 4, 4, 4 ] \\
		\hline
		 $(305, 76, 27, 16)$  & 1 & \ref{Fuji-PDS-family} & [ 0, 16, 1, 16, 1, 16, 1, 16, 1, 1, 1, 1, 1, 1, 1, 1, 1 ] \\
		\hline
		 $(320, 132, 46, 60)$  & 1635 & ? & [ 0, 16, 8, 30, 2, 4, 16, 16, 16, 8, 16 ] \\
		\hline
		 $(343, 54, 5, 9)$  & 3 & \cite{Kantor_1986} & [ 0, 1, 1, 1, 1, 1, 1, 1, 1, 1, 1, 1, 1, 1, 1, 1, 1, 1, 1, 1, 1, 1, 1, 1, 1, 1, 1, 1, 1, 1, 1, 1, 1, 1, 1, 1, 1, 1, 1, 1, 1, 1, 1, 1, 1, 1, 1, 1, 1, 1, 1, 1, 1, 1, 1 ] \\
		   & 4 & \cite{Swartz_2015} & [ 0, 1, 1, 1, 1, 1, 1, 1, 1, 1, 1, 1, 1, 1, 1, 1, 1, 1, 1, 1, 1, 1, 1, 1, 1, 1, 1, 1, 1, 1, 1, 1, 1, 1, 1, 1, 1, 1, 1, 1, 1, 1, 1, 1, 1, 1, 1, 1, 1, 1, 1, 1, 1, 1, 1 ] \\
		\hline
		 $(343, 150, 53, 75)$  & 3 & \cite{pol-dav-smi-swa} & [ 0, 3, 3, 1, 3, 3, 3, 1, 3, 3, 3, 3, 1, 3, 3, 3, 3, 3, 1, 3, 3, 3, 3, 3, 3, 1, 3, 3, 3, 3, 3, 3, 3, 1, 3, 3, 3, 3, 3, 3, 3, 3, 3, 3, 3, 3, 3, 3, 3, 3, 3, 3, 3, 3, 3 ] \\
		   & 4 & ? & [ 0, 3, 3, 1, 3, 3, 3, 1, 3, 3, 3, 3, 1, 3, 3, 3, 3, 3, 1, 3, 3, 3, 3, 3, 3, 1, 3, 3, 3, 3, 3, 3, 3, 1, 3, 3, 3, 3, 3, 3, 3, 3, 3, 3, 3, 3, 3, 3, 3, 3, 3, 3, 3, 3, 3 ] \\
		\hline
		 $(351, 50, 13, 6)$  & 12 & ? & [ 0, 3, 7, 7, 3, 3, 3, 3, 3, 3, 3, 3, 3, 3, 3 ] \\
		\hline
		 $(351, 50, 25, 4)$  & 12 & \cite{pol-dav-smi-swa} & [ 0, 4, 1, 1, 4, 4, 4, 4, 4, 4, 4, 4, 4, 4, 4 ] \\
		\hline
		 $(351, 110, 37, 33)$  & 3 & ? & [ 0, 9, 1, 1, 1, 9, 3, 3, 1, 1, 1, 1, 1, 1, 3, 3, 3, 3, 3, 1, 1, 1, 1, 1, 1, 1, 3, 3, 3, 3, 3, 1, 1, 1, 1, 1, 1, 1, 3, 3, 3, 1, 1, 1, 1, 1, 3, 1, 1, 1, 1, 1, 1, 1, 1, 1, 1, 1, 1, 1, 1, 1, 1 ] \\
		   &  &  & [ 0, 7, 1, 1, 1, 7, 5, 1, 1, 1, 1, 1, 1, 1, 3, 1, 3, 5, 1, 1, 1, 1, 1, 1, 1, 1, 5, 3, 1, 3, 5, 1, 1, 1, 1, 1, 1, 1, 5, 3, 3, 1, 1, 1, 1, 1, 5, 1, 1, 1, 1, 1, 1, 1, 1, 1, 1, 1, 1, 1, 1, 1, 1 ] \\
		   & 4 & ? & [ 0, 15, 3, 1, 1, 15, 9, 3, 1, 1, 1, 1, 9, 9, 1, 1, 1, 1, 1, 9, 1, 1, 1, 1, 1, 1, 1, 1, 1, 1, 1, 1, 1, 1, 1, 1, 1, 1, 1, 1, 1, 1, 1, 1, 1, 1, 1 ] \\
		   &  &  & [ 0, 9, 3, 1, 1, 9, 15, 3, 1, 1, 1, 1, 9, 9, 1, 1, 1, 1, 1, 15, 1, 1, 1, 1, 1, 1, 1, 1, 1, 1, 1, 1, 1, 1, 1, 1, 1, 1, 1, 1, 1, 1, 1, 1, 1, 1, 1 ] \\
		   & 5 & ? & [ 0, 15, 3, 1, 1, 15, 9, 3, 1, 1, 1, 1, 9, 9, 1, 1, 1, 1, 1, 9, 1, 1, 1, 1, 1, 1, 1, 1, 1, 1, 1, 1, 1, 1, 1, 1, 1, 1, 1, 1, 1, 1, 1, 1, 1, 1, 1 ] \\
		   &  &  & [ 0, 9, 3, 1, 1, 9, 15, 3, 1, 1, 1, 1, 9, 9, 1, 1, 1, 1, 1, 15, 1, 1, 1, 1, 1, 1, 1, 1, 1, 1, 1, 1, 1, 1, 1, 1, 1, 1, 1, 1, 1, 1, 1, 1, 1, 1, 1 ] \\
		   & 7 & ? & [ 0, 15, 3, 1, 1, 15, 9, 3, 1, 1, 1, 1, 9, 9, 1, 1, 1, 1, 1, 9, 1, 1, 1, 1, 1, 1, 1, 1, 1, 1, 1, 1, 1, 1, 1, 1, 1, 1, 1, 1, 1, 1, 1, 1, 1, 1, 1 ] \\
		   & 8 & ? & [ 0, 15, 3, 1, 1, 15, 9, 3, 1, 1, 1, 1, 9, 9, 1, 1, 1, 1, 1, 9, 1, 1, 1, 1, 1, 1, 1, 1, 1, 1, 1, 1, 1, 1, 1, 1, 1, 1, 1, 1, 1, 1, 1, 1, 1, 1, 1 ] \\
		   & 12 & ? & [ 0, 9, 1, 1, 9, 9, 9, 9, 9, 9, 9, 9, 9, 9, 9 ] \\
		   & 13 & ? & [ 0, 9, 1, 1, 1, 9, 3, 3, 1, 1, 1, 1, 1, 1, 3, 3, 3, 3, 3, 1, 1, 1, 1, 1, 1, 1, 3, 3, 3, 3, 3, 1, 1, 1, 1, 1, 1, 1, 3, 3, 3, 1, 1, 1, 1, 1, 3, 1, 1, 1, 1, 1, 1, 1, 1, 1, 1, 1, 1, 1, 1, 1, 1 ] \\
		   &  &  & [ 0, 7, 1, 1, 1, 7, 5, 5, 1, 1, 1, 1, 1, 1, 3, 3, 3, 3, 3, 1, 1, 1, 1, 1, 1, 1, 5, 5, 5, 1, 1, 1, 1, 1, 1, 1, 1, 1, 1, 1, 5, 1, 1, 1, 1, 1, 3, 1, 1, 1, 1, 1, 1, 1, 1, 1, 1, 1, 1, 1, 1, 1, 1 ] \\
		\hline
		 $(351, 112, 43, 32)$  & 12 & ? & [ 0, 9, 2, 2, 9, 9, 9, 9, 9, 9, 9, 9, 9, 9, 9 ] \\
		\hline
		 $(351, 126, 45, 45)$  & 12 & ? & [ 0, 9, 9, 9, 9, 9, 9, 9, 9, 9, 9, 9, 9, 9, 9 ] \\
		\hline
		 $(351, 140, 73, 44)$  & 12 & ? & [ 0, 11, 4, 4, 11, 11, 11, 11, 11, 11, 11, 11, 11, 11, 11 ] \\
		\hline
		 $(363, 170, 73, 85)$  & 2 & ? & [ 0, 55, 3, 1, 55, 3, 1, 1, 3, 1, 1, 1, 3, 1, 3, 1, 3, 3, 1, 3, 1, 1, 1, 1, 3, 1, 1, 1, 1, 1, 1, 1, 1, 3, 1, 1, 1, 1, 1, 1, 1, 1, 1 ] \\
		\hline
		 $(381, 114, 29, 36)$  & 1 & ? & [ 0, 36, 1, 36, 1, 1, 1, 1, 1, 1, 1, 1, 1, 1, 1, 1, 1, 1, 1, 1, 1, 1, 1, 1, 1, 1, 1, 1, 1, 1, 1, 1, 1, 1, 1, 1, 1, 1, 1, 1, 1, 1, 1, 1, 1 ] \\
		\hline
		 $(381, 140, 55, 49)$  & 1 & ? & [ 0, 49, 1, 49, 1, 1, 1, 1, 1, 1, 1, 1, 1, 1, 1, 1, 1, 1, 1, 1, 1, 1, 1, 1, 1, 1, 1, 1, 1, 1, 1, 1, 1, 1, 1, 1, 1, 1, 1, 1, 1, 1, 1, 1, 1 ] \\
		\hline
		 $(400, 156, 74, 52)$  & 206 & ? & [ 0, 20, 20, 10, 10, 2, 20, 4, 20, 10, 4, 20, 16 ] \\
		   & 207 & ? & [ 0, 10, 4, 4, 10, 2, 20, 20, 10, 20, 16, 16, 2, 4, 2, 16 ] \\
		\hline
		 $(405, 96, 18, 24)$  & 15 & ? & [ 0, 18, 4, 4, 0, 2, 18, 4, 2, 0, 0, 4, 0, 2, 0, 18, 0, 0, 2, 0, 18 ] \\
		   &  &  & [ 0, 18, 4, 4, 0, 0, 18, 4, 0, 0, 2, 4, 0, 0, 2, 18, 2, 2, 0, 0, 18 ] \\
		   &  &  & [ 0, 18, 4, 4, 0, 0, 18, 4, 0, 0, 0, 4, 4, 0, 0, 18, 0, 0, 0, 4, 18 ] \\
		   &  &  & [ 0, 18, 4, 2, 4, 0, 18, 4, 0, 4, 0, 2, 0, 0, 0, 18, 2, 2, 0, 0, 18 ] \\
		   &  &  & [ 0, 18, 4, 2, 2, 2, 18, 4, 0, 2, 0, 2, 2, 2, 0, 18, 0, 0, 0, 2, 18 ] \\
		   &  &  & [ 0, 18, 4, 2, 0, 0, 18, 4, 2, 0, 2, 2, 2, 0, 2, 18, 0, 0, 2, 2, 18 ] \\
		   &  &  & [ 0, 18, 4, 2, 0, 0, 18, 4, 2, 0, 0, 2, 0, 0, 0, 18, 4, 4, 2, 0, 18 ] \\
		   &  &  & [ 0, 18, 4, 0, 2, 0, 18, 4, 0, 2, 0, 0, 2, 0, 0, 18, 4, 4, 0, 2, 18 ] \\
		   &  &  & [ 0, 18, 2, 2, 2, 0, 18, 2, 2, 2, 0, 2, 2, 0, 0, 18, 2, 2, 2, 2, 18 ] \\
		\hline
		 $(441, 88, 35, 13)$  & 9 & ? & [ 0, 1, 1, 1, 1, 1, 7, 3, 1, 3, 7, 7, 3, 7, 3, 1, 1, 7, 3, 3, 7, 7, 3, 3, 7 ] \\
		\hline
		 $(441, 152, 43, 57)$  & 9 & ? & [ 0, 5, 5, 1, 1, 5, 19, 7, 5, 7, 3, 11, 7, 11, 7, 1, 1, 19, 7, 7, 3, 3, 7, 7, 3 ] \\
		   &  &  & [ 0, 5, 5, 1, 1, 5, 11, 7, 5, 7, 3, 19, 7, 19, 7, 1, 1, 11, 7, 7, 3, 3, 7, 7, 3 ] \\
		\hline
		 $(441, 176, 85, 60)$  & 9 & ? & [ 0, 6, 6, 1, 1, 6, 20, 7, 6, 7, 3, 20, 7, 20, 7, 1, 1, 20, 7, 7, 3, 3, 7, 7, 3 ] \\
		\hline
		 $(441, 184, 87, 69)$  & 9 & ? & [ 0, 3, 3, 3, 3, 3, 21, 9, 3, 9, 1, 21, 9, 21, 9, 3, 3, 21, 9, 9, 1, 1, 9, 9, 1 ] \\
		\hline
		 $(448, 150, 50, 50)$  & 178 & ? & [ 0, 20, 5, 0, 20, 5, 5, 0, 5, 0, 5, 5, 20, 20, 20, 20 ] \\
		   & 179 & ? & [ 0, 20, 15, 0, 20, 15, 20, 20, 20, 20 ] \\
		\hline
		 $(448, 162, 66, 54)$  & 178 & ? & [ 0, 24, 6, 6, 24, 0, 0, 0, 0, 0, 6, 0, 24, 24, 24, 24 ] \\
		   &  &  & [ 0, 24, 3, 6, 24, 3, 0, 0, 0, 0, 3, 3, 24, 24, 24, 24 ] \\
		   & 179 & ? & [ 0, 24, 6, 6, 24, 6, 24, 24, 24, 24 ] \\
		   & 1393 & ? & [ 0, 24, 6, 6, 24, 6, 0, 0, 0, 0, 24, 0, 0, 24, 24, 24 ] \\
		   & 1394 & ? & [ 0, 24, 6, 6, 24, 6, 0, 0, 0, 0, 0, 0, 24, 24, 24, 24 ] \\
		   &  &  & [ 0, 24, 6, 0, 24, 6, 6, 0, 0, 0, 0, 0, 24, 24, 24, 24 ] \\
		   &  &  & [ 0, 24, 0, 0, 24, 6, 6, 6, 0, 0, 0, 0, 24, 24, 24, 24 ] \\
		   &  &  & [ 0, 24, 0, 0, 24, 6, 6, 0, 6, 0, 0, 0, 24, 24, 24, 24 ] \\
		   &  &  & [ 0, 24, 0, 0, 24, 6, 6, 0, 0, 0, 6, 0, 24, 24, 24, 24 ] \\
		\hline
		 $(465, 58, 29, 4)$  & 1 & \cite{pol-dav-smi-swa} & [ 0, 4, 4, 1, 4, 4, 4, 4, 4, 4, 1, 4, 4, 4, 4, 4, 4 ] \\
		\hline
		 $(465, 192, 72, 84)$  & 1 & ? & [ 0, 6, 14, 6, 6, 14, 14, 14, 14, 14, 6, 14, 14, 14, 14, 14, 14 ] \\
		\hline
		 $(497, 186, 55, 78)$  & 1 & ? & [ 0, 26, 3, 26, 3, 26, 26, 3, 26, 3, 26, 3, 3, 3, 3, 3, 3 ] \\
		\hline
		 $(497, 240, 127, 105)$  & 1 & ? & [ 0, 35, 3, 35, 3, 35, 35, 3, 35, 3, 35, 3, 3, 3, 3, 3, 3 ] \\
		\hline
		 $(505, 84, 3, 16)$  & 1 & ? & [ 0, 1, 1, 1, 1, 1, 1, 1, 1, 1, 1, 1, 1, 1, 1, 1, 1, 1, 1, 1, 1, 16, 16, 16, 16 ] \\
		\hline
		 $(505, 120, 39, 25)$  & 1 & ? & [ 0, 1, 1, 1, 1, 1, 1, 1, 1, 1, 1, 1, 1, 1, 1, 1, 1, 1, 1, 1, 1, 25, 25, 25, 25 ] \\
		\hline
		 $(505, 180, 53, 70)$  & 1 & ? & [ 0, 2, 2, 2, 2, 2, 2, 2, 2, 2, 2, 2, 2, 2, 2, 2, 2, 2, 2, 2, 2, 35, 35, 35, 35 ] \\
		\hline
		 $(505, 224, 108, 92)$  & 1 & ? & [ 0, 2, 2, 2, 2, 2, 2, 2, 2, 2, 2, 2, 2, 2, 2, 2, 2, 2, 2, 2, 2, 46, 46, 46, 46 ] \\
		\hline
		 $(507, 154, 41, 49)$  & 1 & ? & [ 0, 1, 1, 1, 1, 1, 1, 1, 1, 1, 1, 1, 1, 1, 1, 1, 1, 1, 1, 1, 1, 1, 1, 1, 1, 1, 1, 1, 1, 1, 1, 1, 1, 1, 1, 1, 1, 1, 1, 1, 1, 1, 1, 1, 1, 1, 1, 1, 1, 1, 1, 1, 1, 1, 1, 1, 1, 49, 49 ] \\
		   & 4 & ? & [ 0, 1, 1, 1, 1, 1, 1, 1, 1, 1, 1, 1, 1, 1, 1, 1, 1, 1, 1, 1, 1, 1, 1, 1, 1, 1, 1, 1, 1, 1, 1, 1, 1, 1, 1, 1, 1, 1, 1, 1, 1, 1, 1, 1, 1, 1, 1, 1, 1, 1, 1, 1, 1, 1, 1, 1, 1, 49, 49 ] \\
		   & 5 & ? & [ 0, 1, 1, 1, 1, 1, 1, 1, 1, 1, 1, 1, 1, 1, 1, 1, 1, 1, 1, 1, 1, 1, 1, 1, 1, 1, 1, 1, 1, 1, 1, 1, 1, 1, 1, 1, 1, 1, 1, 1, 1, 1, 1, 1, 1, 1, 1, 1, 1, 1, 1, 1, 1, 1, 1, 1, 1, 49, 49 ] \\
		\hline
		 $(507, 176, 70, 56)$  & 3 & ? & [ 0, 0, 0, 0, 0, 0, 2, 2, 2, 0, 0, 0, 2, 2, 2, 0, 2, 0, 2, 2, 0, 2, 2, 0, 2, 0, 2, 2, 2, 0, 0, 0, 2, 2, 2, 0, 2, 0, 2, 2, 0, 2, 2, 0, 2, 0, 2, 2, 2, 0, 0, 0, 2, 2, 2, 0, 2, 0, 2, 2, 0, 2, 2, 0, 2, 4, 4, 4, 4, 4, 4, 4, 4, 4, 4, 4, 4, 4, 4, 4, 4, 4, 4, 4, 4, 4, 4, 4, 4, 4, 4 ] \\
		   & 4 & ? & [ 0, 2, 2, 2, 2, 2, 2, 2, 2, 2, 2, 2, 2, 0, 0, 0, 0, 0, 2, 0, 2, 2, 2, 2, 2, 2, 2, 0, 0, 0, 0, 2, 0, 0, 2, 2, 2, 2, 2, 2, 2, 0, 0, 0, 2, 0, 0, 0, 2, 2, 2, 2, 2, 2, 2, 0, 0, 52, 52 ] \\
		   &  &  & [ 0, 2, 2, 2, 2, 2, 2, 2, 2, 2, 2, 2, 0, 2, 0, 0, 0, 0, 2, 0, 2, 2, 2, 2, 2, 2, 0, 2, 0, 0, 0, 2, 0, 0, 2, 2, 2, 2, 2, 2, 0, 2, 0, 0, 2, 0, 0, 0, 2, 2, 2, 2, 2, 2, 0, 2, 0, 52, 52 ] \\
		   &  &  & [ 0, 2, 2, 2, 2, 2, 2, 2, 2, 2, 2, 0, 2, 2, 0, 0, 0, 0, 2, 0, 2, 2, 2, 2, 2, 0, 2, 2, 0, 0, 0, 2, 0, 0, 2, 2, 2, 2, 2, 0, 2, 2, 0, 0, 2, 0, 0, 0, 2, 2, 2, 2, 2, 0, 2, 2, 0, 52, 52 ] \\
		   & 5 & ? & [ 0, 0, 0, 0, 0, 0, 2, 2, 2, 0, 2, 2, 2, 0, 2, 2, 2, 0, 0, 0, 2, 2, 2, 0, 2, 2, 2, 0, 2, 2, 2, 0, 2, 0, 2, 2, 2, 0, 2, 2, 2, 0, 2, 2, 0, 2, 2, 0, 2, 2, 2, 0, 2, 2, 2, 0, 2, 52, 52 ] \\
		\hline
		 $(507, 184, 71, 64)$  & 1 & ? & [ 0, 1, 1, 1, 1, 1, 1, 1, 1, 1, 1, 1, 1, 1, 1, 1, 1, 1, 1, 1, 1, 1, 1, 1, 1, 1, 1, 1, 1, 1, 1, 1, 1, 1, 1, 1, 1, 1, 1, 1, 1, 1, 1, 1, 1, 1, 1, 1, 1, 1, 1, 1, 1, 1, 1, 1, 1, 64, 64 ] \\
		   & 4 & ? & [ 0, 1, 1, 1, 1, 1, 1, 1, 1, 1, 1, 1, 1, 1, 1, 1, 1, 1, 1, 1, 1, 1, 1, 1, 1, 1, 1, 1, 1, 1, 1, 1, 1, 1, 1, 1, 1, 1, 1, 1, 1, 1, 1, 1, 1, 1, 1, 1, 1, 1, 1, 1, 1, 1, 1, 1, 1, 64, 64 ] \\
		   & 5 & ? & [ 0, 1, 1, 1, 1, 1, 1, 1, 1, 1, 1, 1, 1, 1, 1, 1, 1, 1, 1, 1, 1, 1, 1, 1, 1, 1, 1, 1, 1, 1, 1, 1, 1, 1, 1, 1, 1, 1, 1, 1, 1, 1, 1, 1, 1, 1, 1, 1, 1, 1, 1, 1, 1, 1, 1, 1, 1, 64, 64 ] \\
		\hline
		 $(651, 200, 55, 64)$  & 3 & ? & [ 0, 1, 1, 1, 1, 1, 1, 1, 1, 1, 1, 1, 1, 1, 1, 1, 1, 1, 1, 1, 1, 1, 1, 1, 1, 1, 1, 1, 1, 1, 1, 1, 1, 1, 1, 1, 1, 1, 1, 1, 1, 1, 1, 1, 1, 1, 1, 1, 1, 1, 1, 1, 1, 1, 1, 1, 1, 1, 1, 1, 1, 1, 1, 1, 1, 1, 1, 1, 1, 1, 1, 1, 1, 64, 64 ] \\
		   & 4 & ? & [ 0, 1, 1, 1, 1, 1, 1, 1, 1, 1, 1, 1, 1, 1, 1, 1, 1, 1, 1, 1, 1, 1, 1, 1, 1, 1, 1, 1, 1, 1, 1, 1, 1, 1, 1, 1, 1, 1, 1, 1, 1, 1, 1, 1, 1, 1, 1, 1, 1, 1, 1, 1, 1, 1, 1, 1, 1, 1, 1, 1, 1, 1, 1, 1, 1, 1, 1, 1, 1, 1, 1, 1, 1, 64, 64 ] \\
		\hline
		 $(651, 234, 89, 81)$  & 3 & ? & [ 0, 1, 1, 1, 1, 1, 1, 1, 1, 1, 1, 1, 1, 1, 1, 1, 1, 1, 1, 1, 1, 1, 1, 1, 1, 1, 1, 1, 1, 1, 1, 1, 1, 1, 1, 1, 1, 1, 1, 1, 1, 1, 1, 1, 1, 1, 1, 1, 1, 1, 1, 1, 1, 1, 1, 1, 1, 1, 1, 1, 1, 1, 1, 1, 1, 1, 1, 1, 1, 1, 1, 1, 1, 81, 81 ] \\
		   & 4 & ? & [ 0, 1, 1, 1, 1, 1, 1, 1, 1, 1, 1, 1, 1, 1, 1, 1, 1, 1, 1, 1, 1, 1, 1, 1, 1, 1, 1, 1, 1, 1, 1, 1, 1, 1, 1, 1, 1, 1, 1, 1, 1, 1, 1, 1, 1, 1, 1, 1, 1, 1, 1, 1, 1, 1, 1, 1, 1, 1, 1, 1, 1, 1, 1, 1, 1, 1, 1, 1, 1, 1, 1, 1, 1, 81, 81 ] \\
		\hline
		 $(657, 256, 80, 112)$  & 3 & ? & [ 0, 4, 4, 4, 4, 4, 4, 4, 4, 28, 28, 28, 28, 28, 28, 28, 28 ] \\
		\hline
		 $(657, 328, 147, 180)$  & 3 & ? & [ 0, 5, 5, 5, 5, 5, 5, 5, 5, 36, 36, 36, 36, 36, 36, 36, 36 ] \\
		\hline
		 $(689, 256, 120, 80)$  & 1 & ? & [ 0, 4, 4, 4, 4, 20, 20, 20, 20, 20, 20, 20, 20, 20, 20, 20, 20 ] \\
		\hline
		 $(737, 96, 35, 9)$  & 1 & \ref{Clapham-PDS-family} & [ 0, 1, 1, 1, 1, 1, 1, 9, 9, 9, 9, 9, 9, 9, 9, 9, 9 ] \\
		\hline
		 $(755, 130, 9, 25)$  & 1 & ? & [ 0, 1, 1, 1, 1, 1, 1, 1, 1, 1, 1, 1, 1, 1, 1, 1, 1, 1, 1, 1, 1, 1, 1, 1, 1, 1, 1, 1, 1, 1, 1, 25, 25, 25, 25 ] \\
		\hline
		 $(755, 174, 53, 36)$  & 1 & ? & [ 0, 1, 1, 1, 1, 1, 1, 1, 1, 1, 1, 1, 1, 1, 1, 1, 1, 1, 1, 1, 1, 1, 1, 1, 1, 1, 1, 1, 1, 1, 1, 36, 36, 36, 36 ] \\
		\hline
		 $(813, 252, 71, 81)$  & 1 & ? & [ 0, 1, 1, 1, 1, 1, 1, 1, 1, 1, 1, 1, 1, 1, 1, 1, 1, 1, 1, 1, 1, 1, 1, 1, 1, 1, 1, 1, 1, 1, 1, 1, 1, 1, 1, 1, 1, 1, 1, 1, 1, 1, 1, 1, 1, 1, 1, 1, 1, 1, 1, 1, 1, 1, 1, 1, 1, 1, 1, 1, 1, 1, 1, 1, 1, 1, 1, 1, 1, 1, 1, 1, 1, 1, 1, 1, 1, 1, 1, 1, 1, 1, 1, 1, 1, 1, 1, 1, 1, 1, 1, 81, 81 ] \\
		\hline
		 $(813, 290, 109, 100)$  & 1 & ? & [ 0, 1, 1, 1, 1, 1, 1, 1, 1, 1, 1, 1, 1, 1, 1, 1, 1, 1, 1, 1, 1, 1, 1, 1, 1, 1, 1, 1, 1, 1, 1, 1, 1, 1, 1, 1, 1, 1, 1, 1, 1, 1, 1, 1, 1, 1, 1, 1, 1, 1, 1, 1, 1, 1, 1, 1, 1, 1, 1, 1, 1, 1, 1, 1, 1, 1, 1, 1, 1, 1, 1, 1, 1, 1, 1, 1, 1, 1, 1, 1, 1, 1, 1, 1, 1, 1, 1, 1, 1, 1, 1, 100, 100 ] \\
		\hline
		 $(889, 222, 35, 62)$  & 1 & ? & [ 0, 2, 2, 2, 2, 2, 2, 2, 2, 2, 2, 2, 2, 2, 2, 2, 2, 2, 2, 31, 31, 31, 31, 31, 31 ] \\
		\hline
		 $(889, 288, 112, 84)$  & 1 & ? & [ 0, 2, 2, 2, 2, 2, 2, 2, 2, 2, 2, 2, 2, 2, 2, 2, 2, 2, 2, 42, 42, 42, 42, 42, 42 ] \\
		\hline
		 $(903, 82, 41, 4)$  & 1 & \cite{pol-dav-smi-swa} & [ 0, 1, 1, 4, 4, 4, 4, 4, 4, 4, 4, 4, 4, 4, 4, 4, 4, 4, 4, 4, 4, 4, 4 ] \\
		\hline
		 $(981, 140, 43, 16)$  & 3 & \ref{Fuji-PDS-family} & [ 0, 1, 1, 1, 1, 1, 1, 1, 1, 1, 1, 1, 1, 16, 16, 16, 16, 16, 16, 16, 16 ] \\
		\hline
		 $(981, 392, 133, 172)$  & 3 & ? & [ 0, 4, 4, 4, 4, 4, 4, 4, 4, 4, 4, 4, 4, 43, 43, 43, 43, 43, 43, 43, 43 ] \\
		\hline
		 $(981, 480, 254, 216)$  & 3 & ? & [ 0, 4, 4, 4, 4, 4, 4, 4, 4, 4, 4, 4, 4, 54, 54, 54, 54, 54, 54, 54, 54 ] \\
		\hline
		 $(993, 310, 89, 100)$  & 1 & ? & [ 0, 1, 1, 1, 1, 1, 1, 1, 1, 1, 1, 1, 1, 1, 1, 1, 1, 1, 1, 1, 1, 1, 1, 1, 1, 1, 1, 1, 1, 1, 1, 1, 1, 1, 1, 1, 1, 1, 1, 1, 1, 1, 1, 1, 1, 1, 1, 1, 1, 1, 1, 1, 1, 1, 1, 1, 1, 1, 1, 1, 1, 1, 1, 1, 1, 1, 1, 1, 1, 1, 1, 1, 1, 1, 1, 1, 1, 1, 1, 1, 1, 1, 1, 1, 1, 1, 1, 1, 1, 1, 1, 1, 1, 1, 1, 1, 1, 1, 1, 1, 1, 1, 1, 1, 1, 1, 1, 1, 1, 1, 1, 100, 100 ] \\
		\hline
		 $(993, 352, 131, 121)$  & 1 & ? & [ 0, 1, 1, 1, 1, 1, 1, 1, 1, 1, 1, 1, 1, 1, 1, 1, 1, 1, 1, 1, 1, 1, 1, 1, 1, 1, 1, 1, 1, 1, 1, 1, 1, 1, 1, 1, 1, 1, 1, 1, 1, 1, 1, 1, 1, 1, 1, 1, 1, 1, 1, 1, 1, 1, 1, 1, 1, 1, 1, 1, 1, 1, 1, 1, 1, 1, 1, 1, 1, 1, 1, 1, 1, 1, 1, 1, 1, 1, 1, 1, 1, 1, 1, 1, 1, 1, 1, 1, 1, 1, 1, 1, 1, 1, 1, 1, 1, 1, 1, 1, 1, 1, 1, 1, 1, 1, 1, 1, 1, 1, 1, 121, 121 ] \\
		\hline
		 $(1027, 114, 41, 9)$  & 1 & \ref{Clapham-PDS-family} & [ 0, 1, 1, 1, 1, 1, 1, 9, 9, 9, 9, 9, 9, 9, 9, 9, 9, 9, 9 ] \\
		\hline
		 $(1027, 342, 81, 130)$  & 1 & ? & [ 0, 5, 5, 5, 5, 5, 5, 26, 26, 26, 26, 26, 26, 26, 26, 26, 26, 26, 26 ] \\
		\hline
		 $(1027, 450, 225, 175)$  & 1 & ? & [ 0, 5, 5, 5, 5, 5, 5, 35, 35, 35, 35, 35, 35, 35, 35, 35, 35, 35, 35 ] \\
		\hline
		 $(1055, 186, 17, 36)$  & 1 & ? & [ 0, 1, 1, 1, 1, 1, 1, 1, 1, 1, 1, 1, 1, 1, 1, 1, 1, 1, 1, 1, 1, 1, 1, 1, 1, 1, 1, 1, 1, 1, 1, 1, 1, 1, 1, 1, 1, 1, 1, 1, 1, 1, 1, 36, 36, 36, 36 ] \\
		\hline
		 $(1055, 238, 69, 49)$  & 1 & ? & [ 0, 1, 1, 1, 1, 1, 1, 1, 1, 1, 1, 1, 1, 1, 1, 1, 1, 1, 1, 1, 1, 1, 1, 1, 1, 1, 1, 1, 1, 1, 1, 1, 1, 1, 1, 1, 1, 1, 1, 1, 1, 1, 1, 49, 49, 49, 49 ] \\
		\hline
		 $(1081, 90, 45, 4)$  & 1 & \cite{pol-dav-smi-swa} & [ 0, 1, 1, 4, 4, 4, 4, 4, 4, 4, 4, 4, 4, 4, 4, 4, 4, 4, 4, 4, 4, 4, 4, 4, 4 ] \\
		\hline
		 $(1081, 486, 177, 252)$  & 1 & ? & [ 0, 12, 12, 21, 21, 21, 21, 21, 21, 21, 21, 21, 21, 21, 21, 21, 21, 21, 21, 21, 21, 21, 21, 21, 21 ] \\
		\hline
		 $(1191, 374, 109, 121)$  & 1 & ? & [ 0, 1, 1, 1, 1, 1, 1, 1, 1, 1, 1, 1, 1, 1, 1, 1, 1, 1, 1, 1, 1, 1, 1, 1, 1, 1, 1, 1, 1, 1, 1, 1, 1, 1, 1, 1, 1, 1, 1, 1, 1, 1, 1, 1, 1, 1, 1, 1, 1, 1, 1, 1, 1, 1, 1, 1, 1, 1, 1, 1, 1, 1, 1, 1, 1, 1, 1, 1, 1, 1, 1, 1, 1, 1, 1, 1, 1, 1, 1, 1, 1, 1, 1, 1, 1, 1, 1, 1, 1, 1, 1, 1, 1, 1, 1, 1, 1, 1, 1, 1, 1, 1, 1, 1, 1, 1, 1, 1, 1, 1, 1, 1, 1, 1, 1, 1, 1, 1, 1, 1, 1, 1, 1, 1, 1, 1, 1, 1, 1, 1, 1, 1, 1, 121, 121 ] \\
		\hline
		 $(1191, 420, 155, 144)$  & 1 & ? & [ 0, 1, 1, 1, 1, 1, 1, 1, 1, 1, 1, 1, 1, 1, 1, 1, 1, 1, 1, 1, 1, 1, 1, 1, 1, 1, 1, 1, 1, 1, 1, 1, 1, 1, 1, 1, 1, 1, 1, 1, 1, 1, 1, 1, 1, 1, 1, 1, 1, 1, 1, 1, 1, 1, 1, 1, 1, 1, 1, 1, 1, 1, 1, 1, 1, 1, 1, 1, 1, 1, 1, 1, 1, 1, 1, 1, 1, 1, 1, 1, 1, 1, 1, 1, 1, 1, 1, 1, 1, 1, 1, 1, 1, 1, 1, 1, 1, 1, 1, 1, 1, 1, 1, 1, 1, 1, 1, 1, 1, 1, 1, 1, 1, 1, 1, 1, 1, 1, 1, 1, 1, 1, 1, 1, 1, 1, 1, 1, 1, 1, 1, 1, 1, 144, 144 ] \\
		\hline
		 $(1205, 448, 150, 176)$  & 1 & ? & [ 0, 2, 2, 2, 2, 2, 2, 2, 2, 2, 2, 2, 2, 2, 2, 2, 2, 2, 2, 2, 2, 2, 2, 2, 2, 2, 2, 2, 2, 2, 2, 2, 2, 2, 2, 2, 2, 2, 2, 2, 2, 2, 2, 2, 2, 2, 2, 2, 2, 88, 88, 88, 88 ] \\
		\hline
		 $(1205, 516, 235, 210)$  & 1 & ? & [ 0, 2, 2, 2, 2, 2, 2, 2, 2, 2, 2, 2, 2, 2, 2, 2, 2, 2, 2, 2, 2, 2, 2, 2, 2, 2, 2, 2, 2, 2, 2, 2, 2, 2, 2, 2, 2, 2, 2, 2, 2, 2, 2, 2, 2, 2, 2, 2, 2, 105, 105, 105, 105 ] \\
		\hline
	\end{longtable}
\end{center}

\end{document}